\documentclass[11pt]{article}
\usepackage[margin=1.0in]{geometry}
\usepackage[utf8]{inputenc}
\usepackage{mathrsfs,amssymb,amsfonts,amsthm,amsmath,amsbsy}
\usepackage{mathtools}
\usepackage{mathabx}
\usepackage{graphicx}
\usepackage{wrapfig}
\usepackage{caption}
\usepackage{subfig}
\usepackage{dsfont}
\usepackage{float}
\usepackage{upgreek}
\usepackage{listings}
\usepackage{lipsum}
\usepackage[export]{adjustbox}
\usepackage[
backend=bibtex,
style=alphabetic,
citestyle=ieee-alphabetic,
maxnames=10
]{biblatex}
\usepackage{titlesec}
\usepackage{hyperref}
\hypersetup{
	colorlinks,
	citecolor=blue,
	linkcolor=blue,
	urlcolor=green}

\newtheorem{theorem}{Theorem}[section]
\newtheorem{proposition}[theorem]{Proposition}

\newtheorem{lemma}[theorem]{Lemma}

\newtheorem{definition1}[theorem]{Definition}

\newtheorem{remark1}[theorem]{Remark}

\DeclarePairedDelimiter\floor{\lfloor}{\rfloor}

\usepackage[ruled,vlined]{algorithm2e}
\usepackage[font=small]{caption}

\usepackage[ruled,vlined]{algorithm2e}
\usepackage[font=small]{caption}
\usepackage[usenames,dvipsnames]{xcolor}
\usepackage{tikz}
\usepackage{subfig}
\usepackage{dsfont,fancyhdr}
\usetikzlibrary{matrix}
\usepackage{stmaryrd}

\usepackage{color}

\title{Quenched scaling limit for biased random walks on random, heavy tailed conductances: low dimensions}

\author{Umberto De Ambroggio \thanks{Department of Mathematics, National University of Singapore, Email: \texttt{umberto@nus.edu.sg}} \and Carlo Scali \thanks{School of Computation, Information and Technology, Technische Universit{\"a}t M{\"u}nchen, Email: \texttt{carlo.scali@tum.de}}}
\begin{document}

\maketitle

\begin{abstract}
We consider a random walk amongst positive random conductances on $\mathbb{Z}^d, d \ge 2$, with directional bias. When the conductances have a stable distribution with parameter $\gamma \in (0, 1)$, the walk is sub-ballistic. In this regime Fribergh and Kious \cite{Kious_Frib} derived an annealed scaling limit for the appropriately rescaled walk towards the Fractional Kinetics process. We prove the quenched version of this result for all $d \ge 2$.\\

\\

\end{abstract}

\section{Introduction}

Random walks in random environment (RWRE) are a classical model in probability theory, firstly rigorously analyzed by Solomon \cite{Solomon}. This seminal paper highlighted that these models showcase a behavior that is strongly atypical with respect to the one of classical random walks. Let us briefly go over some examples in which this atypical behaviour has been rigorously described, with a particular focus on models which undergo trapping.

Several works have focused their attention towards models in an anisotropic setting. In these works, particular attentions have been devoted to study cases in which the random walk exhibits slowdown due to the effect of trapping. Indeed, slowdown and sub-ballistic behavior has been studied on trees \cite{Frib_trees,aidekon_trees,Hammond_trees}, on percolation clusters \cite{BergerGantertPeres, Frib_perco} on random conductance \cite{Frib_conduc, Kious_Frib} and in several other instances \cite{Interlac, Sabot_Enriquez_Zindy, berger2020}. We also refer to the survey \cite{BA-Frib} for further reference to the literature.

Other important efforts were directed towards the study of anisotropic walks in i.i.d.\ environment. We refer to the celebrated series of papers \cite{Sznitman_Zerner, SznitmanClass, Szn_serie, SznitmanSlowdown, SznitmanEffective, Sznitman, Szn_invariance, Sz1, Sz2} in this field. 

A general focus of many RWRE works has been on establishing annealed and quenched invariance principles. We mention the early works \cite{deMasi1, deMasi2, Kipnis-Var} as well as more recent developments \cite{SidoSz, MP, Berger_invariance}. In this regard, a crucial role has been played by models in which the random walks are reversible, that is the random conductance model (RWRC).

In the case where trapping occurs, only few scaling limits have been proved. We mention the recent works on the random walk in i.i.d.\ Dirichlet environment \cite{Poudevigne, Perrel}. A very precise result for RWRC on $\mathbb{Z}^d, d \ge 2$ has been proved by \cite{Kious_Frib}, where an annealed scaling limit towards a Fractional Kinetics process was derived. In \cite{QuenchedBiasedRWRC}, the second author together with A.\ Fribergh and T.\ Lions, proved that the quenched limit holds for all $d \ge 5$. In light of the results of \cite{BarlowCerny, Cerny2d} for the anisotropic case, it is natural to expect quenched limits for all $d \ge 2$.

In the present work we confirm this and prove that the quenched version of the scaling limit of \cite{Kious_Frib} holds in all $d \ge 2$. To the best of our knowledge, this result is the first quenched scaling limit for a biased, sub-ballistic random walk in $d =2, 3, 4$.

\subsection{Definition of the model and main result}
Throughout, $d \geq 2$ is a fixed integer. Let us start by introducing the model of interest, i.e.\ the random walk amongst random conductances. Let $\mathbf{P}$ be a measure on the set of nearest neighbour edges of the square lattice $E(\mathbb{Z}^d)$, i.e.\ $e \in E(\mathbb{Z}^d)$ if $e = [x, y]$ for $x \sim y$ (adjacent vertices in $\mathbb{Z}^d$). In particular, let $\mathbf{P}(\cdot) = \mu^{\otimes E(\mathbb{Z}^d)}$, where $\mu$ is a probability measure on $(0,+\infty)$. For $e \in E(\mathbb{Z}^d)$, we denote by $c_{*}(e)$ the random variable associated to the conductance of the edge $e$ under the measure $\mathbf{P}$, that is $\mathbf{P}(c_{*}(e) \in A) = \mu(A)$ for each Borel set $A$. 

Let us denote the space of all environments by $\Omega$. Let $\omega \in \Omega$ be a fixed realization of the environment. We assume throughout the article that 
\begin{equation*}
    \mathbf{P}\left(c_{*}(e) \ge u \right) = \mu([u,+\infty)) \coloneqq L(u) u^{-\gamma} \quad \textnormal{for any }u \ge 0,
\end{equation*}
with $\gamma \in (0,1)$ a fixed model parameter and where $L(\cdot)$ is a slowly-varying function. We recall that a function $L \colon \mathbb{R_+} \to \mathbb{R}_+$ is slowly varying at infinity if, for any $a \in \mathbb{R}_+$, we have that $\lim_{t \to \infty} L(at)/L(t) = 1$.

We denote by $\{c^{\omega}_{*}(e)\}_{e \in E(\mathbb{Z}^d)}$ the family of conductances in the fixed environment $\omega\in \Omega$. To introduce the drift, consider a unit vector $\vec{\ell} \in \mathbb{S}^{d-1}$ and $\lambda > 0$, let $\ell \coloneqq \lambda \vec{\ell}$ and set, for $e = [x,y] \in E(\mathbb{Z}^{d})$,
\begin{equation*}
    c^{\omega}(e) = c^{\omega}_{*}(e)\exp((x+y)\cdot  \ell).
\end{equation*}
See Figure~\ref{fig:FigJump} below for an intuitive description of the effect of the bias.
\begin{figure}[H]
    \centering
    \includegraphics[width=0.15\linewidth]{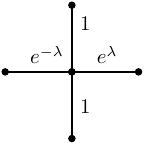}
    \caption{The conductances induced at the origin by the bias when $c^{\omega}_{*}(e) \equiv 1$ for all $e \in E(\mathbb{Z}^2)$ and $\vec{\ell} = e_1$. Any point has the same behaviour as the one drawn above from the point of view of the walk.}
    \label{fig:FigJump}
\end{figure}
We consider the discrete time random walk $(X_n)_{n \geq 0}$ on $\mathbb{Z}^d$ with conductances $(c^\omega(e))_{e \in E(\mathbb{Z}^d)}$. We call $P_{x}^{\omega}$ the quenched law of the random walk on the environment $\omega$, starting from $x \in \mathbb{Z}^d$ and with transition probabilities, for $z \sim y \in \mathbb{Z}^d$,
\begin{equation}\label{QuenchedLawDef}
	p^\omega(z, y) = \frac{c^\omega([z, y])}{\sum\limits_{w \sim z} c^\omega([z, w])},
\end{equation}
and $p^\omega(z, y) = 0$ whenever $z \nsim y$. Let $\mathbb{P}_{x} = \mathbf{E}[P_{x}^{\omega}(\cdot)]$ be the annealed law of the random walk $(X_n)_{n}$ ($\mathbf{E}[\cdot]$ is the expectation with respect to $\mathbf{P}(\cdot)$ and $\mathbf{Var}(\cdot)$ is the corresponding variance). For any $T>0$, let $D^d([0, T])$ denote the space of c\`{a}dl\`{a}g (continuous from the right with left limits) functions from $[0, T]$ to $\mathbb{R}^d$. The main theorem of this article is the following.

\begin{theorem}\label{MainTheorem}
Let $d \ge 2$ and fix any $T > 0$. The following statements hold for $\mathbf{P}$-almost every environment $\omega \in \Omega$. There exists a deterministic $v_0 \in \mathbb{S}^{d-1}$ with $v_0 \cdot \vec{\ell} > 0$ such that, under the quenched law $P^\omega_0(\cdot)$, it holds that
	\begin{equation}\label{FirstStatementMain}
		\left( \frac{X_{\floor{n t}}}{n^\gamma/L(n)} \right)_{0 \le t \le T} \overset{(\mathrm{d})}{\longrightarrow} \left( v_0 C_\infty^{-\gamma} \mathcal{S}_\gamma^{-1}(t) \right)_{0 \le t \le T}
	\end{equation}
 and, furthermore,	\begin{equation}\label{SecondStatementMain}
		\left( \frac{X_{\floor{n t}} - (X_{\floor{n t}} \cdot v_0) v_0}{\sqrt{n^\gamma/L(n)}} \right)_{0 \le t \le T} \overset{(\mathrm{d})}{\longrightarrow}\left( M_d B_{\mathcal{S}_\gamma^{-1}(t)} \right)_{0 \le t \le T},
	\end{equation}
	for a deterministic constant $C_\infty >0$, a deterministic $d \times d$ matrix $M_d$ of rank $d-1$, a standard Brownian motion $B$ and a stable subordinator $\mathcal{S}_\gamma$ of index $\gamma$, independent of $B$. The first convergence holds in the uniform topology on $D^d([0, T])$ while the second holds in the Skorohod's $J_1$-topology on $D^d([0, T])$.
\end{theorem}
\begin{remark1}
    We remark that the result for $d\geq 5$ was established by the second author with A.\ Fribergh and T.\ Lions in \cite{QuenchedBiasedRWRC}; here we derive the result for the case $2\leq d\leq 4$. We will explain the main difficulty to overcome in order to show the result for lower dimensions in the next subsection. 
\end{remark1}

\subsection{Strategy of the proof}

Bolthausen and Sznitman introduced in \cite{Szn_serie} a strategy to prove a functional quenched scaling limit towards Brownian motion starting from an annealed statement, using concentration estimates on certain functionals of the random walk. Later Mourrat \cite{Mourrat} extended it also to stable subordinators. 

The strategy relies on the following observation. Let $F$ be a chosen functional and let $X^{(1)}$ and $X^{(2)}$ be two copies of the random walk $X$. The key idea is that
\begin{equation*}
    \begin{split}
        E&_0^\omega\left[ F(X_n) \right]^2 = E_0^\omega\left[ F(X^{(1)}_n) \right] E_0^\omega\left[ F(X^{(2)}_n) \right], \quad \text{ and } \\
        \mathbf{E}&\left[ E_0^\omega\left[ F(X_n) \right] \right]^2 = \mathbf{E}\left[ E_0^\omega\left[ F(X^{(1)}_n) \right] \right] \mathbf{E}\left[E_0^\omega\left[ F(X^{(2)}_n) \right] \right].
    \end{split}
\end{equation*}
Note that in the first line the two random walks evolve independently in the same environment while in the second line they evolve in two independent copies of the environment. Then
\begin{equation*}
    \begin{split}
        \mathbf{Var}\left(E_0^\omega\left[ F(X_n) \right]\right) 
        &= \mathbf{E}\left[ E_0^\omega\left[ F(X_n) \right]^2 \right] - \mathbf{E}\left[ E_0^\omega\left[ F(X_n) \right] \right]^2\\
        & = \mathbf{E}\left[ E_0^\omega\left[ F(X^{(1)}_n) \right] E_0^\omega\left[ F(X^{(2)}_n) \right] \right] - \mathbf{E}\left[ E_0^\omega\left[ F(X^{(1)}_n) \right] \right] \mathbf{E}\left[E_0^\omega\left[ F(X^{(2)}_n) \right] \right].
    \end{split}
\end{equation*}
The goal then becomes showing that
\begin{equation*}
    \mathbf{E}\left[ E_0^\omega\left[ F(X^{(1)}_n) \right] E_0^\omega\left[ F(X^{(2)}_n) \right] \right] = \mathbf{E}\left[ E_0^\omega\left[ F(X^{(1)}_n) \right] \right] \mathbf{E}\left[E_0^\omega\left[ F(X^{(2)}_n) \right] \right] + R_n,
\end{equation*}
where $R_n$ is some `error term' that needs to be small enough (i.e.\ polynomially small in $n$). 
This means that the variance is small whenever $R_n$ is small, that is, one can approximate the behavior of two random walks evolving in the same environment with the behavior of two random walks evolving in two independent environments.

Typically, the estimates on the variance of the chosen functionals of the random walk are derived by showing that two random walks in the same environment intersect (in the sense that their whole trajectories intersect) only finitely many times in high dimensions. This was the strategy followed in \cite{QuenchedBiasedRWRC}. 

In the ballistic case, \cite{RassoulAghaSeppalainen, Berger_Zeitouni} and \cite{BouchetSabotdosSantos} extended the strategy to lower dimensions by showing that, when the limit is a Brownian motion, it is sufficient to prove that two walks intersect rarely enough. Recently, in \cite{Diffusive}, the second author showed that, under certain conditions on the way the walk explores the space, the second moment of the regeneration times is sufficient for obtaining a quenched limit.

%is that the assumptions on the moments of regeneration times are not optimal for the Gaussian limit (higher moments than the second are assumed). This translates badly in our case because of 

The problem with adapting this strategy to the present case is related to the well-known fact that the maximum of stable random variables is of the same order as the sum. 

To overcome this difficulty we will make the following heuristic rigorous: the crossing of two trajectories is a ``large scale'' event that depends mostly on how (biased) random walks explore the underlying space, while the large increments of regeneration times is due to the presence of isolated edges of very high conductances that trap the walk for a long time. Hence, the latter is a very local event. Thus, we can try to decorrelate the event that two random walks cross each other with the conductances on which the intersection occurs. If this holds, then we can show the main result by showing that two random walks intersect rarely enough.

The core of the proof is contained in Sections~\ref{sect:Intersections}, \ref{sect:Joint} and~\ref{sect:Cross}. In Section~\ref{sect:Intersections} we control the number of times the trajectories of two random walks intersect. In Section~\ref{sect:Joint} we bound the variance of the quenched expectation of the functional of interest by the increment of the regeneration times associated to crossing events. This statement somewhat weakens the usual end-product of similar arguments (see e.g.\ \cite{Berger_Zeitouni, RassoulAghaSeppalainen}) which is purely in terms of the number of intersections. However, coupled with the decorrelation step, this allows us to reduce the impact of each intersection to the variance term. 

In Section~\ref{sect:Cross} we perform the decoupling step. This exploits the joint regeneration structure introduced in \cite{QuenchedBiasedRWRC}. Heuristically, a joint regeneration level is a level $L>0$ in the direction $\vec{\ell}$ such that, when $X^{(1)}$ and $X^{(2)}$ enter the half space $\{z \in \mathbb{Z}^d \colon z \cdot \vec{\ell} \ge L\}$, then they both regenerate. Since the dependence of the environment is only short range, one can show that this structure is Markovian. This means that the future of the walks depends only on their respective starting points. Note that for the two walks to intersect close to a joint regeneration level, the two starting points must be ``close enough''. This observation allows us to decouple the intersection event with the increment of the regeneration times where the crossing happens.

\subsection{Problem simplification} \label{SectionIntroSeries}

Let us state the simplification of the problem based on \cite{Szn_serie, Mourrat} and also used in \cite{QuenchedBiasedRWRC}. In \cite{Kious_Frib} one can find the definition of a regeneration structure with suitable independence properties, more precisely let $\{\tau_k\}_{k \ge 0}$ be the regeneration times (we will summarise their properties in Definition~\ref{def:RegTimesWellDef}). Let $v \coloneqq \mathbb{E}[X_{\tau_2} - X_{\tau_1} ]$ and fix $T>0$. For $t \in [0, T]$, define
\begin{equation}\label{EqnThreeMainQuantities}
	Y_n(t) \coloneqq \frac{X_{\tau_{\floor{tn}}}}{n}, \quad Z_n(t) \coloneqq \frac{X_{\tau_{\floor{tn}}} - v n t}{n^{1/2}} , \quad S_n(t) \coloneqq \frac{\tau_{\floor{tn}}}{\mathrm{Inv}(n)}, \quad \textnormal{and} \quad W_n(t) \coloneqq \left(Z_n(t),S_n(t)\right),
\end{equation}
where $\mathrm{Inv}(u) \coloneqq \inf \left\{ s: \mathbf{P}[c_{*} > s]\le 1/u \right\}$. We remark that $\mathrm{Inv}(n) = L(n)n^{1/\gamma} \approx n^{1/\gamma}$ (where we use the notation $\approx$ for an undefined approximation).

Colloquially, we will often refer to the process $Z_n(t)$ as the \emph{position} and to the process $S_n(t)$ as the \emph{clock}.

Firstly, we will prove that for $\mathbf{P}$-almost all $\omega \in \Omega$, the first three quantities converge in distribution. Subsequently, we will use this fact to show the convergence of $W_n$ in \eqref{EqnThreeMainQuantities}, achieving a joint statement; Theorem~\ref{MainTheorem} will be deduced from these statements. The process $Y_n(t)$ converges almost surely under $\mathbb{P}_0$ to a deterministic limit by \cite[Lemma 11.2]{Kious_Frib} and this gives us already quenched convergence for this quantity. For technical purposes, let us define the process
\begin{equation}\label{ShiftedClock}
	W^*_n(t) \coloneqq W_n\left(t+\tfrac{1}{n}\right) - W_n\left(\tfrac{1}{n}\right).
\end{equation} 
We are now able to state the conditions introduced by \cite{Szn_serie} and \cite{Mourrat} in the following theorem.

\begin{theorem}\label{TheoremVarianceClockProcess}
\noindent Let $d \ge 2$ and fix $T>0$. For any positive and bounded Lipschitz function $F_1 \colon (\mathbb{R}^{d+1})^{m} \rightarrow \mathbb{R}$ and $0 \le t_1 \le \cdots \le t_m \le T$ it holds that
	\begin{equation} \label{EquationBoundVarianceClock}
		\sum_{n = 1 }^{+\infty}\mathbf{Var}\left(E_{0}^{\omega}\left[F_1(W^{*}_{b^n}(t_1),\cdots,W^{*}_{b^n}(t_m))\right]\right)< +\infty,
	\end{equation}
	with $b \in (1,2)$. Moreover, denoting by $Z_n$ the polygonal interpolation of the quantity denoted with the same letter in \eqref{EqnThreeMainQuantities}, we have
	\begin{equation} \label{EquationBoundVarianceTraj}
		\sum_{n = 1 }^{+\infty}\mathbf{Var}\left(E_{0}^{\omega}[F_2(Z_{b^n})]\right)< +\infty,
	\end{equation}
	where $F_2$ is any bounded Lipschitz function from $\mathcal{C}^d([0, T]) \to \mathbb{R}$.
\end{theorem}

We remark that both \eqref{EquationBoundVarianceClock} and \eqref{EquationBoundVarianceTraj} are required (in particular \eqref{EquationBoundVarianceClock} is not sufficient) to obtain tightness in the desired functional space.

\paragraph{Structure of the paper.} In Section~\ref{sect:2} we introduce the notation and recall the properties of regeneration times and regeneration levels. In Section~\ref{sect:Intersections}, we prove that the expected number of intersections in the first $n$ regeneration times is upper bounded by $n^{-1/2+\varepsilon}$. In Section~\ref{sect:MartDiffPos}, we build the martingale difference argument for the process $Z_n$. In Section~\ref{sect:Joint} we improve the martingale difference in order to be able to deal with the process $W_n$. In Section~\ref{sect:Cross} we decouple the intersection event and the associated increment of the regeneration time. In the final section, we use well-known arguments to deduce the main theorem from Theorem~\ref{TheoremVarianceClockProcess}. We will collect in the appendix some useful (but fairly basic) facts about sums of random variables.

\section{Preliminaries}\label{sect:2}
In this section we establish the notation used throughout the article and we recall some key definitions and results from earlier works.
\subsection{Notation}

\paragraph{Basic definitions.} We let $e_1, \dots, e_d$ be the canonical basis of $\mathbb{R}^d$ chosen such that $e_1 \cdot \vec{\ell}\ge e_2 \cdot \vec{\ell} \dots \ge e_d\cdot \vec{\ell} \ge 0$. Furthermore, we fix the notation $e_{d+1}, \dots, e_{2d} = -e_{1}, \dots,-e_{d}$. Let $f_1, \dots, f_d$ be a orthonormal basis of $\mathbb{R}^d$ such that $f_1 = \vec{\ell}$. 

For $x, y \in \mathbb{Z}^d$ we write $x \sim y$ if the two sites are neighbours, furthermore, for a vertex $x \in \mathbb{Z}^d$ and an edge $[e^+, e^{-}] \in E(\mathbb{Z^d})$ we write $x \sim e$ if either $x \sim e^+$ or $x \sim e^-$. For any $x \in \mathbb{Z}^d$ we define 
\begin{equation*}
    \mathcal{V}_{x} \coloneqq \{ x \in \mathbb{Z}^d \colon \|x - z\|_1 \le 1 \} \quad \text{ and } \mathcal{E}_{x} \coloneqq \{[x, y]\colon y \sim x\};
\end{equation*}
those are the vertex and edge neighborhood of a site $x$ (note that $x \in \mathcal{V}_{x}$). For the direction $\vec{\ell}$, and for any $L>0$ we define the \textit{discretised half spaces} as
\begin{equation*}
    \mathcal{H}^+(L)  \coloneqq \{z \in \mathbb{Z}^d \colon x \cdot \vec{\ell} > L \} \quad \text{ and } \quad \mathcal{H}^-(L)  \coloneqq \{z \in \mathbb{Z}^d \colon x \cdot \vec{\ell} \le L \}.
\end{equation*}
Moreover, for any $x \in \mathbb{Z}^d$ we fix $\mathcal{H}^+(x) \coloneqq \mathcal{H}^+(x \cdot \vec{\ell})$ and $\mathcal{H}^-(x) \coloneqq \mathcal{H}^-(x \cdot \vec{\ell})$. We also define the half-planes of edge as 
\begin{align*}
    \mathcal{L}^x &\coloneqq \left\{ [z, y] \colon y \in \mathcal{H}^-(x) \text{ and } z \in \mathcal{H}^-(x) \right\} \cup\mathcal{E}_{x} \\
    \mathcal{R}^x &\coloneqq \left\{ [z, y] \colon y \in \mathcal{H}^+(x) \text{ or } z \in \mathcal{H}^+(x) \right\}\cup\mathcal{E}_{x}.
\end{align*}
For any subset $V \subset \mathbb{Z}^d$ we define $E(V) \coloneqq \{[x, y] \colon x, y \in V\}$, its (outer) vertex-boundary $\partial V \coloneqq \{x \not \in V \colon \exists \, y \in V , x \sim y\}$ and its edge-boundary $\partial_{E} V \coloneqq \{[x, y] \colon x \not \in V, y \in V\}$. We define the directionally tilted boxes of side-lengths $L, L'>0$ and centred at $y \in \mathbb{Z}^d$ 
\begin{equation*}
    \mathcal{B}_{y}(L, L') \coloneqq \left\{x \in \mathbb{Z}^d \colon |(x- y) \cdot \vec{\ell}| \le L, |(x- y) \cdot f_j| \le L' \text{ for all }j = 2, \dots, d\right\}.
\end{equation*}
For a random walk $(X_n)_{n \ge 0}$ on a graph $(V, E)$, we define for any $B \subset V$ the hitting times of $B$:
\begin{equation*}
    T_{B} \coloneqq \inf\{n \ge 0 \colon X_n \in B\}, \quad \text{ and }\quad T^+_{B} \coloneqq \inf\{n \ge 1 \colon X_n \in B\}.
\end{equation*}
We also define the exit time $T^{\mathrm{ex}}_{B} \coloneqq \inf\{n \ge 0 \colon X_n \not\in B\}$. For $R \in \mathbb{R}$ we will write $T_{R} = T_{\mathcal{H}^+(R)}$. We denote by $\theta_n$ the canonical time shift by $n$ units of times. In the context of two random walks, $\theta_{s, t}$ will denote the time shift of the first walk by $s$ unit of times and the second one by $t$.

\paragraph{Probability measures.} We recall that $P_x^\omega$ is the quenched law and $\mathbb{P}_x$ is the annealed law of the random walk started from a vertex $x \in \mathbb{Z}^d$. We will consider two random walks $X^{(1)}$ and $X^{(2)}$; the superscript $(i), i=1,2$ signals that the quantity of interest is related to the $i$-th walk, e.g.\ $\tau^{(1)}_1$ denotes the first regeneration time of the first walk and so on. Most of the time, we will avoid the double superscript and write $X^{(1)}_{\tau_1}$ for $X^{(1)}_{\tau^{(1)}_1}$ (similarly for all other walk-dependent quantities). Given $x,y \in \mathbb{Z}^d$, we denote 
\begin{equation*}
    P^{\omega}_{x, y}\left(X^{(1)} \in A_1, X^{(2)} \in A_2 \right) \coloneqq P^{\omega}_{x}\left(X^{(1)} \in A_1 \right) P^{\omega}_{y}\left(X^{(2)} \in A_2 \right),
\end{equation*}
for any pair of events $A_1, A_2$. Consequently, we define
\begin{equation*}
    \mathbb{P}_{x, y} \left(X^{(1)} \in A_1, X^{(2)} \in A_2 \right) \coloneqq \mathbf{E}\left[P^{\omega}_{x, y}\left(X^{(1)} \in A_1, X^{(2)} \in A_2 \right)\right].
\end{equation*}
The expectations with respect to the last two measures are denoted in the obvious way $E^\omega_{x, y}, \mathbb{E}_{x, y}$. We will also be interested in defining two walks evolving in independent environments, the corresponding measure obeys
\begin{equation*}
    Q_{x, y} \left(X^{(1)} \in A_1, X^{(2)} \in A_2 \right) \coloneqq \mathbb{P}_{x} \left(X^{(1)} \in A_1 \right) \mathbb{P}_{y} \left(X^{(2)} \in A_2 \right).
\end{equation*}
The expectation w.r.t.\ this measure is denoted $\mathbb{E}^{Q}_{x, y}$. For $K \in \mathbb{R}_+$ we will write $\mathbb{P}_{x}^K$ for the annealed law where the conductances adjacent to the starting point are set deterministically to the value $K$. We also set $Q_{x, y}^K(\cdot) = \mathbb{P}_{x}^K(\cdot) \mathbb{P}_{y}^K(\cdot)$. Similarly, we write $\mathbb{P}_{x, y}^K$ for the law of two walks where the conductances adjacent to both starting points are set to $K$.

\subsection{Regeneration times}
We do not give the definition of \emph{regeneration times} since this would distract from the main purpose of the paper. However, we will give detailed references and highlight the properties of regeneration times that we will use in the present work.

For any $K > 1$, the \emph{regeneration times} form a family of increasing random variables $(\tau^{K}_i)_{i \ge 1}$ such that for any $k \ge 1$, the random walk $(X_{\tau_k +i}-X_{\tau_k})_{i \ge 1}$ is independent of $(X_{i})_{0 \le i \le \tau_k}$ and distributed as $(X_i)_{i \ge 0}$ under a measure 
\begin{equation*}
    \mathbb{P}_0^{K}\left(\cdot \mid D^{\vec{\ell}} = +\infty \right),
\end{equation*} 
where $\mathbb{P}^{K}_0(D^{\vec{\ell}}=+\infty) > 0$. This measure has been defined in \cite[Section~5]{Kious_Frib}, see also \cite[Section~3.3]{QuenchedBiasedRWRC} for a more concise presentation of the same material. We will drop the $K$ in the notation $\tau_{i}^K$.

\begin{remark1}\label{rmk:KOpen}
    The parameter $K>0$ appears in the definition because the conductances are not uniformly elliptic. This means that the probability to regenerate at a new maximum is not bounded away from zero. However, one can define the regeneration times to be configuration dependent, i.e.\ restrict to maxima that happen on $K$-open point; these are points $x \in \mathbb{Z}^d$ with $c_*(e) \in [1/K, K]$ for all $e \colon x \in e$. Then, for any fixed $K$, one has strictly positive probability to regenerate at a new maximum that is $K$-open.
\end{remark1}

Roughly speaking, $D^{\vec{\ell}}$ is the first time at which the random walk backtracks to the origin level in the direction $\vec{\ell}$. Then $(\tau_i)_{i \ge 0}$ is the increasing sequence of times at which the random walk never backtracks to the current level at future times. That is why the law of the shifted walk is conditioned on the event $\{D^{\vec{\ell}}=+\infty\}$ and why the regeneration times are not stopping times. The random variable $D \equiv D^{\vec{\ell}}$ is defined in \cite[(5.3)]{Kious_Frib}. See also \eqref{eqn:DSingleRegeneration} for a definition of the random variable $D$.

Another key property of regeneration times is that the diameter of the smallest ball in which the trajectory is contained between two regeneration times has a fast-decaying tail. More precisely, we define for all $k \ge 1$ and any $\alpha>d+3$ (we will fix this constant throughout the paper)
\begin{equation}\label{eqn:RegenerationBox}
    \chi_k = \inf \left\{m \in \mathbb{N}: \{X_i - X_{\tau_{k-1}}: \tau_{k-1}\le i \le \tau_{k}\} \subset \mathcal{B}(m,m^{\alpha})\right\}.
\end{equation}

\begin{definition1}\label{def:RegTimesWellDef}
    A regeneration structure (i.e.\ a sequence of regeneration times and regeneration points) in the direction $\vec{\ell} \in \mathbb{S}^{d-1}$ is well defined if it satisfies the following set of properties:
    \begin{enumerate}
        \item $\{\tau_i\}_{i \ge 1}$ is an increasing sequence of almost surely finite random variables. Moreover, there exists $\eta>0$ depending only on $K$ and $\vec{\ell}$ such that the event $\{D^{\vec{\ell}} = \infty\}$ has probabilities $\mathbb{P}_0(D^{\vec{\ell}} = \infty) \ge \eta$ and $\mathbb{P}^K_0(D^{\vec{\ell}} = \infty) \ge \eta$.
        \item For all $k \in \mathbb{N}$ we have that $X_{\tau_k} \cdot \vec{\ell} > X_j$ for all $0 \le j < \tau_k$ and $X_{\tau_k} \cdot \vec{\ell} < X_i$ for all $i > \tau_k$.
        \item For all $k \ge 0$, $(X_{\tau_k +i}-X_{\tau_k})_{i \ge 1}$ is independent, under $\mathbb{P}_0$, of $(X_{i})_{0 \le i \le \tau_k}$ and, for all $k \ge 1$ it is distributed as $(X_i)_{i \ge 0}$ under the measure $\mathbb{P}_0^{K}(\cdot \mid D^{\vec{\ell}} = +\infty )$.
        \item For all $M>0$ there exists $K_0>0$ such that for all $K \ge K_0$ we have that
        \begin{equation*}
            \mathbb{P}_{0}\left( \chi_1 > n \right)\le Cn^{-M}, \quad \text{ and } \quad \mathbb{P}_{0}^K\left( \chi_1 > n \right)\le Cn^{-M}.
        \end{equation*}
        \item The sequence of increments of regeneration times $\{\tau_{i+1} - \tau_i\}_{i \ge 1}$ is in the domain of attraction of a $\gamma$-stable random variable. This means that each increment has Laplace transform $E[\exp(-\lambda X)] \propto e^{-\lambda^\gamma}$ (the marginal of a stable subordinator).
    \end{enumerate}
\end{definition1}
We observe that in all the definition above $\tau_1$ and $X_{i}, 0 \le i \le \tau_1$ is indeed independent of the future of the trajectory but has a different distribution w.r.t.\ the other regeneration intervals.

We state a result that summarizes several results of \cite{Kious_Frib}. In particular one needs to look at Theorem 5.1-5.3-5.4, Lemma~6.1 and Proposition~10.3 of that paper.
\begin{theorem}\label{theo:FKRegStruct}\cite{Kious_Frib}.
    There exists a regeneration structure in the direction $\vec{\ell}$ that is well-defined in the sense of Definition~\ref{def:RegTimesWellDef}.
\end{theorem}

We also remark that, by their construction, we have that deterministically $X_{\tau_j} \cdot \vec{\ell} - X_{\tau_{j-1}} \cdot \vec{\ell} \ge 2/\sqrt{d}$, we will employ this fact repeatedly. This construction is very useful to understand the behaviour a random walk in random environment. However, when studying two random walks in random environment another construction introduced in \cite{QuenchedBiasedRWRC} is more suitable.

\subsection{Regeneration levels}
Let $X^{(1)}$ and $X^{(2)}$ be two random walks defined as $X$, independent under the quenched law but evolving in the same environment. When considering two random walks in random environment, the good notion to understand their joint behaviour is to consider \emph{joint regeneration levels}. We begin with an informal introduction of these objects. 

For any $K > 0$, the \emph{joint regeneration levels} is a family of increasing random variables $(\mathcal{L}^{K}_i)_{i \ge 1}$ such that for any $k \ge 0$, conditionally on 
\[\{\mathcal{L}_k,(X^{(1)}_{i})_{0 \le i \le T_{\mathcal{L}_k}},(X^{(2)}_{j})_{0 \le j \le T_{\mathcal{L}_k}}\},\]
the random walks $(X^{(1)}_{T_{\mathcal{L}^{K}_k}+i})_{i \ge 0}$ and $(X^{(2)}_{T_{\mathcal{L}^{K}_k}+i})_{i \ge 0}$ are distributed as $(X^{(1)}_i)_{i \ge 0}$ and $(X^{(2)}_j)_{j \ge 0}$ under 
\begin{equation}\label{eqn:MarkovJRL}
    \mathbb{P}_{X^{(1)}_{T_{\mathcal{L}^{K}_k}},X^{(2)}_{T_{\mathcal{L}^{K}_k}}}^{K}\left(\cdot \,|\, \mathcal{D}^{\bullet} = +\infty \right),\quad \text{ where } \quad \mathbb{P}_{X^{(1)}_{T_{\mathcal{L}^{K}_k}},X^{(2)}_{T_{\mathcal{L}^{K}_k}}}^{K}(\mathcal{D}^{\bullet}=+\infty) > 0.
\end{equation}
The conditioned measure appearing in \eqref{eqn:MarkovJRL} and the Markov property just stated are the subject of \cite[Theorem~3.25]{QuenchedBiasedRWRC}. For $j \in \{1,2\}$, we introduce the notation $\rho^{(j)}_i \coloneqq T^{(j)}_{\mathcal{L}^{K}_i}$. We will drop $K$ in the notation, the reason for the appearance of this parameter is explained in Remark~\ref{rmk:KOpen}.

In words, a level $L>0$ is a joint regeneration levels if both $T_{\mathcal{H}^+(L)}^{(1)}$ and $T_{\mathcal{H}^+(L)}^{(2)}$ are regeneration times for the respective walk. That is, upon entering the half-space $\mathcal{H}^+(L)$ both walks regenerate. In particular, for any $j \in \{1,2\}$ we have the inclusion $\{\rho^{(j)}_i\}_{i\ge 1} \subset \{\tau^{(j)}_i\}_{i\ge 1}$.

We note that the definition above was given with $\mathcal{D}^{\bullet} \equiv \mathcal{D}^{\bullet, \vec{\ell}}$ in \cite[Section~3.4]{QuenchedBiasedRWRC} (see also \eqref{definition_D} for a concise definition). Understanding the behaviour of two random walks is much easier using joint regeneration levels since we can decompose their trajectories between regeneration levels and use the above Markov property.

\begin{definition1}\label{def:JointRegLevWellDef}
    Let $\mathcal{U} \coloneqq \{x \in \mathbb{Z}^d \colon -e_1 \cdot \vec{\ell} \le x \le e_1 \cdot \vec{\ell} \}$. A joint regeneration structure (i.e.\ a sequence of joint regeneration levels and joint regeneration points) in the direction $\vec{\ell}\in \mathbb{S}^{d-1}$ is well defined if it satisfies the following properties:
    \begin{enumerate}
        \item $\{\mathcal{L}_i\}_{i \ge 1}$ is an increasing sequence of almost surely finite random variables. Moreover, there exists $\eta>0$ depending only on $K$ and $\vec{\ell}$ such that the event $\{\mathcal{D}^{\bullet} = \infty\}$ has probabilities $\mathbb{P}_{U_1, U_2}(\mathcal{D}^{\bullet} = \infty) \ge \eta$ and $\mathbb{P}^K_{U_1, U_2}(\mathcal{D}^{\bullet} = \infty) \ge \eta$ uniformly over all $U_1, U_2 \in \mathcal{U}$.
        \item For all $k \in \mathbb{N}$, $i = 1, 2$ we have that $X^{(i)}_{T_{\mathcal{L}_k}} \cdot \vec{\ell} > X^{(i)}_j$ for all $0 \le j < T^{(i)}_{\mathcal{L}_k}$ and $X^{(i)}_{\tau_k} \cdot \vec{\ell} < X^{(i)}_u$ for all $u > T^{(i)}_{\mathcal{L}_k}$.
        \item The markovian property, stated in \eqref{eqn:MarkovJRL} and above, holds for $\{\mathcal{L}_i\}_{i \ge 1}$ and $\{\mathcal{D}^{\bullet} = \infty\}$.
        \item For all $M>0$ there exists $K_0>0$ such that for all $K \ge K_0$ we have that, uniformly over all $U_1, U_2 \in \mathcal{U}$.
        \begin{equation*}
            \mathbb{P}_{U_1, U_2}\left( \mathcal{L}_1 > n \right)\le Cn^{-M}, \quad \text{ and } \quad \mathbb{P}^K_{U_1, U_2}\left( \mathcal{L}_1 > n \right)\le Cn^{-M}.
        \end{equation*}
    \end{enumerate}
\end{definition1}

We state a result that summarises several results of \cite{QuenchedBiasedRWRC}. In particular one needs to look at Proposition 3.18, Proposition 3.24, Theorem 3.25 of that paper.

\begin{theorem}\label{theo:FLSRegStruct}\cite{QuenchedBiasedRWRC}.
    There exists a joint regeneration structure in the direction $\vec{\ell}$ that is well-defined in the sense of Definition~\ref{def:JointRegLevWellDef}.
\end{theorem}

From now on we will drop the dependence on $\vec{\ell}$ for the regeneration structures (regular and joint). The reader should always assume that the regeneration times (resp.\ joint regeneration levels) considered are defined in the direction $\vec{\ell}$. In the context of two walks $X^{(1)}, X^{(2)}$ we will use the notation $\{D^\otimes = \infty\} \coloneqq \{D^{(1)} = \infty\} \cap \{D^{(2)} = \infty\} $, that is the event that both regenerate. We remark that, in general $\{D^{\otimes} = \infty\} \neq \{\mathcal{D}^\bullet = \infty\}$ and in fact $\{\mathcal{D}^\bullet = \infty\} \subseteq \{D^{\otimes} = \infty\}$.

To end this section we define, for $n \ge 0$ and $\varepsilon >0$, the event
\begin{equation}\label{eqn:UniformBoundBoxes}
    F_{n, \varepsilon} \coloneqq \left\{ \forall k \in \{ 1, \dots, n^2 \}, \chi_k \le n^\varepsilon \right\}.
\end{equation}
For two random walks evolving independently in the same environment we define the event
\begin{equation}\label{eqn:UniformBoundJointSlabs}
    L_{n, \varepsilon} \coloneqq \left\{ \forall k \in \{ 1, \dots, n^2 \}, \mathcal{L}_{k} - \mathcal{L}_{k - 1} \le n^\varepsilon \right\}.
\end{equation}

\begin{lemma}\label{lemma:GoodEventsBoxes}
For all $M>0$ there exits a $K_0>0$ such that for all $K \ge K_0$
\begin{equation}\label{eqn:FNepsi}
    \mathbb{P}_{0}((F_{n, \varepsilon})^c) \le C n^{-M}.
\end{equation}
Moreover, for all $M>0$ there exits a $K_0>0$ such that for all $K \ge K_0$
\begin{equation}\label{eqn:LNepsi}
    \mathbb{P}_{0, 0}((L_{n, \varepsilon})^c) \le C n^{-M}.
\end{equation}
\end{lemma}

\begin{proof}
    The proof of the first fact is a straightforward consequence of Theorem~\ref{theo:FKRegStruct}. Indeed,
    \begin{equation*}
        \mathbb{P}_{0}((F_{n, \varepsilon})^c) \le \sum_{i = 1}^{n^2} \mathbb{P}_{0}(\chi_i > n^{\varepsilon}) \le n^{2} n^{-M},
    \end{equation*}
    which is enough to conclude as we can choose $M$ arbitrarily large by tuning $K_0$. The second statement follows with the same reasoning from Theorem~\ref{theo:FLSRegStruct}. 
\end{proof}

The reader should assume in the following that $d =2, 3, 4$. Our methods extend to higher dimensions, but we would need to carry the dependence of constants on the dimension throught our proofs.

\section{The intersection structure}\label{sect:Intersections}

\noindent Let $X^{(1)}_{\cdot}$ and $X^{(2)}_{\cdot}$ be two random walks and define the counting random variable
\begin{equation}\label{eqn:Intersections}
    I_{n} \coloneqq \sum_{\substack{z \in \mathcal{B}_{0}(n, n^\alpha) }} \mathds{1}_{\{ X^{(1)} \text{ visits } \mathcal{V}_z \}} \mathds{1}_{\{ X^{(2)} \text{ visits } \mathcal{V}_z \}}.
\end{equation}
We give the following bound on the number of intersections.

\begin{proposition}\label{prop:FewCrossing}
Fix $\delta>0$. There exists $K_0 > 0$ such that, for all $K \ge K_0$ and for every $n \in \mathbb{N}$, we have
\begin{equation*}
    \mathbf{E}\left[ E^\omega_{0, 0}\left[ I_n \right] \right] \le C n^{1/2 + \delta},
\end{equation*}
where $C=C(\delta)>0$ depends only on $\delta$.
\end{proposition}
We will actually prove a slightly different (and admittedly less-transparent) result regarding the intersections which, despite its less straightforward interpretation, it immediately implies Proposition~\ref{prop:FewCrossing} and it is a bit easier to employ in our computations.
In order to present such result, we need to introduce some further notation.

For $n \in \mathbb{N}$ and $\varepsilon\in (0, \tfrac{1}{10})$ we define the set
\begin{equation}\label{eqn:BadRegLevels}
    \mathrm{JRL}^\le(n, \varepsilon) \coloneqq \left\{ k \colon \mathcal{L}_{k} \le n, \Big\| X^{(1)}_{T_{\mathcal{L}_k}} - X^{(2)}_{T_{\mathcal{L}_k}} \Big\|_2 \le n^{10\varepsilon} \right\}.
\end{equation}
Note that, in the definition above, the constant $10$ does not play any meaningful role; any absolute constant $c>0$ not depending on $n, \varepsilon$ would work. 

The following is the main result of this section.
\begin{proposition}\label{prop:CrossingJRL}
There exists a constant $c>0$ depending only on the parameters of the model such that, for any given $\varepsilon > 0$, there exists $K_0 > 0$ such that for all $K\ge K_0$ and for all $n \in \mathbb{N}$, we have
\begin{equation*}
    \mathbf{E}\left[ E^\omega_{0, 0}\left[ |\mathrm{JRL}^\le(n, \varepsilon)| \right] \right] \le C n^{1/2 + c\varepsilon},
\end{equation*}
where $C$ depends only on $\varepsilon > 0$.
\end{proposition}
As we said earlier, Proposition \ref{prop:FewCrossing} is a direct consequence of the above result; this will be established after proving Proposition \ref{prop:CrossingJRL} and it will constitute the last result of this section.

\begin{remark1}
    We highlight that the results of Proposition~\ref{prop:FewCrossing} and Proposition~\ref{prop:CrossingJRL} are close to being sharp only in $d = 2$ and they could be improved in $d=3, 4$. However, the lowest dimension considered should be seen as the worst case (and most challenging) for our method of proof.
\end{remark1}

The proof of Proposition \ref{prop:CrossingJRL} requires some intermediate steps. We start by analysing the intersection behavior of two RWRC evolving in \textit{independent} environments. 

 We recall that the constant $\alpha>d+3$ was set above \eqref{eqn:RegenerationBox}, since we are in $d \le 4$ we can consider $\alpha  \le 8$. Recall also that $\vec{v}_{0} \coloneqq v/\|v\|$, where $v \coloneqq \mathbb{E}^K_{0}[X_{\tau_1} | D = \infty]$. Denote the other vectors in the orthonormal basis of $\mathbb{R}^d$ together with $\vec{v}_{0}$ by $\vec{u} = \vec{u}_0, \vec{u}_1, \dots, \vec{u}_{d-2}$ chosen such that $\vec{u}_i \cdot \vec{\ell} \ge 0$.

\begin{lemma}\label{lemma:SeparationIndependentV2}
For all $\varepsilon>0$, there exists a constant $c = c(\beta, \vec{\ell}, \lambda, d, K)>0$ such that, uniformly over all $U_1, U_2 \in \mathcal{U}$, we have
    \begin{equation*}
        Q^K_{U_1, U_2}\Bigg( \inf_{\substack{k_1 \colon |X^{(1)}_{k_1} \cdot \vec{\ell} - L|\le 1 \\ k_2 \colon |X^{(2)}_{k_2} \cdot \vec{\ell} - L|\le 1}} \left| \langle (X^{(1)}_{k_1} - X^{(2)}_{k_2}), \vec{u}\rangle \right| \ge n^{10\varepsilon} \text{ for all } L \in [n^{40\varepsilon}, n] \,\, \Big| \,\,D^{\otimes}=+\infty\Bigg) \ge \frac{c}{n^{1/2 + \varepsilon}}.
    \end{equation*}
\end{lemma}

\begin{proof}
Here, we may use $\langle \cdot \,, \cdot \rangle$ for the inner product in $\mathbb{R}^d$ for a more explicit notation. Consider the unit vector $\vec{u}$ and define, for all $j \ge 0$, the quantity 
\begin{equation}\label{eqn:OrthogonalDistance}
    O_j \coloneqq \left\langle X^{(1)}_{\tau_j} - X^{(2)}_{\tau_j}, \vec{u} \right\rangle.
\end{equation}
%We introduce some events related to the statement. 
In order to prove that the two random walks stay at a certain distance for a long time, we need to ensure that their regeneration points are far from each other and their regeneration boxes (recall the definition~\eqref{eqn:RegenerationBox}) are small with respect to the distance of the regeneration points. Let us now introduce these events formally. Fix some universal constant $\hat{c}>0$ (that we will adjust so that the events will have the desired probabilities for $j$ small). We set:
\begin{equation}\label{eqn:GoodEventsExploration}
\begin{split}
    B_j &\coloneqq \left\{ \chi^{(i)}_j \le \hat{c}j^\varepsilon; i = 1, 2 \right\},  \\
    C_j &\coloneqq \left\{ \left| \left\langle X^{(i)}_{\tau_j} - \mathbb{E}^K_{U_i} [ X^{(i)}_{\tau_j} | D = \infty] , \vec{v}_0 \right\rangle \right| < \hat{c}j^{\kappa}; i =1, 2 \right\}, \\
    D_{j} &\coloneqq \left\{ \max_{k \in \llbracket j-j^{\kappa}, j\rrbracket} \left| \left\langle X^{(i)}_{\tau_k} - X^{(i)}_{\tau_j} , \vec{u} \right\rangle \right| < \hat{c}j^{\eta} ; i =1, 2 \right\},\\
    G_j &\coloneqq \left\{ \max_{k = 0, \dots, d-2} \left| \left\langle X^{(i)}_{\tau_j} , \vec{u}_k \right\rangle \right|  < \hat{c}j^{\kappa} ; i =1, 2 \right\}.
    \end{split}
\end{equation}
Here $\kappa>1/2$, $\eta \in (\kappa/2, 1/2)$ and $\varepsilon>0$ is small at will. We know from \cite[Lemma 6.2]{Kious_Frib} that, for all $M>0$, there exits a $K_0>0$ such that for any $K \ge K_0$ the following statements hold:
\begin{itemize}
    \item $\mathbb{P}((B_j)^c) \le C j^{-M}$.
    \item $\mathbb{P}((C_j)^c) \le C j^{-M(2\kappa - 1)}$ since $\mathbb{E}_{0}[|\langle X^{(i)}_{\tau_2} - X^{(i)}_{\tau_1}, \vec{v}_0 \rangle|^{2M}] < \infty$, $\mathbb{E}^K_{0}[\langle X^{(i)}_{\tau_1} - \mathbb{E}^K_{U_i} [ X^{(i)}_{\tau_1} | D = \infty], \vec{v}_0 \rangle | D =\infty] = 0$ and applying Lemma~\ref{lemma:PolynomialBound}. 
    \item $\mathbb{P}((D_j)^c) \le C j^{-M(2\eta - \kappa)}$ since $\mathbb{E}_{0}[|\langle X^{(i)}_{\tau_2} - X^{(i)}_{\tau_1}, \vec{u} \rangle|^{2M}] < \infty$, $\mathbb{E}^K_{0}[\langle X^{(i)}_{\tau_1}, \vec{u} \rangle | D =\infty] = 0$ and applying Lemma~\ref{lemma:PolynomialBound}.
    \item $\mathbb{P}((G_j)^c) \le C j^{-M(2\kappa-1)}$ since $\mathbb{E}_{0}[|\langle X^{(i)}_{\tau_2} - X^{(i)}_{\tau_1} , \vec{u}_k \rangle|^{2M}] < \infty$,$\mathbb{E}^K_{0}[\langle X^{(i)}_{\tau_1}, \vec{u}_k \rangle | D =\infty] = 0$ for all $k = 0, \dots, d-1$ and applying Lemma~\ref{lemma:PolynomialBound}.
\end{itemize}
\noindent We can then choose $\kappa$ and $\eta$ close to $1/2$ such that $U_{j} \coloneqq B_j \cap C_j \cap D_j \cap G_j$ has probability $\mathbb{P}(U_j)\ge 1 - \tfrac{1}{j^{a}}$ with $a > 3/2$.

\begin{figure}[H]
    \centering
    \includegraphics[width=0.5\linewidth]{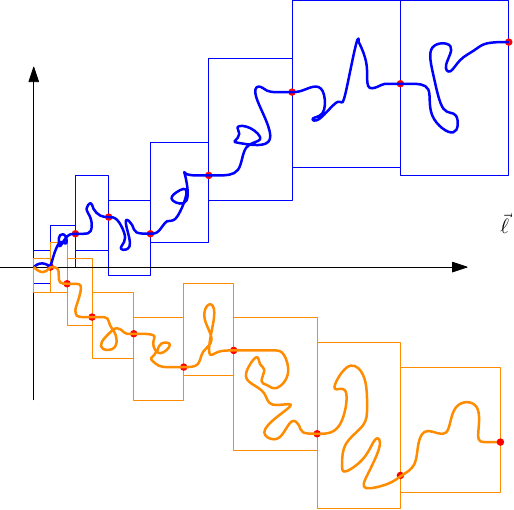}
    \caption{In blue the trajectory of $X^{(1)}$, in orange the trajectory of $X^{(2)}$ and in red the regeneration points. Here is the idea behind the definition of the events $A(n)$ and its components $B_j, C_j, D_j, G_j$. The regeneration boxes in which the two random walks evolve may get larger, fluctuate in their position and size but keep getting further apart from each other. }
    \label{fig:FigFar}
\end{figure}

We claim that, under $Q^K_{U_1, U_2}( \cdot | D^{\otimes}=+\infty)$, the process $(\widetilde{O}_j)_{j \ge 0} = (O_j - (U_1 - U_2)\cdot \vec{u})_{j \ge 0}$ is a random walk on $\mathbb{R}$ (it is a sum of i.i.d.\ mean $0$ random variables). Indeed, this is an immediate consequence of the fact that $(X^{(1)}_{\tau_j})_{j \ge 0}$ and $(X^{(2)}_{\tau_j})_{j \ge 0}$ are two independent and identically distributed random walks on $\mathbb{Z}^d$. We highlight that this is true under $Q^K_{U_1, U_2}( \cdot | D^{\otimes}=+\infty)$ and it would not be true for two walks evolving in the same environment. Moreover $\mathbb{E}^K_{0}[X^{(i)}_{\tau_1} - X^{(i)}_{\tau_0}] = \vec{v}$ and $\vec{v} \cdot \vec{u} = 0$.

It is easy to see that \cite[Lemma 6.2]{Kious_Frib} also implies that $\mathbb{E}[\widetilde{O}_j^{2M}] < \infty$ for $M$ large. Hence, we can apply \cite[Lemma~4.3]{Berger_Zeitouni} to the events $U_j$ and the random walk $\widetilde{O}_j$, in particular to the event $A(n) \coloneqq \{\text{For all }j<n, \widetilde{O}_j \ge \floor{j^{1/2 - \eta'}} \}$. This yields 
\begin{equation*}
    Q_{U_1, U_2}^K\left(\text{For all }j<n, \widetilde{O}_j \ge \floor{j^{1/2 - \eta'}} \text{ and } U_j \text{ occurs} \mid D^{\otimes}=+\infty \right) \ge \frac{C}{n^{1/2 + \delta}},
\end{equation*}
where $\eta' \in (0, 1/2-\eta)$ and $\delta>0$ can be chosen arbitrary small. 

\paragraph{Claim:} There exists a constant $c_1$, depending only on the parameters of the model, such that
\begin{equation}\label{eqn:EquationSeparationGivenEvent}
    A(c_1 n) \cap \bigcap_{j = 1}^{c_1 n} U_j \subseteq \Bigg\{\text{for all } L \in [n^{40\varepsilon}, n], \inf_{\substack{k_1 \colon |X^{(1)}_{k_1} \cdot \vec{\ell} - L|\le 1 \\ k_2 \colon |X^{(2)}_{k_2} \cdot \vec{\ell} - L|\le 1}} \left| \langle (X^{(1)}_{k_1} - X^{(2)}_{k_2}), \vec{u}\rangle \right| \ge n^{10\varepsilon}\Bigg\}.
\end{equation}
For a visual aid of how the separation is achieved see Figure~\ref{fig:FigFar} above.

We drop the dependence of the events on $\hat{c}$ as it is the same for all of them. Fix any level $40\varepsilon \le L \le n$. Let $j(L)$ be the index of the first (possibly unique) regeneration box intersecting level $L$ by either $X^{(1)}$ or $X^{(2)}$. Note that, on the event, $\cap_{j = 1}^n (B_j \cap C_j \cap G_j)$, in order to reach distance $L$ we need to use at least and at most $ c_2L \le j(L) \le c_3L$ regeneration times (with $c_2, c_3$ depending only on the parameters of the model). Furthermore, on $A(n)$ we have that at the $j$-th regeneration points the walks are at distance $\|X^{(1)}_{\tau_j} - X^{(2)}_{\tau_j}\| \ge j^{1/2 - \eta'}$. Applying this to $j(L)$ we see that the distance $|\langle X^{(1)}_{\tau_{j(L)}} - X^{(2)}_{\tau_{j(L)}}, \vec{u} \rangle |$ is at least
\[(c_2L)^{1/2 - \eta'} > c_4 n^{20\varepsilon - 40\varepsilon \eta'} \ge c_4 n^{18 \varepsilon }\]
for $\eta'$ chosen properly. Furthermore, we notice that on the events $\cap_{j = 1}^n B_j$ 
we have that the two associated regeneration boxes are at distance larger than $j(L)^{1/2 - \eta'} - 2c_{d}j(L)^{\alpha \varepsilon}$, where $c_d$ is an absolute constant depending only on the dimension. We also have that
\[j(L)^{1/2 - \eta'} - 2c_{d}j(L)^{\alpha \varepsilon} \ge n^{20\varepsilon - 40\varepsilon \eta'} - 2c_{d}j(c_3n)^{5\varepsilon} \ge c_5 n^{15 \varepsilon}.\]
We have two options:
\begin{itemize}
    \item $X^{(1)}_{\tau_{j(L)}}$ is \textbf{leading}, which means that $X^{(1)}_{\tau_{j(L)}} \cdot \vec{\ell} \ge X^{(2)}_{\tau_{j(L)}} \cdot \vec{\ell}$.
    \item $X^{(1)}_{\tau_{j(L)}}$ is \textbf{trailing}, which means that $X^{(1)}_{\tau_{j(L)}} \cdot \vec{\ell} < X^{(2)}_{\tau_{j(L)}} \cdot \vec{\ell}$.
\end{itemize}
These two situations are symmetric as the labelling of the two walks is arbitrary. Without loss of generality we analyze the first one.

We aim to show that the closest distance between a regeneration box of $X^{(2)}$ and the regeneration box associated with $\{X^{(1)}_{\tau_{j(L)} + k}\}_{k = 1}^{\tau_{j(L)+1}}$ is increasing in $L$; for clarity of exposition we write $j = j(L)$. 
We start by noticing that there exists a universal constant $c_6 = c(\lambda, \vec{\ell}, \vec{v}_0, d)$, depending only on the parameters of the model and $\vec{v}$ (which is not a parameter but is deterministic nonetheless) such that the box that crosses the level $L$ is associated to one point among $X^{(2)}_{\tau_j}, \dots ,X^{(2)}_{\tau_{j + 2c_6 j^{\kappa}}}$. Indeed, on the event $\cap_{j = 1}^n C_j$ we know that the distance between $(X^{(1)}_{\tau_j}$ and $X^{(2)}_{\tau_j})$ in the direction $\vec{\ell}$ is less than $j^{\kappa}$. Moreover, for each $t > 0$ it takes at most $c_6 t$ regeneration times to the walk to travel distance at least $t$ in the direction $\vec{\ell}$. Hence, all the boxes that happen after $X^{(2)}_{\tau_{j + 2c_6 j^{\kappa}}}$ are after the walk $X^{(2)}_{\tau_{j + 2c_6 j^{\kappa}}}$ has overcome the level $L$ (and by regenerating is not ``coming back''). 

Finally, for all $j$ large enough, on the event $\cap_{j = 1}^n D_j$ the regeneration point among $X^{(2)}_{\tau_j}, \dots ,X^{(2)}_{\tau_{j + 2c_6 j^{\kappa}}}$ which is the closest to $X^{(1)}_{\tau_j}$ has a distance bounded below by $j^{1/2 - \eta'} - c_7j^{\eta}$. Here, $c_7$ is another absolute constant depending only on the geometry of the problem and not on $j$ or $n$. In turn, this implies that the closest a regeneration box can get is $j^{1/2 - \eta'} - c_2j^{\eta} - 2 j^{\varepsilon \alpha} \ge c_8L^{1/2 - \eta'} \ge n^{15 \varepsilon}$. Notice, that all these bounds below are increasing in $j$ (and hence in $L$). This concludes the proof. 
\end{proof} 

Before proceeding we define the event
\begin{equation}\label{eqn:SeparationEvent}
    \mathrm{Sep}(n) \coloneqq \Bigg\{\text{for all } L \in [0, n], \inf_{\substack{k_1 \colon |X^{(1)}_{k_1} \cdot \vec{\ell} - L|\le 1 \\ k_2 \colon |X^{(2)}_{k_2} \cdot \vec{\ell} - L|\le 1}} \left| \langle (X^{(1)}_{k_1} - X^{(2)}_{k_2}), \vec{u}\rangle \right| \ge 3\Bigg\}.
\end{equation}
Observe that, on this event, the two walks do not see the same conductance in a large portion of the space.

\begin{lemma}\label{lemma:SeparationEventually}
For all $\varepsilon>0$, there exist two constants $\Lambda = \Lambda(\beta, \vec{\ell}, \lambda, d, K)$ $c = c(\beta, \vec{\ell}, \lambda, d, K)>0$ such that, uniformly over all $U_1, U_2 \in \mathcal{U}$ such that $\|(U_1 - U_2)\cdot \vec{u}\| \ge \Lambda$, we have
    \begin{equation*}
        Q^K_{U_1, U_2}\Bigg( \text{for all } L \in [n^{40\varepsilon}, n], \inf_{\substack{k_1 \colon |X^1_{k_1} \cdot \vec{\ell} - L|\le 1 \\ k_2 \colon |X^2_{k_2} \cdot \vec{\ell} - L|\le 1}} \left| \langle (X^1_{k_1} - X^2_{k_2}), \vec{u}\rangle \right| \ge n^{10\varepsilon} \text{ and }\mathrm{Sep}(n) \,\, \Big| \,\,D^{\otimes}=+\infty\Bigg) \ge \frac{c}{n^{1/2 + \varepsilon}}.
    \end{equation*}
 \end{lemma}
 \begin{proof}
     The proof relies on the observations made in the proof of Lemma~\ref{lemma:SeparationIndependentV2}. Assume once again that the event
     \begin{equation*}
         A(c_1 n) \cap \bigcap_{j = 1}^{c_1 n} U_j
     \end{equation*}
     holds. 
     We claim that, on this event, for all $L>0$, there exists a non-decreasing function $f(L)\colon f(L) \to \infty$ as $L\to \infty$ and a constant $\tilde{c}>0$ (depending only on the parameters) such that 
     \begin{equation}\label{eqn:FarAlways}
         \inf_{\substack{k_1 \colon |X^{(1)}_{k_1} \cdot \vec{\ell} - L|\le 1 \\ k_2 \colon |X^{(2)}_{k_2} \cdot \vec{\ell} - L|\le 1}} \left| \langle (X^{(1)}_{k_1} - X^{(2)}_{k_2}), \vec{u}\rangle \right| \ge \tilde{c}f(L).
     \end{equation}
     Before proving this claim let us show the main result. Indeed, if \eqref{eqn:FarAlways} happens then there must exist $L_1 \colon \tilde{c} f(L_1)>3$. Furthermore, we know that to reach such $L_1$ the two walks take a number of steps (regeneration times) of the order $L_1$. Hence their orthogonal movement is at most $L_1^{\eta}$. Thus, if we choose the starting points $U_1, U_2$ at an initial distance $\|(U_1 - U_2)\cdot \vec{u}\| \ge \Lambda$, they will not cross if $\Lambda$ is large enough. We note that this choice depends only on $L_1$, whose choice depends only on the parameters $\beta, \vec{\ell}, \lambda, d, K$. Finally, under $A(c_1 n)$ we know that the two walks regenerate after reaching distance $n$; this fact rules out any possible ``crossing'' in $[0, n]$ from regeneration boxes that happen in the future.

     We now turn to showing the claim; note that we only need to prove it for $L$ large enough (independent of $n$) as the constant $\tilde{c}$ can be chosen small enough to ensure that the statement hold for small $L$. Fix any level $L$. As noted before, one can define $j(L)$ to be the smallest index associated with a regeneration box that intersects level $L$ (either by $X^{(1)}$ or $X^{(2)}$). Recall that there exist constants $c_2, c_3 >0$, not depending on $L$, such that for all $L \in [0, n]$ we have $c_2L \le j(L)\le c_3L$. Without loss of generality we assume that $X^{(1)}$ is leading. Before proceeding, let us first recall that in the proof of Lemma~\ref{lemma:SeparationIndependentV2} we established that:
     \begin{enumerate}
         \item $|\langle X^{(1)}_{\tau_{j(L)}} - X^{(2)}_{\tau_{j(L)}}, \vec{u} \rangle |$ is at least $(c_2L)^{1/2 - \eta'}$.
         \item The box of $X^{(2)}$ that crosses the level $L$ is associated to one point among $X^{(2)}_{\tau^2_j}, \dots ,X^{(2)}_{\tau^2_{j + 2c_6 j^{\kappa}}}$
         \item The regeneration point among $X^{(2)}_{\tau_j}, \dots ,X^{(2)}_{\tau_{j + 2c_6 j^{\kappa}}}$ which is the closest to $X^{(1)}_{\tau_j}$ has a distance bounded below by $j^{1/2 - \eta'} - c_7j^{\eta}$.
     \end{enumerate}
     We note that all these bounds are increasing in $j$ and hence in $L$. This proves the claim.
 \end{proof}

 The last lemma is crucial to compare the separation event under $Q^K_{U_1, U_2}$ and under $\mathbb{P}^K_{U_1, U_2}$. In order to do so we prove a slight modification of a result given in \cite[Proposition~4.6]{QuenchedBiasedRWRC}. We start by defining
 \begin{equation}
     \begin{split}
         \mathcal{F}_{n} &\coloneqq \sigma\left(\{\{X^{(1)}\}_{0}^{T_n},\{X^{(2)}\}_{0}^{T_n}\}\cup \left\{ c^*(e) \colon e \in \mathcal{E}_{\{X^{(1)}\}_{0}^{T_n}} \cup \mathcal{E}_{\{X^{(2)}\}_{0}^{T_n}} \right\}  \right);
     \end{split}
 \end{equation}
 this is the sigma field generated by the two random walks before hitting a certain distance in space. We are not stressing it in the notation, but we recall that $T_n^{(1)} \neq T_{n}^{(2)}$.
 \begin{lemma}\label{lemma:Technical}
     For any positive function $f$ measurable with respect to $\mathcal{F}_{n}$ and any $U_1, U_2 \in \mathbb{Z}^d$ we have
     \begin{equation*}
         \mathbb{E}^K_{U_1, U_2} \left[f \mathds{1}_{\{\mathrm{Sep(n)}\}}\right] = \mathbb{E}^{Q^K}_{U_1, U_2} \left[f \mathds{1}_{\{\mathrm{Sep(n)}\}}\right].
     \end{equation*}
 \end{lemma}
 \begin{proof}
     Firstly, we consider all finite paths $\{\pi^{(1)}(k)\}_{k = 0}^{m_1}$ with $m_1$ arbitrary but finite, starting from $U_1$ and ending in $\mathcal{H}^+(n)$. More specifically, $\pi^{(1)}(k)$ is constructed such that $\pi^{(1)}(1) = U_1$ and there exists a finite $m_1 > 0$ such that $\pi^{(1)}(m_1) \coloneqq \inf\{k\ge 0\colon \pi^{(1)}(k) \cdot \vec{\ell} > n\}$. We similarly define a corresponding path $\pi^{(2)}$, which we associate to $X^{(2)}$. We compute the probability that $X^{(1)}$ and $X^{(2)}$ follow such paths.
     
     Note that we can restrict to those paths by the fact that the random walk is almost surely transient in the direction $\vec{\ell}$. Furthermore, $\pi^{(1)}$ and $\pi^{(2)}$ are chosen so that none of the vertices along them is at a distance less than $3$ from any of the vertices in the other path.
     
     We restrict to all functions
     \[f = \mathds{1}_{\{X^{(1)}\}_{0}^{m_1} = \{\pi^{(1)}\}_{0}^{m_1}} g_1 \mathds{1}_{\{X^{(2)}\}_{0}^{m_2} = \{\pi^{(2)}\}_{0}^{m_2}} g_2 \]
     where $g_1, g_2$ are measurable with respect to
     \[\sigma\left(\left\{ c^*(e) \colon e \in \mathcal{E}_{\{\pi^{(1)}\}_{0}^{m_1}} \right\}  \right) \text{ and } \sigma\left(\left\{ c^*(e) \colon e \in \mathcal{E}_{\{\pi^{(2)}\}_{0}^{m_2}} \right\}  \right),\]
     respectively. Since $\mathcal{E}_{\{\pi^{(1)}\}_{0}^{m_1}} \cap \mathcal{E}_{\{\pi^{(2)}\}_{0}^{m_2}} = \emptyset$ and using the independence of the conductances we can write
     \begin{align*}
         \mathbb{E}^K_{U_1, U_2}\left[ f \right] &= \mathbf{E}\left[E^\omega_{U_1}\left[\mathds{1}_{\{X^{(1)}\}_{0}^{m_1} = \{\pi^{(1)}\}_{0}^{m_1}} g_1\right] E^\omega_{U_2}\left[\mathds{1}_{\{X^{(2)}\}_{0}^{m_2} = \{\pi^{(2)}\}_{0}^{m_2}} g_2\right]\right]\\
         &= \mathbf{E}\left[E^\omega_{U_1}\left[\mathds{1}_{\{X^{(1)}\}_{0}^{m_1} = \{\pi^{(1)}\}_{0}^{m_1}} g_1\right] \right] \mathbf{E}\left[E^\omega_{U_2}\left[\mathds{1}_{\{X^{(2)}\}_{0}^{m_2} = \{\pi^{(2)}\}_{0}^{m_2}} g_2\right]\right]\\
         &= \mathbb{E}^{Q^K}_{U_1, U_2} \left[f\right].
     \end{align*}
     This is enough to conclude since, in the general case, $f\mathds{1}_{\{\mathrm{Sep(n)}\}}$ can be written as an increasing limit of linear combinations of functions of this type.
 \end{proof}

\begin{lemma}\label{lemma:SeparationSameEnv}
For all $\varepsilon>0$, there exists a constant $c = c(\beta, \vec{\ell}, \lambda, d, K)>0$ such that, uniformly over all $U_1, U_2 \in \mathcal{U}$, we have
     \begin{equation*}
        \mathbb{P}^K_{U_1, U_2}\Bigg( \text{for all } L \in [n^{50\varepsilon}, n], \inf_{\substack{k_1 \colon |X^{(1)}_{k_1} \cdot \vec{\ell} - L|\le 1 \\ k_2 \colon |X^{(2)}_{k_2} \cdot \vec{\ell} - L|\le 1}} \left| \langle (X^{(1)}_{k_1} - X^{(2)}_{k_2}), \vec{u}\rangle \right| \ge n^{10\varepsilon} \,\, \Big| \,\,\mathcal{D}^{\bullet}=+\infty\Bigg) \ge \frac{c}{n^{1/2 + \varepsilon}}
     \end{equation*}
 \end{lemma}

\begin{proof}[Proof of Lemma~\ref{lemma:SeparationSameEnv}]
    We will need to consider two cases. The first and easier one is when $U_1, U_2 \in \mathcal{U}$ are such that $\|(U_1 - U_2)\cdot \vec{u}\| \ge \Lambda$. Then the result follows from an almost straightforward application of Lemmas~\ref{lemma:SeparationEventually} and~\ref{lemma:Technical}. The only slight issue is associated with the conditioning. To circumvent this we recall the definition of $M$, given in the appendix at \eqref{definition_M_other}, and note that
    \begin{equation*}
        \{M \ge n\} \cap \mathrm{Sep}(n) \in \mathcal{F}_n.
    \end{equation*}
    Hence, observing that $\{D^\otimes = \infty\} \cap \mathrm{Sep}(n) \subset \{M \ge n\} \cap \mathrm{Sep}(n)$, by Lemma~\ref{lemma:Technical}
    \begin{align*}
         \mathbb{P}^K_{U_1, U_2}\Bigg( &\text{for all } L \in [n^{40\varepsilon}, n], \inf_{\substack{k_1 \colon |X^{(1)}_{k_1} \cdot \vec{v} - L|\le 1 \\ k_2 \colon |X^{(2)}_{k_2} \cdot \vec{v} - L|\le 1}} \left| \langle (X^{(1)}_{k_1} - X^{(2)}_{k_2}), \vec{u}\rangle \right| \ge n^{10\varepsilon} , M \ge n, \mathrm{Sep}(n) \Bigg) \\
         &= Q^K_{U_1, U_2}\Bigg( \text{for all } L \in [n^{40\varepsilon}, n], \inf_{\substack{k_1 \colon |X^{(1)}_{k_1} \cdot \vec{v} - L|\le 1 \\ k_2 \colon |X^{(2)}_{k_2} \cdot \vec{v} - L|\le 1}} \left| \langle (X^{(1)}_{k_1} - X^{(2)}_{k_2}), \vec{u}\rangle \right| \ge n^{10\varepsilon} , \mathrm{Sep}(n), M \ge n \Bigg) \\
         & \ge \frac{c}{n^{1/2+\varepsilon}},
     \end{align*}
     as $Q^K_{U_1, U_2}(D^\otimes = \infty) \ge \eta > 0$. However, we observe that by \cite[Lemma~3.19]{QuenchedBiasedRWRC} we have
     \begin{equation*}
         \mathbb{P}^K_{U_1, U_2}\Bigg( M \ge n, \mathcal{D}^\bullet < \infty \Bigg) \le Ce^{-cn}.
     \end{equation*}
     We also recall the fact stated in Theorem~\ref{theo:FLSRegStruct} that $\mathbb{P}^K_{U_1, U_2}(\mathcal{D}^\bullet = \infty) \ge \eta > 0$. To conclude one can apply the standard manipulation $P(A | B) \ge P(A , B) \ge P(A) - P(A \cap B^c)$ for any two events $A, B$.

    Now we restrict to the case where $U_1, U_2 \in \mathcal{U}$ satisfy $\|(U_1 - U_2)\cdot \vec{u}\| \le \Lambda$. Then we can proceed by defining two paths $\pi^{(1)}$ and $\pi^{(2)}$ that start from $U_1$ and $U_2$ respectively and get at distance $\Lambda$, stay compatible with the event $\{\mathcal{D}^\bullet = \infty\}$ and end at a potential joint regeneration level on two points $V_1, V_2$. Furthermore, we ask all the vertices along both paths to be $K$-open. We note that
    \begin{align*}
        \mathbb{P}^K_{U_1, U_2}\Bigg( \text{for all } L &\in [n^{50\varepsilon}, n], \inf_{\substack{k_1 \colon |X^{(1)}_{k_1} \cdot \vec{\ell} - L|\le 1 \\ k_2 \colon |X^{(2)}_{k_2} \cdot \vec{\ell} - L|\le 1}} \left| \langle (X^{(1)}_{k_1} - X^{(2)}_{k_2}), \vec{u}\rangle \right| \ge n^{10\varepsilon} \,\, \Big| \,\,\mathcal{D}^{\bullet}=+\infty\Bigg) \\
        &\ge \mathbb{P}^K_{U_1, U_2}\left( \{X^{(1)}\} = \pi^{(1)}, \{X^{(2)}\} = \pi^{(2)}, \theta_{T^{(1)}_{V_1}, T^{(2)}_{V_2}} \circ \mathcal{D}^\bullet = \infty \right) 
        \\ \times &\mathbb{P}^K_{V_1, V_2}\Bigg( \text{for all } L \in [n^{40\varepsilon}, n], \inf_{\substack{k_1 \colon |X^{(1)}_{k_1} \cdot \vec{\ell} - L|\le 1 \\ k_2 \colon |X^{(2)}_{k_2} \cdot \vec{\ell} - L|\le 1}} \left| \langle (X^{(1)}_{k_1} - X^{(2)}_{k_2}), \vec{u}\rangle \right| \ge n^{10\varepsilon} \,\, \Big| \,\,\mathcal{D}^{\bullet}=+\infty\Bigg)\\
        &\ge c(K, \Lambda) \frac{c}{n^{1/2+\varepsilon}}.
     \end{align*}
     Note that we have used that $\Lambda>0$ is a finite constant which implies that
     \begin{equation*}
         \mathbb{P}^K_{U_1, U_2}\left( \{X^{(1)}\} = \pi^{(1)}, \{X^{(2)}\} = \pi^{(2)}, \theta_{T^{(1)}_{V_1}, T^{(2)}_{V_2}} \circ \mathcal{D}^\bullet = \infty \right) \ge c(K, \Lambda)>0. 
     \end{equation*}
     Figure~\ref{fig:FigLambda} clarifies the construction of the paths $\pi^{(1)}$ and $\pi^{(2)}$ and the events involved.
     \begin{figure}[H]
    \centering
    \includegraphics[width=0.7\linewidth]{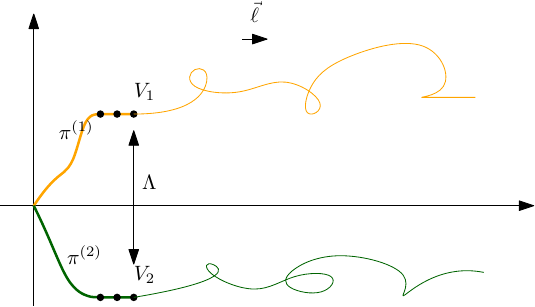}
    \caption{The thick lines are the two paths $\pi^{(1)}$ and $\pi^{(2)}$, note that we ask all the points along these paths to be $K$-open and ask the random walks to follow these (finite) paths. Then upon reaching $V_1$ and $V_2$ a joint regeneration happens with positive probability.}
    \label{fig:FigLambda}
\end{figure}
     Moreover, by proceeding as we did at the beginning of the proof
     \begin{align*}
         \mathbb{P}^K_{V_1, V_2}\Bigg( &\text{for all } L \in [n^{40\varepsilon}, n], \inf_{\substack{k_1 \colon |X^{(1)}_{k_1} \cdot \vec{v} - L|\le 1 \\ k_2 \colon |X^{(2)}_{k_2} \cdot \vec{v} - L|\le 1}} \left| \langle (X^{(1)}_{k_1} - X^{(2)}_{k_2}), \vec{u}\rangle \right| \ge n^{10\varepsilon} \,\, \Big| \,\,\mathcal{D}^{\bullet}=+\infty\Bigg) \\
         &\ge c Q^K_{V_1, V_2}\Bigg( \text{for all } L \in [n^{40\varepsilon}, n], \inf_{\substack{k_1 \colon |X^{(1)}_{k_1} \cdot \vec{v} - L|\le 1 \\ k_2 \colon |X^{(2)}_{k_2} \cdot \vec{v} - L|\le 1}} \left| \langle (X^{(1)}_{k_1} - X^{(2)}_{k_2}), \vec{u}\rangle \right| \ge n^{10\varepsilon}  \text{ and }\mathrm{Sep}(n) \,\, \Big| \,\,D^\otimes =+\infty\Bigg)\\
         &\ge \frac{c}{n^{1/2+\varepsilon}},
     \end{align*}
     completing the proof.
\end{proof}

\begin{proof}[Proof of Proposition~\ref{prop:CrossingJRL}]
To prove Proposition~\ref{prop:CrossingJRL} we argue along the lines of \cite{Berger_Zeitouni}. Let us first define some useful quantities. Firstly, we set 
\begin{equation*}
    M_1 \coloneqq \inf\Bigg\{L \ge n^{50\varepsilon} \colon \inf_{\substack{k_1 \colon |X^{(1)}_{k_1} \cdot \vec{\ell} - L|\le 1 \\ k_2 \colon |X^{(2)}_{k_2} \cdot \vec{\ell} - L|\le 1}} \left| \langle (X^{(1)}_{k_1} - X^{(2)}_{k_2}), \vec{u}\rangle \right| \le n^{10\varepsilon} \Bigg\}.
\end{equation*}
Let us also define a sequences of random variables $\{\psi_{n}\}_{n \in \mathbb{N}}$ and $\{\vartheta_{n}\}_{n \in \mathbb{N}}$ inductively. Let us fix $\vartheta_0 \coloneqq  0$ and $\psi_1 \coloneqq  \mathcal{L}_1$. Then, for $k \ge 1$
\begin{equation}\label{eqn:SeparatingJRL}
    \vartheta_k \coloneqq \theta_{T^{(1)}_{\psi_k}, T^{(2)}_{\psi_k}} \circ M_1, \quad\quad \text{and} \quad\quad \psi_{k+1} \coloneqq \inf \left\{ \mathcal{L}_{j} \colon \mathcal{L}_{j}> \vartheta_k \right\}.
\end{equation}
Furthermore, we define $J_k \coloneqq \vartheta_k - \psi_k$ and set 
\begin{equation}\label{eqn:NumberOfAttemptsJRL}
    \text{Cross}(n) \coloneqq \min\left\{ m \in \mathbb{N} \colon \sum_{i = 1}^m j_i > n \right\}.
\end{equation}
It follows from the definitions of these quantities that 
\begin{equation*}
    \mathrm{JRL}^\le(n, \varepsilon) \subseteq \mathrm{JRL}^{\mathrm{app}} \coloneqq \left\{ k \colon \mathcal{L}_{k} \in \bigcup_{i =1}^\infty [\vartheta_i, \psi_{i+1} + n^{50\varepsilon}], \mathcal{L}_{k} \le n \right\}.
\end{equation*}
Furthermore, on the events defined in \eqref{eqn:UniformBoundBoxes} and \eqref{eqn:UniformBoundJointSlabs}, we have that 
\begin{equation*}
    |\mathrm{JRL}^{\mathrm{app}}| \le n^{60 \varepsilon} |\{k \colon \vartheta_k \le n\}| \le n^{60 \varepsilon}\text{Cross}(n).
\end{equation*}
An immediate consequence of this fact and of the fact that $\mathbb{P}(F_{n, \varepsilon} \cap L_{n, \varepsilon}) \ge 1 - Cn^{-M}$ is that
\begin{equation*}
    \mathbb{E}_{0, 0}\left[|\mathrm{JRL}^\le(n, \varepsilon)| \right] \le 4n^{60\varepsilon} \mathbb{E}_{0, 0}\left[ \text{Cross}(n) \right] + Cn^{-M} n^{d\alpha} \le 8n^{60\varepsilon} \mathbb{E}_{0, 0}\left[ \text{Cross}(n) \right],
\end{equation*}
as we can choose $K>0$ such that $M > d \alpha$ making the second term negligible.

Let us momentarily claim the following fact
\begin{equation}\label{eqn:ProbFewCrossings}
    \mathbb{P}_{0, 0}\left( \text{Cross}(n) > s \right) \le \exp\left( - c \frac{s}{n^{1/2 + \varepsilon}}\right),
\end{equation}
for some positive constant $c$ depending only on $K, \vec{\ell}, \lambda, d$. 
Assuming \eqref{eqn:ProbFewCrossings} and choosing $s = n^{1/2+2\varepsilon}$, we obtain
\begin{align*}
    \mathbb{E}_{0, 0}\left[ |\mathrm{JRL}^\le(n, \varepsilon)| \right] &\le 8n^{60\varepsilon} \left(\mathbb{E}_{0, 0}\left[ \text{Cross}(n) \mathds{1}_{\{\text{Cross}(n) > n^{1/2+2\varepsilon}\}} \right] + n^{1/2+2\varepsilon} \right)\\
    &\le 8n^{60\varepsilon} \left( n\mathbb{P}_{0, 0}\left(\text{Cross}(n)>n^{1/2+2\varepsilon} \right) + n^{1/2+2\varepsilon} \right)\\
    &\le 8n^{60\varepsilon} \left( n e^{-n^\varepsilon} + n^{1/2+2\varepsilon} \right) \le 16 n^{1/2 + 70\varepsilon},
\end{align*}
which concludes the proof.

We are left to show that $\eqref{eqn:ProbFewCrossings}$ holds. Suppose that there exists some $C>0$ such that for every $m \in \mathbb{N}$ 
\begin{equation}\label{eqn:LBjs}
    \mathbb{P}_{0, 0} \left( j_m > n | j_1 \le n, \dots, j_{m-1}\le n \right) \ge \frac{C}{n^{1/2 + \varepsilon}}.
\end{equation}
Then we readily see that 
\begin{align*}
    \mathbb{P}_{0, 0}\left( \text{Cross}(n) > s \right)\leq \mathbb{P}_{0, 0}\big(\sum_{i=1}^{m}j_i\leq n \text{ }\forall m\leq s\big)
    &\leq \mathbb{P}_{0, 0}\big(\sum_{i=1}^{s}j_i\leq n \big)\\
    &\leq \mathbb{P}_{0, 0}(j_i\leq n \text{ }\forall i\leq s)\\
    &=\prod_{i=1}^s\mathbb{P}_{0, 0}(j_i\leq n|j_1,\dots,j_{i-1}\leq n) \\
    &=\prod_{i=1}^s[1-\mathbb{P}_{0, 0}(j_i> n|j_1,\dots,j_{i-1}\leq n)]\\
    &\leq \big(1-\frac{C}{n^{1/2 + \varepsilon}}\big)^s\leq e^{-C\frac{s}{n^{1/2+\varepsilon}}},
\end{align*}
where for the last inequality we have used the standard bound $1+x\leq e^x$, which is valid for all $x\in \mathbb{R}$.

To show \eqref{eqn:LBjs} and conclude the proof we aim to use Lemma~\ref{lemma:SeparationSameEnv}. We set the notation
\begin{equation}\label{eqn:EventSeparation}
    G(n) \coloneqq \Bigg\{\text{for all } L \in [n^{50\varepsilon}, n], \inf_{\substack{k_1 \colon |X^1_{k_1} \cdot \vec{\ell} - L|\le 1 \\ k_2 \colon |X^2_{k_2} \cdot \vec{\ell} - L|\le 1}} \left| \langle (X^1_{k_1} - X^2_{k_2}), \vec{u}\rangle \right| \ge n^{10\varepsilon} \Bigg\}.
\end{equation}
The other key ingredient is \cite[Theorem~3.25]{QuenchedBiasedRWRC}. Indeed, from there we obtain 
\begin{align*}
    \mathbb{P}_{0, 0}&(j_m > n,j_1 \le n, \dots, j_{m-1}\le n) \\
    &\ge \sum_{U_1,U_2}\mathbb{P}_{0, 0}(j_1 \le n, \dots, j_{m-1}\le n, X^{1}_{T^1_{\psi_m}},X^{2}_{T^2_{\psi_m}}) \mathbb{P}^K_{U_1, U_2}\left( G(n) \big| \,\,\mathcal{D}^{\bullet}=+\infty\right).
\end{align*}
Observing that $U_1-U_2 \in \mathcal{U}$ and applying Lemma~\ref{lemma:SeparationSameEnv} gives the result.
\end{proof}

\begin{proof}[Proof of Proposition~\ref{prop:FewCrossing}]
    The proof is a straightforward consequence of Proposition~\ref{prop:CrossingJRL}. Indeed, define the set of joint regeneration levels 
\begin{equation*}
    \mathrm{JRL}^>(n, \varepsilon) \coloneqq \left\{k \colon \mathcal{L}_k \le n \text{ and } \left\|X^{(1)}_{T^{(1)}_{\mathcal{L}_k}} - X^{(2)}_{T^{(2)}_{\mathcal{L}_k}}\right\|> n^{10\varepsilon}\right\}.
\end{equation*}
Recall the events $F_{n, \varepsilon} \cap L_{n, \varepsilon}$ defined at \eqref{eqn:UniformBoundBoxes}, \eqref{eqn:UniformBoundJointSlabs}. We define the event
\begin{equation}\label{eqn:CrossAfterFarJointReg}
    \mathrm{FC}_{n, \varepsilon}^c \coloneqq \left\{ \exists k \in \mathrm{JRL}^>(n, \varepsilon) \colon  \inf_{\substack{k_1 \in K_1 \\ k_2 \in K_2 }} \left\| X^{(1)}_{k_1} - X^{(2)}_{k_2} \right\| \le 3, K_i = [T^{(i)}_{\mathcal{L}_k}, T^{(i)}_{\mathcal{L}_{k+1}}] \cap \mathbb{N}, i = 1, 2   \right\}.
\end{equation}
One can easily check that $\mathrm{FC}_{n, \varepsilon}^c \subseteq F_{n, \varepsilon}^c \cup L_{n, \varepsilon}^c$ (see Lemma~\ref{lemma:GoodEventsBoxes} for the two events on the right). We also note that on the event $\mathrm{FC}_{n, \varepsilon} \cap F_{n, \varepsilon} \cap L_{n, \varepsilon}$ there cannot be a crossing in the regeneration slabs associated to $\mathrm{JRL}^>(n, \varepsilon)$. Indeed, the two walks start at a distance $n^{10\varepsilon}$ and can move at most $n^{\alpha \varepsilon}$ (recall that we only have the restriction $\alpha > d+3$).

Hence,
\begin{equation*}
    \mathbf{E}\left[ E^\omega_{0, 0}\left[ I_n \right] \right] \le \mathbf{E}\left[ E^\omega_{0, 0}\left[ I_n \mathds{1}_{\{\mathrm{FC}_{n, \varepsilon} \cap F_{n, \varepsilon} \cap L_{n, \varepsilon}\}}\right]\right] + \mathbf{E}\left[ E^\omega_{0, 0}\left[ I_n \mathds{1}_{\{\mathrm{FC}_{n, \varepsilon}\}^c}\right] \right] \le Cn^{1/2 + c\varepsilon} + Cn^{d\alpha - M},
\end{equation*}
Finally, we notice that we can first fix $\varepsilon>0$ such that $c\varepsilon<\delta$ and then $K_0$ such that for all $K \ge K_0$ $d\alpha - M$ is smaller than $1/2$.
\end{proof}

\section{Martingale difference and variance bound for the position}\label{sect:MartDiffPos}

The goal of this section is to prove \eqref{EquationBoundVarianceTraj}. To this end, we will employ a strategy that draws inspiration from the works \cite{Berger_Zeitouni, RassoulAghaSeppalainen}.

Recall that $T_{R} = \inf\{ n \ge 0: X_{n} \cdot \vec{\ell} \ge R\}.$ For two paths $\pi_1$ and $\pi_2$ of length $m$ and $k$ respectively, let $\pi_1 \circ \pi_2$ denote their concatenation, i.e.\ for $ 0 \le s \le m+k$ we set
\begin{equation}\label{eqn:ConcatenationPaths}
    \pi_1 \circ \pi_2(s) \coloneqq 
    \begin{cases}
    \pi_1(s) \quad &\text{if } s \le m,\\
    \pi_1(m) + \pi_{2}(s - m) \quad &\text{if } s \in \llbracket m, k+m \rrbracket.
    \end{cases}
\end{equation}
Let $\{X\}_{a}^{b}$ denote the trajectory of the random walk between times $a$ and $b$, for non-negative integers $a\leq b$.

\subsection{Martingale difference}

\begin{proof}[Proof of \eqref{EquationBoundVarianceTraj}] 
Throughout the proof we will assume, without loss of generality, that $F_2$ in the statement is bounded by $1$ and has Lipschitz constant $1$.

\vspace{0.3cm}

\noindent \textbf{Step 1: Reduction to a finite box.}
Let $E_n$ be a (good) event such that $\mathbb{P}_{0}\left( E_n \right) \ge 1 - Cn^{-c}$ for two constants $C, c>0$ and all $n \in \mathbb{N}$. Then, using Jensen's inequality, we readily see that for any such event
\begin{align*}
    \mathbf{Var}\left(E_{0}^{\omega}[F_2(Z_{n})]\right) &= \mathbf{E}\left[E_{0}^{\omega}[F_2(Z_{n})]^2\right] -  \mathbf{E}\left[E_{0}^{\omega}[F_2(Z_{n})]\right]^2\\
    &\le \mathbf{E}\left[E_{0}^{\omega}[F_2(Z_{n})\mathds{1}_{\{E_n\}}]^2\right] -  \mathbf{E}\left[E_{0}^{\omega}[F_2(Z_{n})\mathds{1}_{\{E_n\}}]\right]^2 \\
    &+ 2\mathbf{E}\left[E_{0}^{\omega}[F_2(Z_{n}) \mathds{1}_{\{E_n^c\}}]^2\right] + 2  \mathbf{E}\left[E_{0}^{\omega}[F_2(Z_{n})\mathds{1}_{\{E^c_n\}}] E_{0}^{\omega}[F_2(Z_{n})\mathds{1}_{\{E_n\}}]\right] \\
&\hspace{1cm}+2\mathbf{E}\left[E_{0}^{\omega}[F_2(Z_{n})\mathds{1}_{\{E_n\}}]\right]\mathbf{E}\left[E_{0}^{\omega}[F_2(Z_{n})\mathds{1}_{\{E^c_n\}}]\right]\\
    & \le \mathbf{E}\left[E_{0}^{\omega}[F_2(Z_{n}) \mathds{1}_{\{E_n\}}]^2\right] -  \mathbf{E}\left[E_{0}^{\omega}[F_2(Z_{n})\mathds{1}_{\{E_n\}}]\right]^2 + 5Cn^{-c}\\
    & = \mathbf{Var}\left(E_{0}^{\omega}[\widetilde{F}_2(Z_{n})]\right) + 5Cn^{-c},
\end{align*}
where we set $\widetilde{F}_2(Z_{n}) \coloneqq F_2(Z_{n}) \mathds{1}_{\{E_n\}}$. From now on, we will talk about the `variance of $E_{0}^{\omega}[F_2(Z_{n})]$ on $E_n$' referring to the variance of $E_{0}^{\omega}[\widetilde{F}_2(Z_{n})]$; this is justified by the above estimate.

Define $\tilde{c} \coloneqq \vec{v} \cdot \vec{\ell}$. We introduce the events 
\begin{equation}\label{eqn:NotTooFar}
    K_{n} \coloneqq \{X_{\tau_{n}} \cdot \vec{\ell} \le 3 \tilde{c}n\}  \text{ and } J_n \coloneqq \{ T_{\partial \mathcal{B}(3 \tilde{c} n, (3 \tilde{c} n)^\alpha)} \ge T_{\partial^+ \mathcal{B}(3 \tilde{c} n, (3 \tilde{c} n)^\alpha)} \}.
\end{equation}
First of all, using the decomposition
\[(X_{\tau_n} -X_{\tau_{1}}) \cdot \vec{\ell}= \sum_{i = 2}^n (X_{\tau_i} -X_{\tau_{i-1}}) \cdot \vec{\ell}\]
we split
\begin{align*}
    \mathbb{P}_{0}\left(K_{n}^c \right) 
    &\le \mathbb{P}_{0}\left(X_{\tau_{1}} \cdot \vec{\ell} > \tilde{c}n \right) + \mathbb{P}_{0}\left((X_{\tau_n} -X_{\tau_{1}}) \cdot \vec{\ell} > 2\tilde{c}n \right) \\
    &\le C n^{-M} + \mathbb{P}_{0}\left(\sum_{i = 2}^n (X_{\tau_i} -X_{\tau_{i-1}}) \cdot \vec{\ell} > 2\tilde{c}n \right),
\end{align*}
where we applied Theorem~\ref{theo:FKRegStruct} to bound the probability that $X_{\tau_{1}} \cdot \vec{\ell}$ is greater than $\tilde{c}n$. For the second probability above, note that the $(X_{\tau_i} -X_{\tau_{i-1}}) \cdot \vec{\ell}$ are i.i.d.\ random variables with finite $M$-th moment, for arbitrary $M>0$. 
Furthermore, by definition, their common mean is $\tilde{c}$. 

In order to simplify the notation, let us denote $Y_i \coloneqq (X_{\tau_i} -X_{\tau_{i-1}}) \cdot \vec{\ell}$ so that $\mathbb{E}_{0}[Y_i] = \tilde{c}$. Hence we see that
\begin{align*}
    \mathbb{P}_{0}\left(\sum_{i = 2}^n (X_{\tau_i} -X_{\tau_{i-1}}) \cdot \vec{\ell} > 2\tilde{c}n \right) 
    \le \mathbb{P}_{0}\left(\sum_{i = 2}^n \left(Y_i - \mathbb{E}_{0}[Y_1] \right) > \tilde{c} n \right) \le Cn^{-M},
\end{align*}
where we applied Lemma~\ref{lemma:PolynomialBound} of the Appendix. All in all, we arrive at $\mathbb{P}_0(K^c_n)\leq Cn^{-M}$. Moreover, from \cite[Theorem~3.1]{Kious_Frib} we obtain $\mathbb{P}_{0}(J_n^c)\le Ce^{- c n}$ and hence, defining $E_{n, \varepsilon} \coloneqq F_{n, \varepsilon} \cap K_n \cap J_n$, it follows from (\ref{eqn:FNepsi}) together with the last two estimates that $\mathbb{P}_0((E_{n, \varepsilon})^c) \le C n^{-M}$. 

We highlight that, on the event $E_{n, \varepsilon}$, the trajectory of the random walk $\{X_k\}_{k = 1}^{\tau_n}$ depends only on the conductances inside the tilted box $\mathcal{B}(3 \tilde{c} n, (3 \tilde{c} n)^\alpha)$.

\vspace{0.3cm}

\noindent \textbf{Step 2: Martingale difference.}
Let us consider the box $\mathcal{B}(c_1 n, (c_1 n)^\alpha)$ for $c_1 \coloneqq 3 \tilde{c}$ and $n \in \mathbb{N}$. Let $\{\mathrm{ord}(z): z \in \mathcal{B}(c_1 n, (c_1 n)^\alpha)\}$ denote an arbitrary, deterministic ordering of the vertices in $\mathbb{Z}^d \cap \mathcal{B}(c_1 n, (c_1 n)^\alpha)$ such that, for all $x, y \in \mathcal{B}(c_1 n, (c_1 n)^\alpha)$, we have $\mathrm{ord}(z) \ge \mathrm{ord}(y)$ whenever $z \cdot \vec{\ell} \ge y \cdot \vec{\ell} $.
We write $p(z)$ for the \textit{predecessor} of vertex $z$ in such ordering. 

Denote by $\mathcal{G}_{z}$ the $\sigma$-algebra generated by the conductances on those edges which are incident to the sites $x\in \mathbb{Z}^d$ which comes before $z$ in our order; formally, we let
\begin{equation}\label{eqn:SigmaAlgebraRevealment}
    \mathcal{G}_{z} \coloneqq \sigma\left(\left\{ c^*(e), e \in E(\mathbb{Z}^d) \colon e \sim x \text{ for some } x \in \mathbb{Z}^d \text{ with } \mathrm{ord}(x)\le \mathrm{ord}(z)\right\}\right).
\end{equation}
To highlight the fact that we will be focusing on the trajectory of the random walk between times $0$ and $\tau_n$, we write $F^n(\{X\}_{0}^{\tau_n})$ for the function $\tilde{F}_2(Z_{n})$. Then we can write, on the event $E_{n, \varepsilon}$,
\begin{align}
    \nonumber E_0^\omega\left[ F^n(\{X\}_{0}^{\tau_n}) \right] &- \mathbf{E}\left[ E_0^\omega\left[ F^n(\{X\}_{0}^{\tau_n}) \right] \right]\\
    &= \sum_{z \in \mathcal{B}_{0}(c_1 n, (c_1 n)^\alpha)} \mathbf{E}\left[ E_0^\omega\left[ F^n(\{X\}_{0}^{\tau_n}) \right] | \mathcal{G}_{z} \right] - \mathbf{E}\left[ E_0^\omega\left[ F^n(\{X\}_{0}^{\tau_n}) \right] | \mathcal{G}_{p(z)} \right] \nonumber \\
    & = \sum_{k = -nc_1}^{nc_1} \sum_{z \in \hat{H}_k} \Delta_{z}^n, \label{eqn:SplitMartingale}
\end{align}
where we denote by $\hat{H}_k$ the set of sites whose inner product with $\vec{\ell}$ lies in $[k,k+1)$, i.e.\ $\hat{H}_k \coloneqq \{z \in \mathbb{Z}^d \colon z \cdot \vec{\ell} \in [k, k+1) \}$, whereas 
\begin{equation*}
    \Delta_{z}^n \coloneqq \mathbf{E}\left[ E_0^\omega\left[ F^n(\{X\}_{0}^{\tau_n}) \right] | \mathcal{G}_{z} \right] - \mathbf{E}\left[ E_0^\omega\left[ F^n(\{X\}_{0}^{\tau_n}) \right] | \mathcal{G}_{p(z)} \right].
\end{equation*}
Since martingale increments are uncorrelated, we readily obtain the identity
\begin{equation}\label{eqn:BoundVarIntersection}
    \textbf{Var}\left( E_0^\omega\left[ F^n(\{X\}_{0}^{\tau_n}) \right] \right) = \sum_{k = -nc_1}^{nc_1} \sum_{z \in \hat{H}_k} \mathbf{E}\left[ \left(\Delta_{z}^n \right)^2\right].
\end{equation}

\vspace{3px}

\vspace{0.3cm}

\noindent \textbf{Step 3: The approximating quantity and inserting the good event.}
We aim to bound the quantity $\mathbf{E}\left[ \left(\Delta_{z}^n\right)^2\right]$ on the right-hand side of (\ref{eqn:BoundVarIntersection}). To this end, we start by splitting 
\begin{equation*}
    \mathbf{E}\left[ E_0^\omega\left[ F^n(\{X\}_{0}^{\tau_n}) \right] | \mathcal{G}_{z} \right] = \mathbf{E}\left[ E_0^\omega\left[ F^n(\{X\}_{0}^{\tau_n}) \mathds{1}_{\{ X \text{ visits } \mathcal{V}_z \}} \right] | \mathcal{G}_{z} \right] + \mathbf{E}\left[ E_0^\omega\left[ F^n(\{X\}_{0}^{\tau_n}) \mathds{1}_{\{ X \text{ visits } \mathcal{V}_z \}^c} \right] | \mathcal{G}_{z} \right].
\end{equation*}
Note that for any function $G :(\mathbb{Z}^d)^{\mathbb{N}} \to \mathbb{R}_+$, we have 
\begin{align*}
     \mathbf{E}\left[ E_0^\omega\left[ G(\{X\}_{\ge 0}) \mathds{1}_{\{ X \text{ visits } \mathcal{V}_z \}^c} \right] | \mathcal{G}_{z} \right] = \mathbf{E}\left[ E_0^\omega\left[ G(\{X\}_{\ge 0}) \mathds{1}_{\{ X \text{ visits } \mathcal{V}_z \}^c} \right] | \mathcal{G}_{p(z)} \right].
\end{align*}
In particular,
\begin{equation*}
    \mathbf{E}\left[ E_0^\omega\left[ F^n(\{X\}_{0}^{\tau_n}) \mathds{1}_{\{ X \text{ visits } \mathcal{V}_z \}^c} \right] | \mathcal{G}_{z} \right] = \mathbf{E}\left[ E_0^\omega\left[ F^n(\{X\}_{0}^{\tau_n}) \mathds{1}_{\{ X \text{ visits } \mathcal{V}_z \}^c} \right] | \mathcal{G}_{p(z)} \right].
\end{equation*}
We deduce that
\begin{align*}
\Delta_{z}^n= 
    &\mathbf{E}\left[ E_0^\omega\left[ F^n(\{X\}_{0}^{\tau_n}) \mathds{1}_{\{ X \text{ visits } \mathcal{V}_z \}} \right] | \mathcal{G}_{z} \right] - \mathbf{E}\left[ E_0^\omega\left[ F^n(\{X\}_{0}^{\tau_n}) \mathds{1}_{\{ X \text{ visits } \mathcal{V}_z \}} \right] | \mathcal{G}_{p(z)} \right].
\end{align*}
From now on, with the purpose of simplifying notation let us write $\mathbf{E}_z\left[ \cdot  \right]$ for $\mathbf{E}\left[ \cdot \, |\mathcal{G}_{z} \right]$ (and similarly for $p(z)$).

Consider any bounded function $\hat{F}$ of the trajectory such that
\begin{align*}
    \mathbf{E}_z\left[ E^\omega_0[\hat{F} \mathds{1}_{\{ X \text{ visits } \mathcal{V}_z \}}] \right] = \mathbf{E}_{p(z)}\left[ E^\omega_0[\hat{F}\mathds{1}_{\{ X \text{ visits } \mathcal{V}_z \}}] \right].
\end{align*}
Then, using the identity $(a+b)^2 \le 2(a^2 + b^2)$ and the boundedness of all the function considered 
\begin{align*}
    \mathbf{E}\left[ \left(\Delta_{z}^n\right)^2\right] &= \mathbf{E}\Bigg[ \bigg( \mathbf{E}_z\left[ E^\omega_0[F^n(\{X\}_{0}^{\tau_n}) \mathds{1}_{\{ X \text{ visits } \mathcal{V}_z \}}] \right] - \mathbf{E}_z\left[ E^\omega_0[\hat{F}\mathds{1}_{\{ X \text{ visits } \mathcal{V}_z \}}] \right] \\
    & + \mathbf{E}_{p(z)}\left[ E^\omega_0[\hat{F}\mathds{1}_{\{ X \text{ visits } \mathcal{V}_z \}}] \right] - \mathbf{E}_{p(z)}\left[ E^\omega_0[F^n(\{X\}_{0}^{\tau_n}) \mathds{1}_{\{ X \text{ visits } \mathcal{V}_z \}}] \right]\bigg)^2\Bigg] \\
    & \le 2 \mathbf{E}\Bigg[ \bigg( \mathbf{E}_z\left[ E^\omega_0[F^n(\{X\}_{0}^{\tau_n}) \mathds{1}_{\{ X \text{ visits } \mathcal{V}_z \}}] \right] - \mathbf{E}_z\left[ E^\omega_0[\hat{F} \mathds{1}_{\{ X \text{ visits } \mathcal{V}_z \}}\mathds{1}_{E_{n, \varepsilon}}] \right] \\
    & + \mathbf{E}_{p(z)}\left[ E^\omega_0[\hat{F}\mathds{1}_{\{ X \text{ visits } \mathcal{V}_z \}} \mathds{1}_{E_{n, \varepsilon}}] \right] - \mathbf{E}_{p(z)}\left[ E^\omega_0[F^n(\{X\}_{0}^{\tau_n}) \mathds{1}_{\{ X \text{ visits } \mathcal{V}_z \}}] \right]\bigg)^2\Bigg]\\
    & + 8 \mathbf{E}\left[ E^\omega_0[\mathds{1}_{E_{n, \varepsilon}^c}]^2\right].
\end{align*}
Finally notice that 
\begin{align*}
    \sum_{z \in \mathcal{B}_{0}(c_1 n, (c_1 n)^\alpha)} \mathbf{E}\left[ E^\omega_0[\mathds{1}_{E_{n, \varepsilon}^c}]^2\right] \le c_1^{\alpha d}n^{2d \alpha} 4 \mathbb{P}_0\left( E_{n, \varepsilon}^c \right)\le C n^{-M/2},
\end{align*}
if we choose $M$ large enough. We use the computation above to insert the good event in our approximating quantity. The minor issue here is that $\mathds{1}_{E_{n, \varepsilon}}$ does depend on the conductances in $\mathcal{E}_z$, hence  $\mathbf{E}_z\left[ E^\omega_0[\mathds{1}_{E_{n, \varepsilon}}] \right] \neq \mathbf{E}_{p(z)}\left[ E^\omega_0[\mathds{1}_{E_{n, \varepsilon}}] \right]$. However, the probability of $E_{n, \varepsilon}$ is sufficiently large that we are able to always work on it.

\vspace{0.3cm}

\noindent \textbf{Step 4: Uniform bound on bad areas.}
We introduce
\begin{itemize}
    \item $\tau_{j^+} \coloneqq \tau^+_z \coloneqq \inf\{ \tau_{k} \colon X_{\tau_{k}} \cdot \vec{\ell} > z \cdot \vec{\ell}\}$. Here $j^+$ is the random index associated with $\tau_{j^+}$.
    \item $\tau^-_z \coloneqq \sup\{ k \le T_{(z-e_1) \cdot \vec{\ell}} \,  \colon D\circ \theta_k + k > T_{(z-e_1) \cdot \vec{\ell})} \}$ with $T_{(z-e_1) \cdot \vec{\ell}} \coloneqq \inf\{ k \in \mathbb{N} \colon X_k \cdot \vec{\ell} \ge (z-e_1) \cdot \vec{\ell}) \}$. Clearly, $T_{(z-e_1) \cdot \vec{\ell}} \le T_{\mathcal{V}_z}$.
\end{itemize}  
In words $\tau^-_z$ is the last time before hitting the set $\mathcal{H}^+_{z-e_1}$, which is compatible to be a regeneration time. Note that, by construction of regeneration times, $\tau^+_z > \tau^-_z$. We also highlight that $\tau^-_z$ is not necessarily a regeneration time while $\tau^+_z$ is. We note that, by almost sure transience in the direction $\vec{\ell}$, it makes sense to consider $\tau^-_z$ to be a finite quantity. See Figure~\ref{fig:FigSplit} for a graphical representation of the definitions given above.

\begin{figure}[H]
    \centering
    \includegraphics[width=0.7\linewidth]{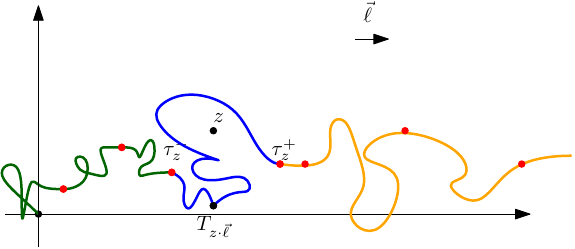}
    \caption{The construction of the approximating quantity. The green path is the one before $\tau_z^-$, the blue one is the one that we erase. The orange path is the one after $\tau_z^+$. In red the candidate regeneration points in the green path and the actual regeneration points in the orange one. The approximating path is the gluing of the green and the orange parts.}
    \label{fig:FigSplit}
\end{figure}

By using the notation specified in \eqref{eqn:ConcatenationPaths}, we split the path of $X$ and rewrite
\begin{equation*}
    F^n(\{X\}_{0}^{\tau_n}) \mathds{1}_{\{ X \text{ visits } \mathcal{V}_z \}} = F^n\left(\{X\}_{0}^{\tau^-_z} \circ \left\{ \{X\}_{\tau^-_z}^{\tau^+_z} - X_{\tau^-_z} \right\} \circ \left\{ \{X\}_{\tau^+_z}^{\tau_{n}} - X_{\tau^+_z} \right\} \right)\mathds{1}_{\{ X \text{ visits } \mathcal{V}_z \}}.
\end{equation*}
Define
\begin{equation*}
    \widetilde{F}^n(\{X\}_{0}^{\tau_n}) \coloneqq F^n\left(\{X\}_{0}^{\tau^-_z} \circ \left\{ \{X\}_{\tau^+_z}^{\tau_{n}} - X_{\tau^+_z} \right\} \right)\mathds{1}_{\{ X \text{ visits } \mathcal{V}_z \}},
\end{equation*}
where by $\{X\}_{0}^{\tau^-_z}$ we mean that we consider, under the relevant measure that generates the trajectory, the linear interpolation at candidate regeneration points. Finally, the index $\tilde{j}^+$ such that $\tau^+_{z} = \tau_{\tilde{j}^+}$ is the number of candidate regeneration points in $\{X\}_{0}^{\tau^-_z}$. Note that $\tilde{j}^+ \neq j^+$; more precisely, we always have $\tilde{j}^+ \geq j^+$. We make a few observations:
\begin{itemize}
    \item $\{X\}_{0}^{\tau^-_z}$ does not depend on $\mathcal{E}_z$, as this quantity is measurable w.r.t.\ the information generated by $\{X\}_{0}^{T_{(z-e_1)\cdot \vec{\ell}}}$.
    \item $\mathds{1}_{\{ X \text{ visits } \mathcal{V}_z \}}$ does not depend on $\mathcal{E}_z$.
    \item $\left\{\{X\}_{\tau^+_z}^{\tau_{n}} - X_{\tau^+_z} \right\}$ also does not depend on $\mathcal{E}_z$. Under both $\mathbf{E}_z\left[P^\omega_0(\cdot)\right]$ and $\mathbf{E}_{p(z)}\left[P^\omega_0(\cdot)\right]$ the increments $\left\{\{X\}_{\tau_k}^{\tau_{k+1}} - X_{\tau_k} \right\}_{k = \tilde{j}^+}^{n}$ are i.i.d.\ random variables distributed as $\left\{\{X\}_{\tau_1}^{\tau_{2}} - X_{\tau_1} \right\}$ under $\mathbb{P}_0$. Observe that $j^+$ does depend on $\mathcal{E}_z$ but $\tilde{j}^+$ does not, that is why we must distinguish between the two quantities.
\end{itemize}
Using all the above observations we obtain
\begin{equation*}
    \mathbf{E}_z\left[E^\omega_0\left[\widetilde{F}^n(\{X\}_{0}^{\tau_n})\right]\right] =\mathbf{E}_{p(z)}\left[E^\omega_0\left[\widetilde{F}^n(\{X\}_{0}^{\tau_n})\right]\right].
\end{equation*}
%\le n^{-1/2}  \mathbf{E}_{p(z)}\left[E^\omega_0\left[\mathds{1}_{\{ X \text{ visits } \mathcal{V}_z \}}\right]\right] \mathbb{E}_0\left[\left\|  \{X\}_{0}^{\tau_{Cn^\varepsilon}} \right\|_{\infty} \mid D = \infty\right],where we used the Lipschitz property of $F_2$ inherited by $\widetilde{F}$.
We also note that under the two measures $\mathbf{E}_z\left[P^\omega_0(\cdot)\right],\mathbf{E}_{p(z)}\left[P^\omega_0(\cdot)\right]$ and on the event $E_{n, \varepsilon}$, the index $j^+ = 1, \dots, n$ such that $\tau^+_z = \tau_{j^+}$ can differ by at most $Cn^\varepsilon$ with respect to the one assigned by our construction $\tilde{j}^+$. Indeed, the trajectory $\{X\}_{0}^{\tau^-_z}$ is the same under the two measures (they can be perfectly coupled). However, since we are working on $E_{n, \varepsilon}$, the walk cannot backtrack more than $n^\varepsilon$ after the time $\tau^-_z$. In a slab in direction $\vec{\ell}$ of width $n^\varepsilon$ there are at most $Cn^\varepsilon$ potential regeneration points that can be ``erased'', hence the index of $\tau^+_z$ might differ by at most $Cn^\varepsilon$, i.e.\ $\tilde{j}^+ \leq j^+ + Cn^{\varepsilon}$. 

Moreover, similarly to our previous observation, under both $\mathbf{E}_z\left[P^\omega_0(\cdot)\right]$ and $\mathbf{E}_{p(z)}\left[P^\omega_0(\cdot)\right]$ the increments $\left\{\{X\}_{\tau_k}^{\tau_{k+1}} - X_{\tau_k} \right\}_{k = j^+}^{n}$ are i.i.d.\ random variables distributed as $\left\{\{X\}_{\tau_1}^{\tau_{2}} - X_{\tau_1} \right\}$ under $\mathbb{P}_0$. Hence we obtain 
\begin{align*}
    \mathbf{E}_{p(z)}&\left[E^\omega_0\left[ \left\| \left\{\{X\}_{\tau_k}^{\tau_{k+1}} - X_{\tau_k} \right\}_{k = j^+}^{n} - \left\{\{X\}_{\tau_k}^{\tau_{k+1}} - X_{\tau_k} \right\}_{k = \tilde{j}^+}^{n} \right\|_\infty \right]\right] \\
    &\le \mathbb{E}_0\left[\sup_{k \in [0, n]}\left\| \{X\}^{\tau_{k + Cn^\varepsilon}}_{\tau_k} - X_{\tau_k} \right\|_{\infty} \mid D = \infty\right],
\end{align*}
and the same is true under $\mathbf{E}_{z}$. Applying \eqref{eqn:FNepsi}, we immediately observe that
\begin{equation}\label{eqn:ExtraTerm}
    \mathbb{E}_0\left[\sup_{k \in [0, n]}\left\|  \{X\}^{\tau_{k + Cn^\varepsilon}}_{\tau_k} - X_{\tau_k} \right\|_{\infty} \mid D = \infty\right] \le Cn^{2d\alpha\varepsilon}.
\end{equation}
On the event $E_{n, \varepsilon}$, we know that the trajectory during a regeneration interval moves at most $n^{\varepsilon\alpha}$ in each direction. Furthermore, between $\tau^-_z$ and $\tau^+_z$ there is at most one regeneration time. Hence, we see that for some $C>0$
\begin{equation}\label{eqn:CoreTerm}
    \left\|  \{X\}_{\tau^-_z}^{\tau^+_z} - X_{\tau^-_z} \right\|_{\infty} \le Cn^{2d\varepsilon\alpha}.
\end{equation}
We now focus on the two terms 
\begin{align*}
    \mathbf{E}_{p(z)}&\left[ E^\omega_0[\widetilde{F}^n(\{X\}_{0}^{\tau_n})\mathds{1}_{\{ X \text{ visits } \mathcal{V}_z \}} \mathds{1}_{E_{n, \varepsilon}}] \right] - \mathbf{E}_{p(z)}\left[ E^\omega_0[F^n(\{X\}_{0}^{\tau_n}) \mathds{1}_{\{ X \text{ visits } \mathcal{V}_z \}}\mathds{1}_{E_{n, \varepsilon}}] \right] \\
    \mathbf{E}_{z}&\left[ E^\omega_0[\widetilde{F}^n(\{X\}_{0}^{\tau_n})\mathds{1}_{\{ X \text{ visits } \mathcal{V}_z \}} \mathds{1}_{E_{n, \varepsilon}}] \right] - \mathbf{E}_{z}\left[ E^\omega_0[F^n(\{X\}_{0}^{\tau_n}) \mathds{1}_{\{ X \text{ visits } \mathcal{V}_z \}}\mathds{1}_{E_{n, \varepsilon}}] \right];
\end{align*}
the argument is symmetric for these two, so we will just write it once. 

By the Lipschitz property of $F_2$ (and hence of $\tilde{F}_2$), recalling the normalization by $n^{-1/2}$ in the definition of $Z_n$ (see \ref{EqnThreeMainQuantities}), and applying \eqref{eqn:ExtraTerm} and \eqref{eqn:CoreTerm} together with the triangle inequality
\begin{equation} \label{eqn:ErasingMiddle}
\begin{split}
    &\Big|  \mathbf{E}_{p(z)}\left[ E^\omega_0[\widetilde{F}^n(\{X\}_{0}^{\tau_n})\mathds{1}_{\{ X \text{ visits } \mathcal{V}_z \}} \mathds{1}_{E_{n, \varepsilon}}] \right] - \mathbf{E}_{p(z)}\left[ E^\omega_0[F^n(\{X\}_{0}^{\tau_n}) \mathds{1}_{\{ X \text{ visits } \mathcal{V}_z \}}\mathds{1}_{E_{n, \varepsilon}}] \right] \Big| \\
    & \le n^{-1/2} \mathbf{E}_{p(z)}\left[E^\omega_0\left[\mathds{1}_{\{ X \text{ visits } \mathcal{V}_z \}} n^{2d\alpha \varepsilon}\right]\right] \\
    &+ n^{-1/2}  \mathbf{E}_{p(z)}\left[E^\omega_0\left[\mathds{1}_{\{ X \text{ visits } \mathcal{V}_z \}}\right]\right] \mathbb{E}_0\left[\sup_{k \in [0, n]}\left\|  \{X\}^{\tau_{k + Cn^\varepsilon}}_{\tau_k} - X_{\tau_k} \right\|_{\infty} \mid D = \infty\right]\\
    & \le n^{-1/2+ 2d\alpha\varepsilon} \mathbf{E}_{p(z)}\left[E^\omega_0\left[\mathds{1}_{\{ X \text{ visits } \mathcal{V}_z \}}\right]\right].
\end{split}
\end{equation}
Using the estimate \eqref{eqn:ErasingMiddle} and Jensen's inequality with some minor manipulation we obtain
\begin{align}
        \nonumber\mathbf{E}&\left[ \left(\Delta_{z}^n\right)^2\right] \\&\le 3\mathbf{E}\left[\left(\mathbf{E}_{p(z)}\left[ E^\omega_0[\widetilde{F}^n(\{X\}_{0}^{\tau_n})\mathds{1}_{\{ X \text{ visits } \mathcal{V}_z \}} \mathds{1}_{E_{n, \varepsilon}}] \right] - \mathbf{E}_{p(z)}\left[ E^\omega_0[F^n(\{X\}_{0}^{\tau_n}) \mathds{1}_{\{ X \text{ visits } \mathcal{V}_z \}}\mathds{1}_{E_{n, \varepsilon}}] \right]\right)^2\right] \nonumber \\
        &+ 3\mathbf{E}\left[\left(\mathbf{E}_{z}\left[ E^\omega_0[\widetilde{F}^n(\{X\}_{0}^{\tau_n})\mathds{1}_{\{ X \text{ visits } \mathcal{V}_z \}} \mathds{1}_{E_{n, \varepsilon}}] \right] - \mathbf{E}_{z}\left[ E^\omega_0[F^n(\{X\}_{0}^{\tau_n}) \mathds{1}_{\{ X \text{ visits } \mathcal{V}_z \}}\mathds{1}_{E_{n, \varepsilon}}] \right]\right)^2\right] \nonumber\\
        & \le Cn^{-1 + 4 d \alpha \varepsilon} \mathbf{E}\left[ E_0^\omega\left[\mathds{1}_{\{ X \text{ visits } \mathcal{V}_z \}} \right]^2 \right].\label{eqn:BoundDeltaz}
\end{align}
\vspace{0.3cm}

\noindent \textbf{Step 5: Reduction to random walks intersections.}
Let $X^{(1)}, X^{(2)}$ be two copies of $X$ evolving in the same environment. Then we can write
\begin{equation*}
    \mathbf{E}\left[ E_0^\omega\left[\mathds{1}_{\{ X \text{ visits } \mathcal{V}_z \}} \right]^2 \right] = \mathbf{E}\left[ E_{0, 0}^\omega\left[\mathds{1}_{\{ X^{(1)} \text{ visits } \mathcal{V}_z \}} \mathds{1}_{\{ X^{(2)} \text{ visits } \mathcal{V}_z \}} \right] \right].
\end{equation*}
Combining \eqref{eqn:BoundVarIntersection} and \eqref{eqn:BoundDeltaz} we obtain
\begin{equation*}
    \textbf{Var}\left( E_0^\omega\left[ F^n(\{X\}_{0}^{\tau_n}) \right] \right) \le Cn^{4d\varepsilon\alpha - 1} \mathbf{E}\left[ \sum_{\substack{z \in \mathcal{B}_{0}(c_1 n, (c_1 n)^\alpha) }} E_{0, 0}^\omega\left[\mathds{1}_{\{ X^1 \text{ visits } \mathcal{V}_z \}} \mathds{1}_{\{ X^2 \text{ visits } \mathcal{V}_z \}} \right] \right].
\end{equation*}
The result follows from Proposition~\ref{prop:FewCrossing} as $\delta$ there and $\varepsilon$ above can be chosen arbitrarily small.  
\end{proof}

\section{The joint process}\label{sect:Joint}
The goal of this section is to prove statement \eqref{EquationBoundVarianceClock}. To do this we will employ a martingale difference argument as in the previous section. 
\begin{proof}[Proof of \eqref{EquationBoundVarianceClock}] 
We will use the setting and notation of Section~\ref{sect:MartDiffPos}. Recall the definition of $W^*_n$ in \eqref{ShiftedClock}, and let $(Z^*_n, S^*_{n})$ be its building blocks. We assume that the hypotheses of Theorem~\ref{TheoremVarianceClockProcess} are satisfied. In what follows we will write $F(W^*_n) \coloneqq F_1(W^*_n(t_1), \dots, W^*_n(t_m))$ to lighten the notation. We also assume, without loss of generality, that $\|F_1\|_\infty \le 1$ and $F_1$ is $1$-Lipschitz.\\

\noindent \textbf{Step 1: Reduction to a finite box.}
The steps are identical to the ones for the position. We can insert the good event $E_{n, \varepsilon}$ to reduce the variance to the one accumulated inside the finite box $\mathcal{B}(c_1 n, (c_1 n)^\alpha)$. \\

\noindent \textbf{Step 2: Martingale difference.}
The steps are identical. At the end we find
\begin{equation}\label{eqn:BoundVarIntersectionJoint}
    \textbf{Var}\left( E_0^\omega\left[ F(W^*_n) \right] \right) = \sum_{k = -nc_1}^{nc_1} \sum_{z \in \hat{H}_k} \mathbf{E}\left[ \left(\Delta_{z}^n \right)^2\right],
\end{equation}
where
\begin{equation*}
    \Delta_{z}^n \coloneqq \mathbf{E}\left[ E_0^\omega\left[ F(W^*_n) \mathds{1}_{E_{n, \varepsilon}}\mathds{1}_{\{X \text{ visits } \mathcal{V}_z\}} \right] | \mathcal{G}_{z} \right] - \mathbf{E}\left[ E_0^\omega\left[ F(W^*_n)\mathds{1}_{E_{n, \varepsilon}}\mathds{1}_{\{X \text{ visits }\mathcal{V}_z\}} \right] | \mathcal{G}_{p(z)} \right].\\
\end{equation*}

\noindent \textbf{Step 3: The approximating quantity and inserting the good event.}
Let $z \in \mathbb{Z}^d$; we start by introducing the approximating quantity. Recall that $\tau_{j^+} = \tau^+_z \coloneqq \inf\{ \tau_{k} \colon X_{\tau_{k}} \cdot \vec{\ell} > z \cdot \vec{\ell}\}$ and $\tau^-_z \coloneqq \sup\{ k \le T_{(z-e_1) \cdot \vec{\ell}} \, \colon D\circ \theta_k + k > T_{(z-e_1) \cdot \vec{\ell}} \}$. Also recall that $j^+$ is the index of $\tau^+_z$ and $\tilde{j}^+$ is the one associated to $\tau^+_z$ when we count the potential regeneration points in $\{X\}_{0}^{\tau_z^-}$. We define the function
\begin{equation}\label{eqn:ApproxFunctionJoint}
    \tilde{F}\left(W^*_n\right) \coloneqq F(\{\tilde{W}^*_{n}\})\mathds{1}_{\{X \text{ visits }\mathcal{V}_z\}},
\end{equation}
where $\tilde{W}^*_{n}$ is the same as $W^*_n$ but is generated using the modified trajectory 
\begin{equation*}
    \{X\}_{0}^{\tau^-_z} \circ \left\{ \{X\}_{\tau^+_z}^{\tau_{n}} - X_{\tau^+_z}\right\}.
\end{equation*}
Once again we note that 
\begin{enumerate}
    \item $\{X\}_{0}^{\tau^-_z} \circ \left\{ \{X\}_{\tau^+_z}^{\tau_{n}} - X_{\tau^+_z}\right\}$ and as a consequence $F(\{\tilde{W}^*_{n}\})$ does not depend on the environment at $\mathcal{E}_z$.
    \item $\mathds{1}_{\{X \text{ visits }\mathcal{V}_z\}}$ does not depend on the environment at $\mathcal{E}_z$ as it is measurable with respect to the environment seen before $T_{(z-e_1) \cdot \vec{\ell}} \le T_{\mathcal{V}_z}$.
\end{enumerate}
Hence, we can proceed as in the corresponding step of Section~\ref{sect:MartDiffPos} and reduce the problem to bounding the quantities
\begin{align*}
     &\mathbf{E}_z\left[ E_0^\omega\left[ F(\{W^*_{n}\})\mathds{1}_{\{X \text{ visits }\mathcal{V}_z\}} \mathds{1}_{E_{n, \varepsilon}} \right] \right] - \mathbf{E}_z\left[ E_0^\omega\left[ F(\{\tilde{W}^*_{n}\})\mathds{1}_{\{X \text{ visits }\mathcal{V}_z\}}\mathds{1}_{E_{n, \varepsilon}} \right] \right]\\
     &\mathbf{E}_{p(z)}\left[ E_0^\omega\left[ F(\{W^*_{n}\})\mathds{1}_{\{X \text{ visits }\mathcal{V}_z\}}\mathds{1}_{E_{n, \varepsilon}} \right] \right] - \mathbf{E}_p(z)\left[ E_0^\omega\left[ F(\{\tilde{W}^*_{n}\})\mathds{1}_{\{X \text{ visits }\mathcal{V}_z\}}\mathds{1}_{E_{n, \varepsilon}} \right] \right].
\end{align*}

\noindent \textbf{Step 4: Separation of position and clock.}
Recall that $F_1$ evaluates its argument as a $m$-point function, $m \in \mathbb{N}$ arbitrary but fixed. Furthermore, its argument is a $(d+1)$-vector valued function. Hence by the triangle inequality, the lipschitz property and boundedness of $F_1$ we obtain
\begin{equation*}
    |F(\{W^*_{n}\}) - F(\{\tilde{W}^*_{n}\})| \le m\Big(\sup_{j = 1, \dots, m} \| Z^*_n(t_j) - \tilde{Z}^*_{n}(t_j) \| + \sup_{j = 1, \dots, m} \left| S^*_{n}(t_j) - \tilde{S}^*_{n}(t_j) \right| \Big) \wedge 1.
\end{equation*}
Thus we need to bound the following four terms:
\begin{equation}\label{eqn:FourTerms}
\begin{split}
    &\mathbf{E}_{p(z)}\left[ E_0^\omega\left[ \Big(\sup_{j = 1, \dots, m} \| Z^*_n(t_j) - \tilde{Z}^*_{n}(t_j) \|\wedge 1 \Big)\mathds{1}_{\{X \text{ visits }\mathcal{V}_z\}}\mathds{1}_{E_{n, \varepsilon}} \right] \right]\\
     &\mathbf{E}_z\left[ E_0^\omega\left[ \Big(\sup_{j = 1, \dots, m} \| Z^*_n(t_j) - \tilde{Z}^*_{n}(t_j) \| \wedge 1 \Big) \mathds{1}_{\{X \text{ visits }\mathcal{V}_z\}}\mathds{1}_{E_{n, \varepsilon}} \right] \right] \\
     &\mathbf{E}_{p(z)}\left[ E_0^\omega\left[ \Big(\sup_{j = 1, \dots, m} \left| S^*_{n}(t_j) - \tilde{S}^*_{n}(t_j) \right|\wedge 1 \Big)\mathds{1}_{\{X \text{ visits }\mathcal{V}_z\}}\mathds{1}_{E_{n, \varepsilon}} \right] \right]\\
     &\mathbf{E}_z\left[ E_0^\omega\left[ \Big(\sup_{j = 1, \dots, m} \left| S^*_{n}(t_j) - \tilde{S}^*_{n}(t_j) \right| \wedge 1 \Big) \mathds{1}_{\{X \text{ visits }\mathcal{V}_z\}}\mathds{1}_{E_{n, \varepsilon}} \right] \right].
\end{split}
\end{equation}
It is rather clear that for $0 \le t_1 \le \dots \le t_m \le 1 $
\begin{equation*}
    \sup_{j = 1, \dots, m} \| Z^*_n(t_j) - \tilde{Z}^*_{n}(t_j) \| \le \sup_{t \in [0, 1]} \| Z^*_n(t) - \tilde{Z}^*_{n}(t) \|.
\end{equation*}
Then, proceeding in the same way as we did to obtain \eqref{eqn:ErasingMiddle}, we see that the first term in \eqref{eqn:FourTerms} is such that 
\begin{align}
    \mathbf{E}_{p(z)}&\left[ E_0^\omega\left[ \Big(\sup_{j = 1, \dots, m} \| Z^*_n(t_j) - \tilde{Z}^*_{n}(t_j) \|\wedge 1 \Big)\mathds{1}_{\{X \text{ visits }\mathcal{V}_z\}}\mathds{1}_{E_{n, \varepsilon}} \right] \right]\nonumber \\
    &\le n^{-1/2+ 2d\alpha\varepsilon} \mathbf{E}_{p(z)}\left[E^\omega_0\left[\mathds{1}_{\{ X \text{ visits } \mathcal{V}_z \}}\right]\right]\label{eqn:PositionBoundJoint}
\end{align}
and the same bound holds for $\mathbf{E}_z$. The rest of the proof follows verbatim.\\

\noindent \textbf{Step 5: Non-uniform bound on bad areas for the clock.}
As already mentioned, the bound for the clock process is more involved. The first thing that we need to understand is what constitutes the difference $S^*_{n}(t_k) - \tilde{S}^*_{n}(t_k)$ for all $k = 1, \dots, m$. To begin, we introduce some indicator functions. Let us define
\begin{align*}
    \mathds{1}_{\{z\}} \coloneqq \mathds{1}_{\{X \text{ visits } \mathcal{V}_z\}} \quad \quad \text{and } \quad \quad \mathds{1}_{\{z, j\}} \coloneqq \mathds{1}_{\{\{X_k\}_{k = \tau_j + 1}^{\tau_{j+1}} \text{ visits } \mathcal{V}_z\}}.
\end{align*}
Crucially, observe that, for all $z$ in the sum \eqref{eqn:BoundVarIntersectionJoint} (using the fact that the events inside $\mathds{1}_{\{z, j\}}$ are mutually exclusive)
\begin{equation*}
    \mathds{1}_{\{z\}} = \sum_{j = 0}^{dc_1 n} \mathds{1}_{\{z, j\}}.
\end{equation*}
Recall that here $c_1$ ensures that, on $E_{n, \varepsilon}$, if $\mathcal{V}_z$ is hit at all it is hit among the first $dc_1 n$ regeneration times. 

We make several observations:
\begin{enumerate}
    \item Whenever the $j$ such that $\mathds{1}_{\{z, j\}} = 1$ is such that $j > t_k n$, then $S^*_{n}(t_k) = \tilde{S}^*_{n}(t_k)$.
    \item Whenever the $j$ such that $\mathds{1}_{\{z, j\}} = 1$ is such that $1 \le j \le t_k n$, then certainly $S^*_{n}(t_k) - \tilde{S}^*_{n}(t_k)$ contains $\tilde{\tau}_j$. Indeed, we recall the fact that by construction $\tau_z^{-} \ge \tau_{j}$ on the event $\{z, j\}$ and $\tau_{z}^+ = \tau_{j+1}$ on the same event. 
    \item Furthermore, our construction of $\{X\}_{0}^{\tau^-_z} \circ \left\{ \{X\}_{\tau^+_z}^{\tau_{n}} - X_{\tau^+_z}\right\}$ shifts the indexing of regeneration times by $0 \le \tilde{j}^+ - j^+ \le Cn^{\varepsilon}$. The inequality follows from $E_{n, \varepsilon}$. We also remark that on $\mathds{1}_{\{z, j\}} = 1$ we have $j^+ = j+1$.
    \item On the event $\mathds{1}_{\{z, 0\}}$, $S^*_{n}(t_k) - \tilde{S}^*_{n}(t_k)$ contains only $\tau_z^{-} - \tau_z^\mathrm{ini}$ where $\tau_z^\mathrm{ini} \coloneqq \inf\{ 0 \le k \le T_{(z-e_1) \cdot \vec{\ell}} \colon D\circ \theta_k + k > T_{(z-e_1) \cdot \vec{\ell}} \}$.
\end{enumerate}
Putting these observations together and using the triangle inequality we obtain (using the scaling by $\mathrm{Inv}(n) \approx n^{1/\gamma}$) for $j \ge 1$
\begin{equation*}
    \left|S^*_{n}(t_k) - \tilde{S}^*_{n}(t_k)\right| \mathds{1}_{\{z, j\}} \le \frac{1}{\mathrm{Inv}(n)} \tilde{\tau}_{j} \mathds{1}_{\{z, j\}} + \frac{1}{\mathrm{Inv}(n)} \left(S^*_{n}\left(t_k + \frac{\tilde{j}^+ - j^+}{n}\right) - S^*_{n}(t_k) \right)\mathds{1}_{\{z, j\}}.
\end{equation*}
Similarly,
\begin{equation*}
    \left|S^*_{n}(t_k) - \tilde{S}^*_{n}(t_k)\right| \mathds{1}_{\{z, 0\}} \le \frac{1}{\mathrm{Inv}(n)} (\tau_z^{-} - \tau_z^\mathrm{ini}) \mathds{1}_{\{z, 0\}} + \frac{1}{\mathrm{Inv}(n)} \left(S^*_{n}\left(t_k + \frac{\tilde{j}^+ - j^+}{n}\right) - S^*_{n}(t_k) \right)\mathds{1}_{\{z, 0\}}.
\end{equation*}
We observe that, under both $\mathbf{E}_{p(z)}$ and $\mathbf{E}_z$, the regeneration times in $S^*_{n}(t_k + \frac{\tilde{j}^+ - j^+}{n}) - S^*_{n}(t_k)$ happen in the part of the environment that is independent of $\mathcal{G}_{p(z)}$ (resp.\ $\mathcal{G}_{z}$). Hence they are distributed as $\tilde{\tau}_2$ under the measure $\mathbb{P}_{0}(\cdot | \, D = \infty)$. Thus we arrive at
\begin{equation}
\begin{split}
    \mathbf{E}_{p(z)}&\left[ E_0^\omega\left[ \Big(\sup_{j = 1, \dots, m} \left| S^*_{n}(t_j) - \tilde{S}^*_{n}(t_j) \right|\wedge 1 \Big)\mathds{1}_{\{X \text{ visits }\mathcal{V}_z\}}\mathds{1}_{E_{n, \varepsilon}} \right] \right] \label{eqn:BoundClockJoint} \\
    &\le \mathbf{E}_{p(z)}\left[ E_0^\omega\left[ \left(\frac{(\tau_z^{-} - \tau_z^\mathrm{ini})}{\mathrm{Inv}(n)} \wedge 1 \right) \mathds{1}{\{z, 0\}} \mathds{1}_{E_{n, \varepsilon}} \right] \right] + \mathbf{E}_{p(z)}\left[ E_0^\omega\left[\sum_{j = 1}^n \left(\frac{\tilde{\tau}_j}{\mathrm{Inv}(n)} \wedge 1 \right) \mathds{1}{\{z, j\}} \mathds{1}_{E_{n, \varepsilon}} \right] \right] \\
    &+ m \mathbf{E}_{p(z)}\left[E^\omega_0\left[\mathds{1}_{\{ X \text{ visits } \mathcal{V}_z \}}\right]\right] \mathbb{E}^K_0\left[\left(\frac{\tau_{Cn^\varepsilon}}{\mathrm{Inv}(n)} \wedge 1\right) \big| D = \infty\right].
\end{split}
\end{equation}
Note that, to obtain the third term, we used the fact that the regeneration time $\tau_{j^+}$ happens on the part of the environment which is independent of the part of the environment used by the random variable $\mathds{1}_{\{ X \text{ visits } \mathcal{V}_z \}}$. Furthermore, we have used the fact that, under $\mathbb{P}^K_{0}(\cdot | \, D = \infty)$, it holds that $\tau_{u} = \tau_{u + k} - \tau_{k}$ (in distribution) for arbitrary $u, k \in \mathbb{N}$. 

Moreover we notice, by the intermediate step \eqref{eqn:MeanTruncated} in the proof of Proposition~\ref{prop:MomentsSumHT}, that
\begin{equation}\label{eqn:TruncJoint}
    \mathbb{E}^K_0\left[\left(\frac{\tau_{Cn^\varepsilon}}{\mathrm{Inv}(n)} \wedge 1\right) \big| D = \infty\right] \le n^{-1+ c_3\varepsilon},
\end{equation}
for an absolute constant $c_3>0$. To bound $\mathbf{E}\left[ \left(\Delta_{z}^n \right)^2\right]$ from above we use \eqref{eqn:FourTerms}, the identity $(a+b)^2 \le 2(a^2+b^2)$ repeatedly, \eqref{eqn:PositionBoundJoint} and \eqref{eqn:TruncJoint} to obtain, for some constant $C>0$,
\begin{equation}
    \begin{split}\label{eqn:ThreeTerms}
         \mathbf{E}\left[ \left(\Delta_{z}^n \right)^2\right] &\le Cn^{-1 + c_4 \varepsilon} \mathbf{E}\left[ E_0^\omega\left[\mathds{1}_{\{ X \text{ visits } \mathcal{V}_z \}} \right]^2 \right]+C\mathbf{E}\left[ E_0^\omega\left[ \left(\frac{(\tau_z^{-} - \tau_z^\mathrm{ini})}{\mathrm{Inv}(n)} \wedge 1 \right) \mathds{1}{\{z, 0\}} \mathds{1}_{E_{n, \varepsilon}} \right]^2\right]\\
         &+C\mathbf{E}\left[ E^\omega_0\left[ \sum_{j = 1}^{c_2n} \left(\frac{\tilde{\tau}_j}{\mathrm{Inv}(n)} \wedge 1\right) \mathds{1}_{\{z, j\}}\mathds{1}_{E_{n, \varepsilon}}\right]^2  \right].
    \end{split}
\end{equation}
Firstly, we note that we already bounded the first term in Step~5 of Section~\ref{sect:MartDiffPos} and obtained
\begin{equation}
    \sum_{z \in \mathcal{B}_{0}(c_1 n, (c_1 n)^\alpha)}Cn^{-1 + c_4 \varepsilon} \mathbf{E}\left[ E_0^\omega\left[\mathds{1}_{\{ X \text{ visits } \mathcal{V}_z \}} \right]^2 \right] \le Cn^{-c},
\end{equation}
for some constant $c>0$.\\

\noindent \textbf{Step 6: The first regeneration.}
We deal now with the second term in \eqref{eqn:ThreeTerms}. We begin by noticing that
\begin{multline*}
    \sum_{z \in \mathcal{B}_{0}(c_1 n, (c_1 n)^\alpha)} \mathbf{E}\left[ E_0^\omega\left[ \left(\frac{(\tau_z^{-} - \tau_z^\mathrm{ini})}{\mathrm{Inv}(n)} \wedge 1 \right) \mathds{1}_{\{z, 0\}} \mathds{1}_{E_{n, \varepsilon}} \right]^2\right]\\
    \le Cn^{d\alpha\varepsilon} \sup_{0 \le |z \cdot \vec{\ell}| \le n^\varepsilon} \mathbb{E}_{0}\left[ \left(\frac{(\tau_z^{-} - \tau_z^\mathrm{ini})}{\mathrm{Inv}(n)} \wedge 1 \right)^2 \right] + Cn^{-M}.
\end{multline*}
Here, we exploited the fact that $\mathbb{P}_{0}(\{z, 0\}) \le Cn^{-M}$ for all $z$ with $|z \cdot \vec{\ell}| \ge n^{\varepsilon}$. From Proposition~\ref{prop:FirstReg} we obtain 
\begin{equation*}
    Cn^{d\alpha\varepsilon} \sup_{0 \le |z \cdot \vec{\ell}| \le n^\varepsilon} \mathbb{E}_{0}\left[ \left(\frac{(\tau_z^{-} - \tau_z^\mathrm{ini})}{\mathrm{Inv}(n)} \wedge 1 \right)^2 \right] \le Cn^{-c},
\end{equation*}
which makes the term negligible. Note that in applying Proposition~\ref{prop:FirstReg} we used the fact that $\mathrm{Inv}(n) \ge cn^{1/\gamma - \varepsilon}$ for any $\varepsilon>0$.\\

\noindent \textbf{Step 7: Reduction to increments of regeneration times during crossings.}
Finally, we are only left to deal with the third term in \eqref{eqn:ThreeTerms}, that is
\begin{equation*}
    \sum_{z \in \mathcal{B}_{0}(c_1 n, (c_1 n)^\alpha)} \mathbf{E}\left[ E^\omega_0\left[ \sum_{j = 1}^{c_2n} \left(\frac{\tilde{\tau}_j}{\mathrm{Inv}(n)} \wedge 1\right) \mathds{1}_{\{z, j\}} \mathds{1}_{E_{n, \varepsilon}}\right]^2  \right].
\end{equation*}
We use the two random walks evolving in the same environment $X^{(1)}$ and $X^{(2)}$. Then the quantity above becomes
\begin{equation}\label{eqn:2WalksIncrements}
    \sum_{z \in \mathcal{B}_{0}(c_1 n, (c_1 n)^\alpha)}\mathbf{E}\left[ E_0^\omega\left[ \sum_{j = 1}^n \mathds{1}^{(1)}_{\{z, j\}} \mathds{1}_{E_{n, \varepsilon}^{(1)}}\left(\frac{\tilde{\tau}^{(1)}_{j}}{\mathrm{Inv}(n)} \wedge 1\right) \right] E_{0}^\omega\left[ \sum_{k = 1}^n \mathds{1}^{(2)}_{\{z, k\}} \mathds{1}_{E_{n, \varepsilon}^{(2)}} \left(\frac{\tilde{\tau}^{(2)}_{k}}{\mathrm{Inv}(n)} \wedge 1\right) \right]\right].
\end{equation} 
Before proceeding, we define the random sets 
\begin{equation}\label{eqn:RandomIndeces}
	\begin{split}
		J^{(1)}_n &\coloneqq \left\{j = 1, \dots ,n \colon \{X^{(1)}\}_{\tau_j}^{\tau_{j+1} - 1} \cap \{X^{(2)}\}_{0}^{\tau_n} \neq \emptyset\right\} \\
		J^{(2)}_n &\coloneqq \left\{j = 1, \dots ,n \colon \{X^{(2)}\}_{\tau_j}^{\tau_{j+1} - 1} \cap \{X^{(1)}\}_{0}^{\tau_n} \neq \emptyset\right\}.
	\end{split}
\end{equation}
We observe that:
\begin{enumerate}
	\item Two indices $j, k$ in the sum \eqref{eqn:2WalksIncrements} are present if and only if there exists $z \in \mathcal{B}_{0}(c_1 n, (c_1 n)^\alpha)$ on which the two walks meet.
	\item A pair of indices $j \in J^{(1)}_n$ and $k \in J^{(2)}_n$ is counted at most $n^{\varepsilon}$ times when summing over $z \in \mathcal{B}_{0}(c_1 n, (c_1 n)^\alpha)$. Indeed, this is guaranteed by the good event $E^{(1)}_{n, \varepsilon} \cap E^{(2)}_{n, \varepsilon}$.
\end{enumerate}
Thus, we can bound from above the sum \eqref{eqn:2WalksIncrements} by 
\begin{equation*}
	n^{\varepsilon}\mathbf{E}\left[ E_{0, 0}^\omega\left[ \sum_{j \in J^{(1)}_n, k \in J^{(2)}_n}\left(\frac{\tilde{\tau}^{(1)}_{j}}{\mathrm{Inv}(n)} \wedge 1\right) \left(\frac{\tilde{\tau}^{(2)}_{k}}{\mathrm{Inv}(n)} \wedge 1\right) \right]\right].
\end{equation*} 
This fact, together with the estimate $\mathrm{Inv}(n) \ge c n^{1/\gamma - \varepsilon}$, yields the final bound
\begin{equation}\label{eqn:FinalBoundClockVar}
	n^{3\varepsilon}\mathbf{E}\left[ E_{0, 0}^\omega\left[ \sum_{j \in J^{(1)}_n, k \in J^{(2)}_n}\left(\frac{\tilde{\tau}^{(1)}_{j}}{n^{1/\gamma}} \wedge 1\right) \left(\frac{\tilde{\tau}^{(2)}_{k}}{n^{1/\gamma}} \wedge 1\right) \right]\right].
\end{equation} 
Finally, the polynomial upper bound for \eqref{eqn:BoundVarIntersectionJoint} can be deduced by applying Cauchy-Schwarz inequality to \eqref{eqn:FinalBoundClockVar} above together with Proposition~\ref{prop:KeyBoundClock}. The proof is now concluded.
\end{proof}

\begin{remark1}
    In our bounding procedure we seem to lose quite a bit since, in principle, it is sufficient to take the product of regeneration times that cross each other, while we take the product of the sums which is clearly a larger quantity. However, we are not losing much with this simplification because of the nature of heavy tailed random variables. Indeed, one can expect the sum over $j \in J^{(1)}_n$ to be of the same order of magnitude of the maximum. But then, since we have a bad control of the correlation between $\tilde{\tau}^{(1)}_j, j \in J^{(1)}_n$ and $\tilde{\tau}^{(2)}_k, k \in J^{(2)}_n$ when a crossing occurs (and we should not expect to have a good one), we will need to use Cauchy-Schwarz to obtain an upper bound. Thus, the upper bound above is, in some sense, tight or at least tight enough.
\end{remark1}

\section{The tail of regeneration times associated to intersections} \label{sect:Cross}
The main goal of this section is to prove Proposition~\ref{prop:KeyBoundClock}, which is the key ingreadient to bound \eqref{eqn:FinalBoundClockVar}.

We define the event 
\begin{equation}\label{eqn:DefLT}
    LT(t) \coloneqq \left\{ \text{there exists } e \in \mathbb{Z}^d \colon c_*^{\omega}(e) \ge t, T_{e}< \tau_1\right\},
\end{equation}
where we recall that $T_e$ is the hitting time of the edge $e$ seen as a set. We slightly extend this definition to
\[LT_i(t) \coloneqq \left\{ \text{there exists } e \in \mathbb{Z}^d \colon c_*^{\omega}(e) \ge t, \tau_i\le T_{e}< \tau_{i+1}\right\}.\]
Furthermore, we define $k_{j}^{(i)} \in \mathbb{N}, i=1,2$ to be the unique index such that $\tau^{(i)}_{k_{j}^{(i)}} = T^{(i)}_{\mathcal{L}_j}$.

\begin{lemma}\label{lemma:SmallProb}
Uniformly over all $t\in[0, n^{1/\gamma}]$ and $U_1, U_2 \in \mathcal{U}$, for all $\varepsilon>0$ there exists $K_0>0$ such that for all $K \ge K_0$ we have, for $i = 1, 2$
\begin{equation*}
    \mathbb{P}^K_{U_1, U_2}\left( \bigcup_{j = 0}^{k^{(i)}_1} LT_j^{(i)}(t)\mid  \mathcal{D}^\bullet = \infty \right) \le Cn^{\varepsilon/2}\mathbb{P}^K_{0}\left( LT(t) \mid D = \infty\right) \le Cn^{\varepsilon} \mathbf{P}\left( c_* \ge t \right),
\end{equation*}
for a constant $C = C(K, \lambda, \vec{\ell}, d)$.
\end{lemma}
\begin{proof} We write the proof for the first walk, that is we fix $i = 1$; the argument for $i = 2$ is the same. Recalling Lemma~\ref{lemma:GoodEventsBoxes}, for $M>0$ arbitrarily large we obtain
    \begin{align*}
        \mathbb{P}^K_{U_1, U_2}\left( LT^{(1)}(t)| \mathcal{D}^\bullet = \infty \right) &= \mathbb{P}^K_{U_1, U_2}\left( LT^{(1)}(t), \, \chi > n^{\varepsilon} | \mathcal{D}^\bullet = \infty \right) + \mathbb{P}^K_{U_1, U_2}\left( LT^{(1)}(t), \, \chi \le n^{\varepsilon}| \mathcal{D}^\bullet = \infty \right) \\
        &\le Cn^{-M} + \mathbb{P}^K_{U_1, U_2}\left( \exists e \in \mathcal{B}_{U_1}(n^{\varepsilon}, n^{\varepsilon\alpha}) \colon c_*(e) \ge t | \mathcal{D}^\bullet = \infty \right)\\
        &\le Cn^{-M} + n^{\alpha d \varepsilon} \mathbf{P}(c_* \ge t)\\
        &= Cn^{-M} + n^{\varepsilon'} \mathbf{P}(c_* \ge t) \le C n^{\varepsilon'} \mathbf{P}(c_* \ge t).
    \end{align*}
    The second inequality follows from the translation invariance of the environment and a union bound. With the same argument we obtain, uniformly over $U_1 \in \mathbb{Z}^d$,
    \begin{equation*}
        \mathbb{P}^K_{U_1}\left( LT^{(1)}(t)| D^{(1)} = \infty \right)  \le C n^{\varepsilon'} \mathbf{P}(c_* \ge t).
    \end{equation*}
    Furthermore, let us consider the event $\{k^{(1)}_1 \le n^{\varepsilon'}\}$; using that for all $L \ge 0$ there are at most $Ld$ regeneration times such that $0 \le X_{\tau_j} \le L$, together with Lemma~\ref{lemma:GoodEventsBoxes}, we obtain  
    \begin{equation*}
        \mathbb{P}^K_{U_1, U_2}\left( k^{(1)}_1 > n^{\varepsilon'}| \mathcal{D}^\bullet = \infty \right) \le \mathbb{P}^K_{U_1, U_2}\left( \mathcal{L}_1 > \frac{1}{d}n^{\varepsilon'}| \mathcal{D}^\bullet = \infty \right) \le Cn^{-M}.
    \end{equation*}
    Then we write
    \begin{align*}
        \mathbb{P}^K_{U_1, U_2}\left( \bigcup_{j = 0}^{k^{(1)}_1} LT_j^{(1)}(t)| \mathcal{D}^\bullet = \infty \right) &=\mathbb{P}^K_{U_1, U_2}\left( \bigcup_{j = 0}^{k^{(1)}_1} LT_j^{(1)}(t), k^{(1)}_1 \le n^{\varepsilon'}| \mathcal{D}^\bullet = \infty \right) + \mathbb{P}^K_{U_1, U_2}\left( k^{(1)}_1 > n^{\varepsilon'}| \mathcal{D}^\bullet = \infty \right)\\
        &\le \mathbb{P}^K_{U_1, U_2}\left( \bigcup_{j = 0}^{n^{\varepsilon'}} LT_j^{(1)}(t)| \mathcal{D}^\bullet = \infty \right) + Cn^{-M}.
    \end{align*}
    By the properties of the regeneration structure given in Definition~\ref{def:RegTimesWellDef} we obtain
    \begin{equation*}
        \mathbb{P}^K_{U_1, U_2}\left( \bigcup_{j = 0}^{n^{\varepsilon'}} LT_j^{(1)}(t)| \mathcal{D}^\bullet = \infty \right) \le \mathbb{P}^K_{U_1, U_2}\left( LT^{(1)}(t)| \mathcal{D}^\bullet = \infty \right) + \mathbb{P}^K_{0}\left(\bigcup_{j = 1}^{n^{\varepsilon'}} LT_j^{(1)}(t)| D = \infty \right).
    \end{equation*}
    Note that the first term on the r.h.s.\ takes care of the first regeneration time after the joint regeneration level and, for the second term on the r.h.s., we use the translation invariance of $\mathbb{P}^K_{x}$. We also highlight that the random walk $\{X_{\tau_1 + j}^{(1)} - X_{\tau_1}^{(1)} \}_{j \ge 0}$ under the measure $\mathbb{P}^K_{U_1, U_2}\left( \cdot | D = \infty \right)$ is distributed as $\{X_{j}\}_{j \ge 0}$ under $\mathbb{P}^K_{0}$. Using a union bound and the estimates given at the beginning of the proof we obtain 
    \begin{equation*}
        \mathbb{P}^K_{U_1, U_2}\left( LT^{(1)}(t)| \mathcal{D}^\bullet = \infty \right) + \mathbb{P}^K_{0}\left(\bigcup_{j = 1}^{n^{\varepsilon'}} LT_j^{(1)}(t)| D = \infty \right) \le C n^{2\varepsilon'} \mathbf{P}(c_* \ge t).
    \end{equation*}
    The proof is finished as we can choose $\varepsilon'$ to be arbitrarily small by tuning $M$ through our choice of $K_0$. 
\end{proof}
Let us introduce the set of indices associated to the increments of regeneration time where no large conductance is met. For $i = 1, 2$
    \begin{equation}\label{eqn:DefLsmall}
        L^{(i)}_{\text{small}}(n) \coloneqq \left\{j = 1, \dots, n \colon LT^{(i)}_j(n^{\frac{3}{4\gamma}} )^c\right\}.
    \end{equation}
The next lemma shows that the set of indices $L^{(1)}_{\text{small}}(n)$ contains, with high probability, our set of interest $J^{(1)}(n)$.
\begin{lemma}\label{lemma:WhenCrossNotLarge}
For some absolute constant $C>0$, we have that
    \begin{equation*}
        \mathbb{P}_{0, 0}\left( J^{(1)}(n) \subseteq L^{(1)}_{\text{small}}(n)\right) \ge 1 - Cn^{-1/8}.
    \end{equation*}
    The same result holds for $J^{(2)}(n)$.
\end{lemma}

\begin{proof}
We will write the proof only for the walk $X^{(1)}$, for simplicity let $L_{\text{small}}(n) = L^{(1)}_{\text{small}}(n)$. Consider the set $ \mathrm{JRL}^\le(n, \varepsilon)$ defined in \eqref{eqn:BadRegLevels}. Recall also the complement set of $\mathrm{JRL}^\le(n, \varepsilon)$ %of joint regeneration levels 
\begin{equation*}
    \mathrm{JRL}^>(n, \varepsilon) \coloneqq \left\{k \colon \mathcal{L}_k \le n \text{ and } \left\|X^{(1)}_{T^{(1)}_{\mathcal{L}_k}} - X^{(2)}_{T^{(2)}_{\mathcal{L}_k}}\right\|> n^{10\varepsilon}\right\}.
\end{equation*}
For $k = 1, \dots, n$ let us define the event 
\begin{equation}\label{eqn:CrossAfterJointReg}
    \mathrm{Bad}_n(k) \coloneqq \left\{ \exists j \in J^{(1)}(n) \colon T^{(1)}_{\mathcal{L}_k}\le \tau^{(1)}_{j} < T^{(1)}_{\mathcal{L}_{k+1}}, LT_j(n^{\frac{3}{4\gamma}}) \right\}.
\end{equation}
Note that, from their definition, the first $n$ regeneration times happen before $\mathcal{L}_n$, hence
\begin{equation*}
    \{J^{(1)}(n) \not\subseteq L_{\text{small}}(n) \} \subseteq \bigcup_{k=1}^n \mathrm{Bad}_n(k).
\end{equation*}
Recall the event $\mathrm{FC}_{c_1n, \varepsilon}^c$ that we defined in \eqref{eqn:CrossAfterFarJointReg} (we will characterize the constant $c_1>0$ shortly). By the observations already made in the proof of Proposition~\ref{prop:FewCrossing} we have $\mathrm{FC}_{c_1n, \varepsilon}^c \subseteq F_{n, \varepsilon}^c \cup L_{n, \varepsilon}^c$. On the event $\mathrm{FC}_{c_1n, \varepsilon}$ there is no $j \in J^{(1)}(n)$ such that $T^{(1)}_{\mathcal{L}_k} \le \tau^{(1)}_j \le T^{(1)}_{\mathcal{L}_{k+1}}$ whenever the index $k \in \mathrm{JRL}^>(c_1n, \varepsilon)$. Note that this occurs with probability $\mathbb{P}_{0, 0}(\mathrm{FC}_{c_1n, \varepsilon}^c) \le \mathbb{P}(F_{n, \varepsilon}^c \cup L_{n, \varepsilon}^c) \le Cn^{-M}$ by Lemma~\ref{lemma:GoodEventsBoxes}. We highlight the fact that $c_1$ is an absolute constant depending only on the parameters of the model and $K$ that ensures that $X^{(1)}_{\tau_n} \cdot \vec{\ell} \le c_1 n$. Note that we can fix $c_1$ in such a way that the event $X^{(1)}_{\tau_n} \cdot \vec{\ell} \le c_1 n$ happens with probability $\ge 1 - Cn^{-M}$ as we showed below \eqref{eqn:NotTooFar}. In formulas
\begin{align*}
    \mathbb{P}_{0, 0}&\left( \bigcup_{k=1}^n \mathrm{Bad}_n(k) \right) \\
    &\le \mathbb{P}_{0, 0}\left( \bigcup_{k=1}^n \mathrm{Bad}_n(k), F_{n, \varepsilon} \cap L_{n, \varepsilon} \cap \{X^{(1)}_{\tau_n} \cdot \vec{\ell} \le c_1 n\} \right) + \mathbb{P}_{0, 0}\left( F_{n, \varepsilon}^c \cup L_{n, \varepsilon}^c \cup \{X^{(1)}_{\tau_n} \cdot \vec{\ell}>c_1n\} \right)\\
    & \le \mathbb{P}_{0, 0}\left( \bigcup_{k=1}^n \mathrm{Bad}_n(k), F_{n, \varepsilon} \cap L_{n, \varepsilon} \cap \{X^{(1)}_{\tau_n} \cdot \vec{\ell} \le c_1 n\} \right) + Cn^{-M}.
\end{align*}
We also highlight that, by definition of the events and the observations we made above
\begin{equation*}
    \bigcup_{k=1}^n \mathrm{Bad}_n(k) \cap F_{n, \varepsilon} \cap L_{n, \varepsilon} \cap \{X^{(1)}_{\tau_n} \cdot \vec{\ell}\le c_1 n\} \subseteq \bigcup_{k=1}^n \mathrm{Bad}_n(k)\cap\{ k \in \mathrm{JRL}^\le(c_1n, \varepsilon)\}.
\end{equation*}
Putting things together we get the intermediate step 
\begin{align*}
        \mathbb{P}_{0, 0}\left( J^{(1)}(n) \not\subseteq L_{\text{small}}(n)\right) &\le \mathbb{P}_{0, 0}\left( \bigcup_{k=1}^n \mathrm{Bad}_n(k) \right) = \mathbb{E}_{0, 0}\left[ \mathds{1}_{\{\bigcup_{k=1}^n \mathrm{Bad}_n(k)\}} \right]\\
        &\le\mathbb{E}_{0, 0}\left[ \mathds{1}_{\{\bigcup_{k=1}^n \mathrm{Bad}_n(k)\cap\{ k \in \mathrm{JRL}^\le(c_1n, \varepsilon) \}\}} \right] + Cn^{-M}.
\end{align*}
Now consider the ordered indices in $\mathrm{JRL}^\le(c_1n, \varepsilon)$, we drop the constant $c_1$ for simplicity as it will be clear that the estimates will not depend on it. Note that the event $\{k \in \mathrm{JRL}^\le(n, \varepsilon) \}$ is measurable w.r.t.\ the sigma-algebra $\Sigma_k$ generated by the joint regeneration levels $\{\mathcal{L}_k,(X^{(1)}_{i})_{0 \le i \le T_{\mathcal{L}_k}},(X^{(2)}_{j})_{0 \le j \le T_{\mathcal{L}_k}}\}$.

Note that $\left\{ \exists j \colon T^{(1)}_{\mathcal{L}_k}\le \tau^{(1)}_{j} < T^{(1)}_{\mathcal{L}_{k+1}}, LT_j(n^{\frac{3}{4\gamma}}) \right\}$ is measurable w.r.t.\ what happens after $\mathcal{L}_k$. By \cite[Theorem~3.25]{QuenchedBiasedRWRC} and Lemma~\ref{lemma:SmallProb} we get that for any $k \in \mathbb{N}$
\begin{align*}
   \mathbb{E}_{0, 0}&\left[ \mathds{1}_{\{\mathrm{Bad}_n(k)\cap\{ k \in \mathrm{JRL}^\le(n, \varepsilon) \}\}} \right] = \mathbb{E}_{0, 0}\left[ \mathds{1}_{\{ k \in \mathrm{JRL}^\le(n, \varepsilon) \}}\mathbb{P}^K_{X^{(1)}_{T_{\mathcal{L}_k}},X^{(2)}_{T_{\mathcal{L}_k}}} \left(\mathrm{Bad}_n(k) | \, \mathcal{D}^\bullet =\infty\right) \right] \\
    &\le\mathbb{E}_{0, 0}\Bigg[ \mathds{1}_{\{ k \in \mathrm{JRL}^\le(n, \varepsilon) \}}\mathbb{P}^K_{X^{(1)}_{T_{\mathcal{L}_k}},X^{(2)}_{T_{\mathcal{L}_k}}} \Bigg( \bigcup_{j = 1}^{k^{(1)}_1} LT_j(n^{\frac{3}{4\gamma}}) | \, \mathcal{D}^\bullet =\infty\Bigg) \Bigg] \\
    &\le \mathbb{E}_{0, 0}\left[ \mathds{1}_{\{ k \in \mathrm{JRL}^\le(n, \varepsilon) \}} Cn^{2\varepsilon} \mathbf{P}\left( c^* \ge n^{\frac{3}{4\gamma}} \right) \right] \le \mathbb{E}_{0, 0}\left[ \mathds{1}_{\{ k \in \mathrm{JRL}^\le(n, \varepsilon) \}} \right] Cn^{-\frac{3}{4} + 3\varepsilon}.
\end{align*}
Here it is crucial that Lemma~\ref{lemma:SmallProb} is valid uniformly over the choice of the starting points. 
 We are now able to bound the remaining term
\begin{align*}
    \mathbb{E}_{0, 0}&\left[ \mathds{1}_{\{\bigcup_{k=1}^n \mathrm{Bad}_n(k)\cap\{ k \in \mathrm{JRL}^\le(n, \varepsilon) \}\}} \right] 
    \\&
    \le \mathbb{E}_{0, 0}\left[ \mathds{1}_{\{\bigcup_{k=1}^{n-1} \mathrm{Bad}_n(k)\cap\{ k \in \mathrm{JRL}^\le(n, \varepsilon) \}\}}  \right] + \mathbb{E}_{0, 0}\left[  \mathds{1}_{\{\mathrm{Bad}_n(n)\cap\{ n \in \mathrm{JRL}^\le(n, \varepsilon) \}\}}  \right]\\
    &\le \mathbb{E}_{0, 0}\left[ \mathds{1}_{\{\bigcup_{k=1}^{n-1} \mathrm{Bad}_n(k)\cap\{ k \in \mathrm{JRL}^\le(n, \varepsilon) \}\}}  \right] + \mathbb{E}_{0, 0}\left[ \mathds{1}_{\{n \in \mathrm{JRL}^\le(n, \varepsilon)\}} \right] Cn^{-\frac{3}{4} + 2\varepsilon}.
\end{align*}
Iterating this procedure we get to 
\begin{align*}
    \mathbb{E}_{0, 0}\left[ \mathds{1}_{\{\bigcup_{k=1}^n \mathrm{Bad}_n(k)\cap\{ k \in \mathrm{JRL}^\le(n, \varepsilon) \}\}} \right] \le \mathbb{E}_{0, 0}\left[ \sum_{k = 1}^n\mathds{1}_{\{k \in \mathrm{JRL}^\le(n, \varepsilon)\}} \right] Cn^{-\frac{3}{4} + 3\varepsilon}.
\end{align*}
By Proposition~\ref{prop:CrossingJRL} we get
\begin{equation*}
    \mathbb{E}_{0, 0}\left[ \sum_{k = 1}^n\mathds{1}_{\{k \in \mathrm{JRL}^\le(n, \varepsilon)\}} \right] \le n^{\frac{1}{2} + c\varepsilon}.
\end{equation*}
All together we have
\begin{equation}
    \mathbb{P}_{0, 0}\left( J^{(1)}(n) \not\subseteq L_{\text{small}}(n)\right) \le C n^{-\frac{1}{4} + c\varepsilon} \le Cn^{-1/8},
\end{equation}
as $\varepsilon>0$ is arbitrary and $c>0$ is an absolute constant, this concludes the proof. 
\end{proof}
\noindent We are now in the position to prove the main result of the section.
\begin{proposition} \label{prop:KeyBoundClock}
    We have that
    \begin{equation*}
        \mathbb{E}_{0, 0}\left[ \left( \sum_{j \in J^{(1)}} \left(\frac{\tilde{\tau}^{(1)}_j}{n^{1/\gamma}} \wedge 1\right) \right)^2\right] \le  Cn^{-c},
    \end{equation*}
    for two absolute constants $C, c>0$. The same result holds for $J^{(2)}(n)$.
\end{proposition}
\begin{proof}
Firstly we notice that 
    \begin{align}
        \mathbb{E}_{0, 0}&\left[ \left( \sum_{j \in J^{(1)}} \left(\frac{\tilde{\tau}^{(1)}_j}{n^{1/\gamma}} \wedge 1\right) \right)^2\right]\nonumber \\
        &= \mathbb{E}_{0, 0}\left[ \left( \sum_{j\in J^{(1)}} \left(\frac{\tilde{\tau}^{(1)}_j}{n^{1/\gamma}} \wedge 1\right) \right)^2\mathds{1}_{J^{(1)}(n) \subseteq L_{\text{small}}(n)}\right] + \mathbb{E}_{0, 0}\left[ \left( \sum_{j=1}^n \left(\frac{\tilde{\tau}^{(1)}_j}{n^{1/\gamma}} \wedge 1\right) \right)^2\mathds{1}_{J^{(1)}(n) \not\subseteq L_{\text{small}}(n)}\right]\nonumber\\
        &\le \mathbb{E}_{0, 0}\left[ \left( \sum_{j=1}^n \left(\frac{\tilde{\tau}^{(1)}_j}{n^{1/\gamma}} \wedge 1\right)\mathds{1}_{LT_j(n^{\frac{3}{4\gamma}})^c} \right)^2\right] + \mathbb{E}_{0, 0}\left[ \left( \sum_{j=1}^n \left(\frac{\tilde{\tau}^{(1)}_j}{n^{1/\gamma}} \wedge 1\right) \right)^2\mathds{1}_{J^{(1)}(n) \not\subseteq L_{\text{small}}(n)}\right].\label{eqn:TwoTermsTruncated} 
    \end{align}
    Let us justify the inequality, note that on the event $J^{(1)}(n) \subseteq L_{\text{small}}(n)$ then the event $LT_j(n^{\frac{3}{4\gamma}})^c$ happens for all $j \in J^{(1)}(n)$. This implies
    \begin{equation*}
        \left( \sum_{j\in J^{(1)}} \left(\frac{\tilde{\tau}^{(1)}_j}{n^{1/\gamma}} \wedge 1\right) \right)^2\mathds{1}_{J^{(1)}(n) \subseteq L_{\text{small}}(n)} \le \left( \sum_{j\in L_{\text{small}}(n)} \left(\frac{\tilde{\tau}^{(1)}_j}{n^{1/\gamma}} \wedge 1\right) \right)^2 = \left( \sum_{j=1}^n \left(\frac{\tilde{\tau}^{(1)}_j}{n^{1/\gamma}} \wedge 1\right)\mathds{1}_{LT_j(n^{\frac{3}{4\gamma}})^c} \right)^2.
    \end{equation*}
    We bound the two terms in \eqref{eqn:TwoTermsTruncated} separately. Applying Cauchy-Schwarz and Proposition~\ref{prop:MomentsSumHT} we get
    \begin{equation}\label{eqn:CSBound}
    \begin{split}
        \mathbb{E}_{0, 0}&\left[ \left( \sum_{j=1}^n \left(\frac{\tilde{\tau}^{(1)}_j}{n^{1/\gamma}} \wedge 1\right) \right)^2\mathds{1}_{J^{(1)}(n) \not\subseteq L_{\text{large}}(n)}\right]\\&\le \mathbb{E}_{0, 0}\left[ \left( \sum_{j=1}^n \left(\frac{\tilde{\tau}^{(1)}_j}{n^{1/\gamma}} \wedge 1\right) \right)^4\right]^{1/2} \mathbb{P}_{0, 0}\left( J^{(1)}(n) \not\subseteq L_{\text{small}}(n)\right)\\
        &\le C n^{4\varepsilon} n^{-1/8}.
    \end{split}
    \end{equation}
    This takes care of the second term. The first term is more involved. Let us recall the quantities defined in \cite[(6.2),(6.3)]{Kious_Frib}, let $E_{<t}$ be the set of edges $e \in \mathbb{Z}^d$ with $c_{*}(e)<t$,
    \begin{equation*}
        \tau_1^{\ge t} \coloneqq \sum_{e \in E^c_{<t} \cap \{\mathcal{R}_0\setminus \mathcal{E}_0 \}} |\{k \in [1, \tau_1] \colon [X_{k-1}, X_{k}] = e\}|.
    \end{equation*}
    Moreover, $\tau_{1}^{< t} = \tau_1 - \tau_1^{\ge t}$. Furthermore define, for $* = <, \ge$ 
    \begin{equation*}
        \tilde{\tau_j}^{* t} \coloneqq \tau_{1}^{* t}\left(\{X_{\tau_{j} + i} - X_{\tau_{j}}\}_{i \ge 1}\right)
    \end{equation*} 
    Note that, $\tilde{\tau_j}^{\ge t}$ is ``simply'' the amount of time the random walk spends on large conductances in the time interval $[\tau_{j}, \tau_{j+1}]$. It is obvious that $ \tilde{\tau_j}^{< t} \ge \tilde{\tau_j} \mathds{1}_{LT_{j}(t)^c},$ indeed, they coincide on $LT_{j}(t)^c$ and the left hand side is non zero on $LT_{j}(t)$. Fix $\eta>0$ (we will choose its value small enough later) and define the event
    \begin{equation}
      \mathrm{NG}_n \coloneqq  \bigcup_{i=1}^{n}\{\tilde{\tau}_{j} > n^{\frac{1 - \eta}{\gamma}}\}\cap LT_j(n^{\frac{3}{4\gamma}})^c.
    \end{equation}
    By the observation above
    \begin{align*}
        \mathbb{P}_{0, 0} \left(\mathrm{NG}_n\right) \le n \mathbb{P}_{0, 0} \left( \tilde{\tau}_{1}^{<n^{\frac{3}{4\gamma}}} > n^{\frac{1 - \eta}{\gamma}}\right) \le n \cdot n^{-1 + \eta - \frac{(1 - 3\frac{1-\eta}{4})(1-\gamma)}{2}} \le n^{-\eta}.
    \end{align*}
    In the second inequality we applied \cite[Lemma~6.2]{Kious_Frib}. Note that, crucially, $(1 - 3\frac{1-\eta}{4})$ is arbitrarily close to $1/4$ for $\eta$ small enough. Proceeding as in \eqref{eqn:CSBound} we obtain
    \begin{equation*}
        \mathbb{E}_{0, 0}\left[ \left( \sum_{j=1}^n \left(\frac{\tilde{\tau}^{(1)}_j}{n^{1/\gamma}} \wedge 1\right) \right)^2\mathds{1}_{\mathrm{NG}_n}\right]\le C n^{4\varepsilon} n^{-\eta} \le C n^{-\eta/2}.
    \end{equation*}
    Hence, we are only left with the task of bounding
    \begin{equation*}
        \mathbb{E}_{0, 0}\left[ \left( \sum_{j=1}^n \left(\frac{\tilde{\tau}^{(1)}_j}{n^{1/\gamma}} \wedge 1\right)\mathds{1}_{\tilde{\tau}_{j} \le n^{(1-\eta)/\gamma}} \right)^2\right] \le \mathbb{E}_{0, 0}\left[ \left( \sum_{j=1}^n \left(\frac{\tilde{\tau}^{(1)}_j}{n^{1/\gamma}} \wedge n^{-\eta}\right) \right)^2\right],
    \end{equation*}
    a direct application of Lemma~\ref{lemma:MomentsSumHT2} concludes the proof with $c = \eta \gamma /2$. The proof for $J^{(2)}(n)$ follows identically.
\end{proof}

\subsection{The first regeneration period}

\begin{proposition}\label{prop:FirstReg}
    There exist two positive constants $C, c > 0$ such that
    \begin{equation*}
        \sup_{0 \le |z \cdot \vec{\ell}| \le n^\varepsilon} \mathbb{E}_{0}\left[ \left(\frac{(\tau_z^{-} - \tau_z^\mathrm{ini})}{n^{1/\gamma}} \wedge 1 \right)^2 \right] \le Cn^{-c}.
    \end{equation*}
\end{proposition}
\begin{proof}
    By the definition of $\tau_z^{-}$ and $\tau_z^\mathrm{ini}$ and using the translation invariance of the environment we obtain 
    \begin{equation*}
        \sup_{0 \le |z \cdot \vec{\ell}| \le n^\varepsilon}\mathbb{E}_{0}\left[ \left(\frac{(\tau_z^{-} - \tau_z^\mathrm{ini})}{n^{1/\gamma}} \wedge 1 \right)^2 \right] \le \sup_{0 \le |z \cdot \vec{\ell}| \le n^\varepsilon} \mathbb{E}^K_{0}\left[ \left(\frac{\tau_z^{-}}{n^{1/\gamma}} \wedge 1 \right)^2 \mid D \ge T_{z}\right].
    \end{equation*}
    Let us now show the proof for $|z \cdot \vec{\ell}| = n^\varepsilon$; all other values follow in the same way. We fix $t \ge 1$ arbitrary and show that
    \begin{equation}\label{eqn:ConclusionFirst}
        \mathbb{P}^K_{0}\left( \tau_z^{-} > t \mid D \ge T_{n^\varepsilon}\right) \le Cn^{\varepsilon} \mathbb{P}^K_{0}\left( \tau_1 > t \mid D = \infty\right);
    \end{equation}
    then the result follows from Proposition~\ref{prop:MomentsSumHT}. Let us now prove the above inequality. Using the definition of conditional probability
    \begin{align*}
        \mathbb{P}^K_{0}\left( \tau_z^{-} > t \mid D \ge T_{n^\varepsilon}\right) &= \frac{\mathbb{P}^K_{0}\left( \tau_z^{-} > t , D \ge T_{n^\varepsilon}\right)}{\mathbb{P}^K_{0}\left( D \ge T_{n^\varepsilon}\right)} \le \frac{\mathbb{P}^K_{0}\left( \tau_z^{-} > t , D \ge \tau_z^{-} \right)}{\mathbb{P}^K_{0}\left( D \ge T_{n^\varepsilon}\right)} \\
        & \le \frac{\mathbb{P}^K_{0}\left( \tau_z^{-} > t , D \ge \tau_z^{-} \right) \mathbb{P}^K_{0}\left( D \circ \theta_{\tau_z^-} =\infty\right)}{\mathbb{P}^K_{0}\left( D = \infty\right)^2}.
    \end{align*}
    To justify the last step we appeal to translation invariance and the fact that $\tau_z^-$ is constructed so that $X_{\tau_z^-}$ is $K$-open to obtain
    \begin{equation*}
        \mathbb{P}^K_{0}\left( D \circ \theta_{\tau_z^-} =\infty\right) = \mathbb{P}^K_{0}\left( D =\infty\right) = \rho>0.
    \end{equation*}
    Furthermore, we note that 
    \begin{equation*}
        \mathbb{P}^K_{0}\left( \tau_z^{-} > t , D \ge \tau_z^{-} \right) \mathbb{P}^K_{0}\left( D \circ \theta_{\tau_z^-} =\infty\right) = \mathbb{P}^K_{0}\left( \tau_z^{-} > t , D \ge \tau_z^{-} , D \circ \theta_{\tau_z^-} =\infty\right).
    \end{equation*}
    We observe that $\{D \ge \tau_z^{-} , D \circ \theta_{\tau_z^-} =\infty\} \subseteq \{D = \infty\}$, and, deterministically, there exists $c_1> 0$ such that $\tau_z^- \le \tau_{c_1 n^\varepsilon}$ as we know that $|z\cdot \vec{\ell}| \le n^\varepsilon$. Then, using this last observations and the fact that we have just showed that \eqref{eqn:ConclusionFirst} holds, we obtain 
    \begin{equation*}
        \sup_{0 \le |z \cdot \vec{\ell}| \le n^\varepsilon} \mathbb{E}^K_{0}\left[ \left(\frac{\tau_z^{-}}{n^{1/\gamma}} \wedge 1 \right)^2 \mid D \ge T_{z}\right] \le n^{\varepsilon} \mathbb{E}^K_{0}\left[ \left(\frac{\tau_{n^{\varepsilon}}}{n^{1/\gamma}} \wedge 1 \right) \mid D = \infty\right].
    \end{equation*}
    Using the same argument that we used in Lemma~\ref{lemma:WhenCrossNotLarge}, it is straightforward to show that $\{1, \dots, n^{\varepsilon}\} \subseteq L_{\text{small}}(n)$ with probability larger than $1 - Cn^{-1/8}$. The result then follows by the same argument given in Proposition~\ref{prop:KeyBoundClock}, we avoid repeating it here.
\end{proof}

\section{Proof of quenched limits}

The proof is the same as the one given in \cite[Section~6]{QuenchedBiasedRWRC}, for completeness and to keep this work self-contained we will give an overview of the main steps, but we will forego some details. 

We firstly set some notation, let $D^d([0, T])$ is the space of functions $f\colon [0, T]\to\mathbb{R}^d$ that are right continuous with left limits (c\`{a}dl\`{a}g). We will write $D([0, T])$ in place of $D^1([0, T])$. Similarly, let $\mathcal{C}^d([0, T])$ be the set of continuous functions. In what follows we will endow our spaces with the uniform topology $U$ and Skorohod's $J_1$ and $M_1$ topologies (see \cite[Chapter~11-13]{whitt}).

We recall that $v = \mathbb{E}_{0}[X_{\tau_1} |\, D = \infty]$ and that $v_0 = v/\|v\|$. Furthermore, let $\Sigma$ be the deterministic matrix introduced in \cite[(11.2)]{Kious_Frib}, that is the covariance matrix of the vector valued quantity $X_{\tau_1}$ under the measure $\mathbb{P}_{0}(\cdot  |\, D = \infty)$. We also set $I_d$ to be $d$-dimensional identity matrix and $P_{v_0}$ to be the projection matrix onto $v_0$. Then we set the notation
\begin{equation*}
    M_d \coloneqq C_{\infty}^{-\gamma/2}(I_d - P_{v_0})\sqrt{\Sigma}.
\end{equation*}
Note that the well-definedness of $\sqrt{\Sigma}$ was argued in \cite{Kious_Frib}.

\begin{proposition}\label{prop:AuxiliaryProcessConvergence}
Let $d \ge 2$, fix any $T > 0$, the followings holds for almost every environment $\omega \in \Omega$. The law of $(Y_n(t))_{0 \le t \le T}$ converges under $P^\omega_0(\cdot)$ in the $U$ topology on $D^d$ towards the distribution of $(v t)_{0 \le t \le T}$. Moreover, the law of $(Z_n(t), S_n(t))_{0 \le t \le T}$ under $P^\omega_0(\cdot)$ converges in $D^d([0, T]) \times D([0, T])$ in the $J_1 \times M_1$ topology as $n \to \infty$ towards the law of $(\sqrt{\Sigma} B_t, C_\infty \mathcal{S}_\gamma(t))_{0 \le t \le T}$, where $B$ is a standard Brownian motion and $\mathcal{S}_\gamma$ is a stable Subordinator of index $\gamma$ independent of $B$. 
\end{proposition}

\begin{proof}
    The statement on $(Y_n(t))_{0 \le t \le T}$ follows directly from \cite[(11.2)]{Kious_Frib}. 
    \paragraph{Marginal convergence of $S_{n}$ and $Z_n$.} We start by considering $S^*_n$. The tightness of the process in $M_1$ follows from finite dimensional distribution (f.d.d.) convergence as both $S_n$ and $\mathcal{S}_\gamma$ are almost surely non-decreasing (see \cite[Theorem~12.12.3]{whitt}). We focus on proving f.d.d.\ convergence. Observe that
    \begin{equation*}
        S_n(t) =  \frac{1}{\mathrm{Inv}(n)} \left(\tau_1 + S_n^*(t) - (\tau_{n+1} - \tau_n) \right).
    \end{equation*}
    We know that $\tau_1$ and $(\tau_{n+1} - \tau_n)$ are almost surely finite quantities, so by Slutsky's theorem we can reduce the problem to showing that $S^*_n$ converges in the f.d.d.\ sense towards $\mathcal{S}_\gamma$.

To do so we define the function
    \begin{equation*}
    F(S_n^{*}) = F(S_n^{*}(t_1),\cdots,S_n^{*}(t_m)) = \exp\left(-\lambda_1 S_n^{*}(t_1)- \lambda_2\left(S_n^{*}(t_2)- S_n^{*}(t_1)\right)\cdots-\lambda_m\left( S^{*}_n(t_m)-S_n^{*}(t_{m-1})\right)\right),
\end{equation*}
where $\lambda_1,\cdots,\lambda_m \in \mathbb{R}^m$ and $0 < t_1\le\cdots\le t_m \le T$. By \cite[Lemma 11.2]{Kious_Frib} we have annealed convergence, so that Laplace functionals converge, in particular
\begin{equation*}
    \mathbb{E}_{0}\left[F(S_n^*)\right] = \mathbf{E}\left[ E_{0}^\omega\left[ F(S_n^*) \right]\right] \underset{n \to +\infty}{\rightarrow} \mathbb{E}_{0}\left[ F(\mathcal{S}_\gamma) \right] = \exp\left( - t_1 \psi(\lambda_1) -  \dots - \psi(\lambda_{m})\left(t_m - t_{m-1}\right)  \right),
\end{equation*}
where $\psi$ is the Laplace exponent of the stable Subordinator. We observe that $F$ is a Lipschitz function of $W^*$ that satisfies the hypotheses of Theorem~\ref{TheoremVarianceClockProcess}, hence from \eqref{EquationBoundVarianceClock} we get that for any $b \in (1,2)$
\begin{equation*}
    \sum_{k} \mathbf{Var}\left(E_{0}^\omega \left[F(S_{b^k}^{*})\right]\right) < \infty,
\end{equation*}
which, by a straighforward application of Chebyshev's inequality and Borel-Cantelli lemma to $|E_{0}^\omega \left[F(S_{b^k}^{*})\right] - \mathbb{E}_{0} \left[F(S_{b^k}^{*})\right]|$ implies that $\mathbf{P}$-a.s.\
\begin{equation*}
    E_{0}^\omega \left[F(S_{b^k}^{*})\right] \to \mathbb{E}_{0} \left[F(S_{b^k}^{*})\right], \quad\quad \text{as }k \to \infty.
\end{equation*}
This convergence can be seen to hold almost surely for any $m \in \mathbb{N}$, any $\lambda_1, \dots, \lambda_m$ and $t_1, \dots, t_m$ and $b$ rationals on an event of full full $\mathbf{P}$ measure. It can be extended to any $\lambda_1, \dots, \lambda_m$ and $t_1, \dots, t_m$ and $b$ by monotonicity (see \cite[pg.~51]{QuenchedBiasedRWRC} for the details).

To extend the convergence to all $n \in \mathbb{N}$, one defines $k_n$ to be the largest integer such that $b^{k_n} \le n$. Using the arguments given in \cite[pg.~52]{QuenchedBiasedRWRC} one can show that
\begin{equation*}
    \limsup_{n \to \infty} \left| E_{0}^\omega \left[F(S_{n}^{*})\right] - E_{0}^\omega \left[F(S_{b^{n}}^{*})\right] \right| \le \mathbb{E}_{0} \left[ \max_{1 \le i \le m} \bigg(|\mathcal{S}_\gamma(b t_i) - \mathcal{S}_\gamma(t_i)|,|(b^{-\frac{1}{\gamma}}-1)\mathcal{S}_\gamma(t_i)| \bigg) \wedge 1 \right]. 
\end{equation*}
Note that as $b \to 0$ the r.h.s.\ goes to $0$ by the dominated convergence theorem and the right-continuity of $\mathcal{S}_\gamma$. This is sufficient to establish f.d.d.\ convergence of $S^*_n$ as we showed that $\mathbf{P}$-a.s.\
\begin{equation*}
    \lim_{n \to \infty} \left| E_{0}^\omega \left[F(S_{n}^{*})\right] - \mathbb{E}_{0}[F(\mathcal{S}_\gamma)] \right| = 0,
\end{equation*}
and the convergence of the Laplace functionals implies weak convergence. 

Before showing joint convergence let us comment on the marginal convergence of $Z_n$. This is nowadays a classical result which is given by \cite[Lemma~4.1]{Szn_serie} and the argument requires only \eqref{EquationBoundVarianceTraj} and annealed convergence \cite[Lemma~11.2]{Kious_Frib}.

\paragraph{Joint convergence.} As we observed before, using Slutsly's theorem it is enough to show convergence of $W^*_n$ towards the limit $W$. By the fact that both marginals $S_n$ and $Z_n$ converge (in the quenched sense) the joint process $(S_n, Z_n)$ is tight under $P^\omega_0$ for $\mathbf{P}$-a.e.\ $\omega \in \Omega$, see \cite[Theorem~11.6.7]{whitt} for a reference.

Let $\lambda_1, \dots, \lambda_m \in \mathbb{R}^{d+1}$ and $0 < t_1 < \dots < t_m \le T$ and define the function
\begin{equation}\label{EqnFourierFunctional}
    G(w) = \exp\Big( i \lambda_1 \cdot w(t_1)+i \lambda_2 \cdot \left( w(t_2) - w(t_1) \right) + \dots + i \lambda_m \cdot \left( w(t_m) - w(t_{m - 1}) \right) \Big).
\end{equation}
By the annealed convergence of \cite[Lemma 11.2]{Kious_Frib} we have that
\begin{equation*}
    \mathbb{E}_{0}\left[ G(W^{*}_n) \right] \underset{n \to +\infty}{\longrightarrow} \mathbb{E}_{0}\left[ G(W) \right].
\end{equation*}
We observe that the real part and the imaginary part of $G$ are two bounded and Lipschitz functions, and thus their positive and negative parts are bounded and Lipschitz functions. All these four functions satisfy the hypotheses of Theorem~\ref{TheoremVarianceClockProcess}, applying it we obtain that for $\mathbf{P}-$almost every environment $\omega \in \Omega$, for any $b \in (1, 2)$
\begin{align*}
    E^{\omega}_{0}[G(W^{*}_{b^{k}})] \underset{k \to +\infty}{\longrightarrow} \mathbb{E}_{0}[G(W)].
\end{align*}
The extension to a generic $n$ is done similarly as we did for $S^*_n$ and details can be found in \cite[pg.52]{QuenchedBiasedRWRC}
\end{proof}

We now give an overview of the proof of the main results, the same arguments are given in \cite{Kious_Frib} and \cite{QuenchedBiasedRWRC} for the proof of their respective main theorems.

\begin{proof}[Proof of Theorem~\ref{MainTheorem}] In all what follows we will fix $\omega \in \Omega$ such that the results of Proposition~\ref{prop:AuxiliaryProcessConvergence} hold. Moreover, all the convergences in distribution that we will show are under the quenched law $P_0^\omega$.

We will drop $[0, T]$ in the notations $D^d([0, T])$ and write $D^d$, all other spaces considered will get the same notation. Let $D_{\uparrow}$ and $D_{\uparrow\uparrow}$ be, respectively, the sets of non-decreasing and strictly increasing c\`{a}dl\`{a}g functions. For $x \in D_{\uparrow}$ we let $x^{-1}$ denote the usual right-continuous inverse. We know that $\mathcal{S}_\gamma$ is a strictly increasing, pure-jump process, thus $\mathcal{S}_\gamma^{-1}$ is non-deacreasing and almost surely continuous. We get that
\begin{equation}\label{eqn:firstEquationLastProof} 
    \left(Y_{n^\gamma/L(n)}\left( S_{n^\gamma/L(n)}^{-1}\left(\frac{nt}{\mathrm{Inv}\left(n^\gamma/L(n)\right)}\right) \right) \right)_{t \in [0, T]} \to \left( v C_\infty^{-\gamma} \mathcal{S}_\gamma^{-1}(t) \right)_{t \in [0, T]},
\end{equation}
in the uniform topology. This holds because of the following reasons:
\begin{enumerate}
    \item The deterministic process $(nt/\mathrm{Inv}(n^\gamma/L(n)))_{t \in [0, T]}$ converges uniformly to $(t)_{t \in [0, T]}$.
    \item The right continuous inverse is continuous from $(D_{u, \uparrow\uparrow}, M_1)$ to $(\mathcal{C}, U)$ by \cite[Corollary~13.6.4]{whitt}.
    \item If $(x_n, y_n) \to (x, y)$ in $D^d \times D_{\uparrow}$ with $(x, y) \in \mathcal{C}^d \times \mathcal{C}_{\uparrow}$, then $x_n \circ y_n \to x \circ y$ in the $U$ topology (see \cite[Theorem 13.2.1]{whitt})
\end{enumerate}
Moreover, applying the same results, it holds that
\begin{equation}\label{eqn:secondEquationLastProof} 
    \left(Z_{n^\gamma/L(n)}\left( S_{n^\gamma/L(n)}^{-1}\left(\frac{nt}{\mathrm{Inv}\left(n^\gamma/L(n)\right)}\right) \right) \right)_{t \in [0, T]} \stackrel{(\mathrm{d})}{\to} \left( C_\infty^{-\gamma/2} B_{\mathcal{S}_\gamma^{-1}(t)} \sqrt{\Sigma} \right)_{t \in [0, T]},
\end{equation}
in the $J_1$ topology. Furthermore, we claim that $\mathbf{P}$-a.s.\ with very high $P^\omega_0$ probability we have that
\begin{equation}\label{eqn:DifferenceLast}
    \max_{t \in [0, T]} \left \| X_{\tau_{\left \lfloor \frac{n^\gamma}{L(n)} S_{n^\gamma/L(n)}^{-1}(nt/\mathrm{Inv}(n^\gamma/L(n)))\right\rfloor}} - X_{\floor{nt}} \right\|_{\infty} \le \delta n^{\gamma/4}
\end{equation}
Indeed, this is true because:
\begin{itemize}
    \item The regeneration time $\tau_{\left \lfloor \frac{n^\gamma}{L(n)} S_{n^\gamma/L(n)}^{-1}(nt/\mathrm{Inv}(n^\gamma/L(n)))\right\rfloor}$ is the first $\tau_k$ such that $\tau_k > n$.
    \item Then the quantity on the l.h.s.\ of \eqref{eqn:DifferenceLast} is bounded above that the largest $\|\chi_{k}\|$ for $k = 1, \dots, n$.
\end{itemize}
Then the claim follows from Lemma~\ref{lemma:GoodEventsBoxes}, Markov's inequality and the Borel-Cantelli lemma.

Equations \eqref{eqn:firstEquationLastProof}, \eqref{eqn:secondEquationLastProof} and \eqref{eqn:DifferenceLast} and Slutsky's theorem imply \eqref{FirstStatementMain}. Finally, they also imply that $\mathbf{P}$-a.s., under $P^\omega_0$
\begin{equation}\label{EqnProofMainTheo}
    \left(\frac{X_{\floor{nt}} - v \frac{n^\gamma}{L(n)} S_{n^\gamma/L(n)}^{-1}\left(\frac{nt}{\mathrm{Inv}\left(n^\gamma/L(n)\right)}\right)}{\sqrt{n^\gamma/L(n)}} \right)_{t \in [0, T]} \stackrel{(\mathrm{d})}{\to} \left( C_\infty^{-\gamma/2} B_{\mathcal{S}_\gamma^{-1}(t)} \sqrt{\Sigma} \right)_{t \in [0, T]}.
\end{equation}
To get \eqref{SecondStatementMain} we apply $(I_d - P_{v_0})$ to \eqref{EqnProofMainTheo}. Linear transofmations preserve $J_1$-convergence if the limit is continuous.
\end{proof}

\appendix

\section{Basic facts}
The results in this section are mainly technical and fairly basic, but since we could not find clear references, we include the proofs for the reader's convenience. 
\begin{lemma}\label{lemma:PolynomialBound}
    Let $(X_i)_{i = 1}^n$ be a family of independent random variables such that $\mathbb{E}[X_1] = 0$ and $\mathbb{E}[X^{2M}_1]<\infty$ for some $M \in \mathbb{N}$. Then there exists some constant $C>0$ (which depends on the distribution of the $X_i$) such that, for each $\lambda>0$,
    \begin{equation*}
        \mathbb{P}\left( \sum_{i = 1}^n X_i > \lambda\right) \le C \frac{n^M}{\lambda^{2M}}.
    \end{equation*}
\end{lemma}
\begin{proof}
We write
\[\mathbb{E}\left[\left(\sum_{i = 1}^n X_i\right)^{2M}\right]=\sum_{i_1,\dots,i_{2M}=1}^{n}\mathbb{E}[X_{i_1}\cdots X_{i_{2M}}].\]
For $(i_1,\dots,i_{2M})\in \{1,\dots,n\}^{2M}$, we denote by $N=N(i_1,\dots,i_{2M})$ the number of distinct values in the $2M$-dimensional vector $(i_1,\dots,i_{2M})$. Then we split
\begin{equation}\label{eqn:PowerSum}
    \sum_{i_1,\dots,i_{2M}=1}^{n}\mathbb{E}[X_{i_1}\cdots X_{i_{2M}}]=\sum_{j=1}^{2M}\sum_{i_1,\dots,i_{2M}=1}^{n}\mathbb{E}[X_{i_1}\cdots X_{i_{2M}}]I_j,
\end{equation}
where $I_j=I_j(i_1,\dots,i_{2M})$ equals $1$ if $N=j$ and is zero otherwise. Note that when $j=1$ we have $n$ ways of choosing the specific value $h\in \{i_1,\dots,i_{2M}\}$ for which $i_1=\dots =i_{2M}=h$ and 
\[\mathbb{E}[X_{i_1}\cdots X_{i_{2M}}]=\mathbb{E}[X^{2M}_1];\]
whence 
\[\sum_{i_1,\dots,i_{2M}=1}^{n}\mathbb{E}[X_{i_1}\cdots X_{i_{2M}}]I_1=n\mathbb{E}[X^{2M}_1]\leq Cn.\]
For $j=2$, we have $n(n-1)\leq n^2$ ways of choosing the two values appearing in the sequence $(i_1,\dots,i_{2M})$, and using the assumption involving the $2M$-th moment of the sum $\sum_{i=1}^nX_i$ we see that
\[\sum_{i_1,\dots,i_{2M}=1}^{n}\mathbb{E}[X_{i_1}\cdots X_{i_{2M}}]I_2\leq Cn^2.\]
In general, if $j\leq M$, then there are at most $n^j\leq n^M$ ways of choosing the $j$ distinct values appearing in the vector, so that in the end
\[\sum_{i_1,\dots,i_{2M}=1}^{n}\mathbb{E}[X_{i_1}\cdots X_{i_{2M}}]I_j\leq Cn^j\leq Cn^M.\]
Now the key point is that, if $N> M$, then there exists \textit{at least one} value in the sequence $(i_1,\dots,i_{2M})$ which appears only once (because otherwise we would have $2N>2M$ distinct values in $(i_1,\dots,i_{2M})$, which is not possible). But then, since the $X_i$ are independent and have mean zero, it follows that (for $j>M$) $\mathbb{E}[X_{i_1}\cdots X_{i_{2M}}]I_j=0$. All in all, we arrive at
\[\sum_{j=1}^{2M}\sum_{i_1,\dots,i_{2M}=1}^{n}\mathbb{E}[X_{i_1}\cdots X_{i_{2M}}]I_j=\sum_{j=1}^{M}\sum_{i_1,\dots,i_{2M}=1}^{n}\mathbb{E}[X_{i_1}\cdots X_{i_{2M}}]I_j\leq C'n^M\]
for some constant $C'$ which depends on $M$ as well as the distribution of $X_1$. The result then follows from Markov's inequality.
\end{proof}

We also need some control on certain moments of truncated heavy-tailed random variables.

\begin{proposition}\label{prop:MomentsSumHT}
Let $X_1, X_2, \dots$ be i.i.d.\ random variables with regularly tail of index $\gamma \in (0, 1)$ (let $L(\cdot)$ be the associated slowly varying function). Then, for all $\varepsilon>0$, for all $p \in \mathbb{N}$ there exists $C>0$ such that
\begin{equation*}
    E\left[\left(\sum_{j =1}^n\frac{X_j}{n^{1/\gamma}} \wedge 1\right)^p\right] \le C n^{\varepsilon p}.
\end{equation*}
\end{proposition}
\begin{proof}
Using our hypothesis
\begin{align*}
    E\left[X_{1} \mathds{1}_{\{X_{1} \le n^{1/\gamma}\}} \right] 
    &= \int_{0}^\infty P\left( X_{1} \mathds{1}_{\{X_{1} \le n^{1/\gamma}\}} > t\right) dt \\
    & \le C \int_{1}^{n^{1/\gamma}} t^{-\gamma + \varepsilon'} dt \le C n^{\frac{1+\varepsilon'}{\gamma} - 1 }.
\end{align*}
This holds for all $\varepsilon'>0$, as for all $\varepsilon'>0$ and $t$ large enough $L(t) \le t^{\varepsilon'}$. Furthermore
\begin{equation*}
    E\left[\left(X_1 \wedge n^{1/\gamma}\right)\mathds{1}_{\{X_{1} > n^{1/\gamma}\}}\right] = n^{\frac{1}{\gamma}} \cdot P(X_1>n^{1/\gamma}) \le C n^{\frac{1+\varepsilon'}{\gamma} - 1 } 
\end{equation*}
Hence, letting $\varepsilon = 2\varepsilon'/\gamma$ we obtain
\begin{equation}\label{eqn:MeanTruncated}
    E\left[ \left(\frac{X_j}{n^{1/\gamma}} \wedge 1\right) \right] \le C  n^{-1/\gamma} n^{\frac{1}{\gamma} - 1 +\varepsilon} = C n^{-1 + \varepsilon}.
\end{equation}
Summing over $j=1, \dots, n$ gets us the result for $p = 1$. For $p>1$, we firstly notice that for all $k \in \mathbb{N}$, $X_j^k$ is regularly varying of index $\gamma/k$, hence one gets that
\begin{equation*}
    E\left[\left(X_1 \wedge n^{1/\gamma}\right)^k\right] = E\left[\left(X_1^k \wedge n^{k/\gamma}\right)\right] \le C n^{k\frac{1+\varepsilon'}{\gamma} - 1 }.
\end{equation*}
Which in turn implies
\begin{equation*}
    E\left[ \left(\frac{X_j}{n^{1/\gamma}} \wedge 1\right)^k \right] \le C  n^{-k/\gamma} n^{\frac{k}{\gamma} - 1 +\varepsilon} = C n^{-1 + \varepsilon}.
\end{equation*}
Now, we can always write the expectation of the power of the sum as in \eqref{eqn:PowerSum}. Using the notation employed there and the last step we get that every term is of the form 
\begin{equation*}
    E[X_{i_1}\cdots X_{i_{p}}]I_j \le n^{-j + j\varepsilon},
\end{equation*}
recall that $j$ here is the number of different indices $i_1, \dots, i_{p}$ chosen. Moreover, there is at most $n^j$ ways to choose such indices hence
\begin{equation*}
    \sum_{i_1,\dots,i_{p}=1}^{n}\mathbb{E}[X_{i_1}\cdots X_{i_{p}}]I_j\leq Cn^j n^{-j + j\varepsilon}\leq Cn^{\varepsilon p}.
\end{equation*}
Finally
\begin{align*}
    \mathbb{E}\left[\left(\sum_{i = 1}^n X_i\right)^{p}\right]=\sum_{j=1}^{p}\sum_{i_1,\dots,i_{p}=1}^{n}\mathbb{E}[X_{i_1}\cdots X_{i_{p}}]I_j \le C p n^{\varepsilon p},
\end{align*}
this concludes the proof.
\end{proof}
\begin{lemma}\label{lemma:MomentsSumHT2}
Let $X_1, X_2, \dots$ be i.i.d.\ random variables with regularly tail of index $\gamma \in (0, 1)$ (let $L(\cdot)$ be the associated slowly varying function). Then, for all $p \in (0, 1)$ and all $\varepsilon>0$, there exists $C>0$ such that
\begin{equation*}
    E\left[\left(\sum_{j =1}^n\left(\frac{X_j}{n^{1/\gamma}} \wedge n^{-p\frac{1}{\gamma}}\right)\right)^2 \right] \le C n^{-p(\frac{1}{\gamma} - 1)+\varepsilon}.
\end{equation*}
\end{lemma}
\begin{proof}
The proof is very similar to the previous proposition, let us give the main steps anyway, we will forego some details. We start by computing
\begin{equation*}
    E\left[\left(X_j \wedge n^{(1-p)\frac{1}{\gamma}}\right) \right] \le Cn^{-(1-p) + (1-p)\frac{1}{\gamma} + \varepsilon}.
\end{equation*}
Similarly,
\begin{equation*}
    E\left[\left(X_j \wedge n^{(1-p)\frac{1}{\gamma}}\right)^2 \right] \le Cn^{-2(1-p) + 2(1-p)\frac{1}{\gamma} + \varepsilon}.
\end{equation*}
Then we get that
\begin{align*}
    E\left[\left(\sum_{j =1}^n\left(\frac{X_j}{n^{1/\gamma}} \wedge n^{-p\frac{1}{\gamma}}\right)\right)^2 \right] &\le E\left[\sum_{j =1}^n\left(\frac{X_j}{n^{1/\gamma}} \wedge n^{-p\frac{1}{\gamma}}\right)^2 \right] + n^2 E\left[\left(\frac{X_1}{n^{1/\gamma}} \wedge n^{-p\frac{1}{\gamma}}\right) \right]^2\\
    &\le Cn^{1 - \frac{2}{\gamma}-2(1-p) + 2(1-p)\frac{1}{\gamma} + 2\varepsilon} + Cn^{2 - \frac{2}{\gamma}-2(1-p) + 2(1-p)\frac{1}{\gamma} + \varepsilon}\\
    &\le Cn^{-p(\frac{1}{\gamma} - 1)+\varepsilon}.
\end{align*}
\end{proof}

\section{Regeneration events}

For completeness we give an overview of the definitions of the random variables $D$ and $\mathcal{D}^\bullet$. See \cite[Section~3]{QuenchedBiasedRWRC} for an extensive discussion on their construction.

\subsection{Enhanced walks}

For any environment $\omega$ we define the modification of the quenched law
\begin{equation}\label{eqn:KtransitionsEnhanced}
   p_K^{\omega}(x,y) \coloneqq \frac{(c_{*}(x,y)\wedge K^{-1})e^{(x+y)\cdot \ell}}{\displaystyle\sum_{ z \sim x}{}{(c_{*}(x,z)\lor K)e^{(x+z)\cdot \ell}}},
\end{equation}
and $0$ if $x$ and $y$ are not adjacent. Observe that, for any $x,y \in \mathbb{Z}^d$, we have that $p^{\omega}_{K}(x,y) \le p^{\omega}(x,y)$ (and the probabilities $\{p^{\omega}_{K}(x,y)\}_{y \sim x}$ do not necessarily sum up to $1$). Moreover, the family $\{p^{\omega}_{K}(x,y)\}_{y \sim x}$ is deterministic if the point $x$ is $K$-open.

For any $(x,z) \in \mathbb{Z}^d \times \{0,1\}$ we define the Markov chain $(\tilde{X}_{n})_{n \geq 0} = (X_{n}, Z_{n})_{n \geq 0}$ to have transitions probabilities $p^{\omega}((x_1,z_1),(x_2,z_2))$ defined by
\begin{align*}
    &1. \,\, \tilde{X}_{0} = (x,z), \quad P^{\omega}_{(x,z)}\text{-a.s},\\
    &2.\,\, p^{\omega}((x,z),(y,1)) = p_K^{\omega}(x,y),\\
    &3.\,\, p^{\omega}((x,z),(y,0)) = p^{\omega}(x,y)  - p_K^{\omega}(x,y),
\end{align*}
and denote its quenched law $P^\omega_{(x,z)}$. This definition ensures that the first coordinate has the same law as our original RWRC $(X_n)_n$ in the environment $\omega$. The annealed law of the enhanced walk and the law of two enhanced walks evolving in the same environment or in independent environments are constructed in the obvious way.

\subsection{Regeneration events}

\paragraph{Single regeneration.} Let $\tilde{X}_n = (X_n, Z_n), n \in \mathbb{N}_0$ be an enhanced walk, we define
\begin{align*}
    & \mathbf{BACK} \coloneqq \inf \left \{ n \ge 1 \colon X_n \cdot \vec{\ell} \le X_0 \cdot \vec{\ell }\right \},\\
    & \mathbf{ORI} \coloneqq \inf \left\{ n \ge 1 \colon X_{n-1} \in \mathcal{V}_{X_0} \text{ and } Z_n = 0\right\},
\end{align*}
and consequently 
\begin{equation}\label{eqn:DSingleRegeneration}
    D \coloneqq \mathbf{BACK} \wedge \mathbf{ORI}.
\end{equation}

\paragraph{Joint regeneration.} Let $\tilde{X}^{(1)}_n, \tilde{X}^{(2)}_n, n \in \mathbb{N}_0$ be two enhanced walks evolving in the same environment. We introduce the random variables, for $i \in \{1,2\}$
\begin{align*}
    &\mathbf{BACK}^{(i)}_{\le R} \coloneqq \inf \left \{ 0  < n < T^{(i)}_{\mathcal{H}^{+}_R} \colon X^{(i)}_n \cdot \vec{\ell} \leq X^{(i)}_0 \cdot \vec{\ell} \right\},\\ %\label{definition_BACK} \\
    & \mathbf{ORI}^{(i)}_{\le R} \coloneqq \inf \left \{ 0  < n < T^{(i)}_{\mathcal{H}^{+}_R} \colon X^{(i)}_{n-1} \in \mathcal{V}_{X^{(1)}_0} \cup \mathcal{V}_{X^{(2)}_0} \text{ and } Z^{(i)}_n = 0, \text{ or } X^{(i)}_{n} \in \mathcal{V}_{X^{(1)}_{0} - e_1} \cup \mathcal{V}_{X^{(2)}_{0} - e_1}  \right\}. %\label{definition_ENV} 
\end{align*}
Moreover,
\begin{equation*}\label{definition_Di}
    \mathcal{D}^{\bullet {(i)}}_{\le R}=\mathbf{BACK}^{(i)}_{\le R} \wedge \mathbf{ORI}^{(i)}_{\le R}\quad \text{and} \quad \mathcal{D}^{\bullet i} \coloneqq \lim_{R \to +\infty} \mathcal{D}^{\bullet i}_{\le R}.
\end{equation*}
Hence, we can set
\begin{equation}\label{definition_D}
    \mathcal{D}^{\bullet} \coloneqq \mathcal{D}^{\bullet (1)} \wedge \mathcal{D}^{\bullet (2)}.
\end{equation}
We also define 
\begin{align}\label{definition_M_other}
    M^{i} \coloneqq \inf\left\{R \in \mathbb{R} \colon \mathcal{D}^{\bullet i}_{\le R} < +\infty \right\} \text{ and } M \coloneqq M^1 \wedge M^2.
\end{align} 
One can check that $\{M = \infty\} = \{\mathcal{D}^\bullet = \infty\}$.

\paragraph{Acknowledgments} The authors are thankful to Tanguy Lions for several useful discussions at various stages of this project. CS thanks Noam Berger for valuable insights on this topic.

\printbibliography
\end{document}